\documentclass{article}

\usepackage{graphicx}

\usepackage{bm}
\usepackage{tikz,mathpazo}
\usepackage{pgf}
\usepackage[top=2.5cm, bottom=2.5cm, left=3cm, right=3cm]{geometry}   
\usepackage{indentfirst}
\usepackage{graphicx}
\usepackage{graphics}
\usepackage[toc,page,title,titletoc,header]{appendix}
\usepackage{bm}
\usepackage{bbm}
\usepackage{listings}  
\usepackage{amsmath}
\usepackage{setspace} 
\usepackage{indentfirst}
\usepackage{caption}
\usepackage{multirow} 
\usepackage{lipsum,multicol}
\usepackage{pdfpages}
\usepackage{float}
\usepackage{amsthm}
\usepackage{pdfpages}
\usepackage{url}
\usepackage{colortbl}
\usepackage{subfigure}
\usepackage{epsfig}
\usepackage{epstopdf}
\usetikzlibrary{shapes.geometric, arrows}
\usepackage{fancyhdr}
\usepackage{abstract}
\usepackage{tikz}
\usetikzlibrary{shapes.geometric, arrows}
\usepackage{amssymb}
\usepackage{latexsym}
\usepackage{verbatim}
\usepackage[numbers]{natbib} 
\usepackage{booktabs}

\newcommand{\inner}[1]{\left\langle #1 \right\rangle}
\newcommand{\norm}[1]{\left\Vert #1\right\Vert}
\newcommand{\bb}[1]{\mathbb{#1}}

\newcommand{\conv}[0]{\mathrm{conv}\,}

\newcommand{\X}{{ \ca{X} }}

\newcommand{\ca}[1]{\mathcal{#1}}

\newcommand{\A}{\ca{A}}

\newcommand{\mk}{{m_{k} }}
\newcommand{\vk}{{v_{k} }}
\newcommand{\mkp}{{m_{k+1} }}
\newcommand{\vkp}{{v_{k+1} }}

\newcommand{\xk}{{x_{k} }}

\newcommand{\xkp}{{x_{k+1} }}

\newcommand{\zk}{{z_{k} }}

\newcommand{\D}{\ca{D}}
\newcommand{\W}{\ca{W}}
\newcommand{\Y}{\ca{Y}}

\newcommand{\caH}{\ca{H}}

\newcommand{\Rn}{\mathbb{R}^n}

\usepackage{hyperref}
\hypersetup{
	colorlinks=true,
	linkcolor=blue,
	filecolor=magenta,      
	urlcolor=blue,
	citecolor=blue,
}

\newtheorem{theo}{Theorem}[section]
\newtheorem{lem}[theo]{Lemma}
\newtheorem{prop}[theo]{Proposition}

\newtheorem{coro}[theo]{Corollary}

\newtheorem{defin}[theo]{Definition}
\newtheorem{rmk}[theo]{Remark}
\newtheorem{assumpt}[theo]{Assumption}

\newtheorem{myclaim}{Claim}

\usepackage[figuresright]{rotating}
\usepackage{pdflscape}

\usepackage{caption}
\usepackage{algorithm}
\usepackage{algpseudocode}
\usepackage{longtable}
\usepackage{appendix}

\numberwithin{equation}{section}

\usepackage{array}

\title{Stochastic Subgradient Methods with Guaranteed Global Stability in Nonsmooth Nonconvex Optimization\thanks{The research of Xiaoyin Hu is supported by the National Natural Science Foundation of China (Grant No. 12301408). The research of Kim-Chuan Toh is supported by the Ministry of Education, Singapore, under its Academic Research Fund Tier 2 grant (MOE-T2EP20224-0017).
}
}
\author{Nachuan Xiao\thanks{School of Data Science, The Chinese University of Hong Kong, Shenzhen, China. \texttt{xncxy@cuhk.edu.cn}},
~Xiaoyin Hu\thanks{School of Mathematical Sciences, Shenzhen University, China. \texttt{hxy@amss.ac.cn}},
~Kim-Chuan Toh\thanks{Department of Mathematics, and Institute of Operations Research and Analytics, National University of Singapore, Singapore 119076. \texttt{mattohkc@nus.edu.sg}}}
\date{}

\newcommand{\keywords}[1]{\par\medskip\noindent\textbf{Keywords: }\begingroup\def\and{; }#1\endgroup}
\newcommand{\subclass}[1]{\par\noindent\textbf{Mathematics Subject Classification: }\begingroup\def\and{; }#1\endgroup}
\newenvironment{acknowledgements}{\section*{Acknowledgements}}{}

\begin{document}

\maketitle

\begin{abstract}
In this paper, we focus on providing convergence guarantees for stochastic subgradient methods in minimizing nonsmooth nonconvex functions. We first investigate the global stability of a general framework for stochastic subgradient methods, where the corresponding differential inclusion admits a coercive Lyapunov function. We prove that, for any sequence of sufficiently small stepsizes and approximation parameters, coupled with sufficiently controlled noises, the iterates are uniformly bounded and asymptotically stabilize around the stable set of its corresponding differential inclusion. Moreover, we develop an improved analysis to apply our proposed framework to establish the global stability of a wide range of stochastic subgradient methods, where the corresponding Lyapunov functions are possibly non-coercive. These theoretical results illustrate the promising potential of our proposed framework for establishing the global stability of various stochastic subgradient methods.

\subclass{65K05 \and 90C30}
\end{abstract}

\section{Introduction}
In this paper, we consider the following unconstrained optimization problem, 
\begin{equation}
	\label{Prob_Ori}
	\tag{UOP}
	\begin{aligned}
		\min_{x \in \Rn}\quad f(x) := \frac{1}{N}\sum_{i = 1}^N f_i(x).
	\end{aligned}
\end{equation}
Here for each $i \in [N]:=\{1,\dots, N\}$, the function $f_i: \Rn \to \bb{R}$ is assumed to be locally Lipschitz continuous, possibly nonconvex and nonsmooth over $\Rn$. 

The optimization problem \eqref{Prob_Ori} encompasses a great number of tasks related to training neural networks, especially when these neural networks employ the rectified linear unit (ReLU) and leaky ReLU \cite{maas2013rectifier} as activation functions. Over the past several decades, a great number of optimization approaches have been developed for solving \eqref{Prob_Ori}. Among these methods, the stochastic gradient descent (SGD) method \cite{robbins1951stochastic} is one of the most popular optimization methods.

Based on the SGD method, various stochastic (sub)gradient methods are proposed by employing momentum terms and performing regularization to the update directions.  For instance, the heavy-ball SGD method \cite{polyak1964some} introduces heavy-ball momentum terms to enhance performance. Moreover, the normalized SGD \cite{cutkosky2020momentum,you2017large,you2019large} is proposed by regularizing the update directions in SGD to the unit sphere, while demonstrating remarkable efficiency in training large language models \cite{you2017large,you2019large}.  In addition, the clipped SGD (ClipSGD) method \cite{zhang2020improved,zhang2020adaptive} is developed by employing the clipping operator to avoid extreme values in the update directions, hence accelerating the training under heavy-tailed noises. Recently, the authors in \cite{castera2021inertial} propose a new variant of the SGD method named as second-order inertial optimization method (INNA), which is derived from the ordinary differential equations proposed in \cite{alvarez2002second,alvarez1998dynamical}. Furthermore, some existing works \cite{bernstein2018signsgd,chen2023symbolic} choose to update along the sign of stochastic gradients. This approach reduces memory consumption in distributed settings \cite{bernstein2018signsgd}, and leads to a simple and effective variant of the SGD method named evolved sign momentum (Lion) \cite{chen2023symbolic}. In addition, the widely used ADAM method \citep{kingma2014adam} employs coordinate-wise regularization. Motivated by the ADAM method, numerous efficient Adam-family methods are developed, such as AdaBelief \citep{zhuang2020adabelief}, AMSGrad \citep{reddi2019convergence}, NAdam \citep{dozat2016incorporating}, etc.

Despite the recent progress in developing highly efficient stochastic subgradient methods, the convergence guarantees for these methods in neural network training tasks remain limited. 
As illustrated in \cite{bolte2021conservative}, nonsmooth activation functions, including ReLU and leaky ReLU, are commonly employed in neural network architectures.  Consequently, the loss functions of these neural networks are usually nonsmooth and lack Clarke regularity (e.g., differentiability, weak convexity, etc.). 
However, most existing works on convergence analysis only focus on cases where $f$ is differentiable or weakly convex \citep{barakat2021convergence,de2018convergence,guo2021novel,polyak1964some,shi2021rmsprop,wang2022provable,you2017large,zaheer2018adaptive,zhang2022adam}, thus precluding various important applications in training nonsmooth neural networks.

To handle the nonsmoothness of the neural networks, some recent works \cite{bianchi2021closed,bolte2021conservative,bolte2022long,castera2021inertial,davis2020stochastic,le2023nonsmooth,ruszczynski2020convergence, 
xiao2023adam} employ the ordinary differential equation (ODE) approach \cite{benaim2005stochastic,borkar2009stochastic,duchi2018stochastic,davis2020stochastic} to establish the convergence properties of subgradient methods in minimizing \textit{path-differentiable} functions. 
As demonstrated in \cite{bolte2021conservative,castera2021inertial,davis2020stochastic}, the class of path-differentiable functions is general enough to include a wide range of real-world applications of these subgradient methods, including those neural networks built from nonsmooth activation functions \cite{bolte2021conservative}. 
Among these subgradient methods, \cite{ruszczynski2020convergence} analyzes the heavy-ball proximal SGD method for minimizing Norkin subdifferentiable functions \cite{norkin1980generalized}. That class of functions belongs to the class of path-differentiable functions \cite[Definition 3]{bolte2021conservative}, according to \cite[Theorem 1]{ruszczynski2020convergence} and \cite[Corollary 2]{bolte2021conservative}. More recently, \cite{le2023nonsmooth} investigates the convergence to essential accumulation points \cite{bolte2022long} for heavy-ball SGD methods in the minimization of path-differentiable functions, based on the closed-measure approach proposed by \cite{bianchi2021closed}. Furthermore, 
\cite{hu2022constraint,hu2023improved} apply these methods to solve manifold optimization problems with the constraint dissolving approach 
\cite{xiao2023dissolving}. In addition, \cite{gurbuzbalaban2022stochastic,ruszczynski2021stochastic} design subgradient methods for solving multi-level composition optimization problems.  However, these aforementioned works are limited to the SGD and heavy-ball SGD methods, making them inapplicable to a wide range of efficient stochastic subgradient methods, such as SignSGD \cite{bernstein2018signsgd}, normalized SGD \cite{you2017large,you2019large}, ADAM \cite{kingma2014adam}, etc. 

When establishing the convergence properties of a given subgradient method, it is important to guarantee its \textit{global stability} (i.e., uniform boundedness of its iterates), which is required in  existing ODE approaches \cite{benaim2005stochastic,borkar2009stochastic,duchi2018stochastic,davis2020stochastic}. Despite this, in a number of existing works \cite{davis2020stochastic,ruszczynski2020convergence,bolte2021conservative,castera2021inertial,hu2022constraint,le2023nonsmooth}, the global stability is presented as a prior assumption without providing explicit guarantees. To establish the global stability for subgradient methods, \cite{bianchi2019constant,bianchi2022convergence} analyze the SGD method with constant stepsizes by linking it to its differential inclusion and considering its iterates as a Markov chain, thus proving convergence to stationary points with high probability. 
Moreover, some recent works \cite{bolte2024inexact,josz2023global} establish the global stability for the SGD method and the heavy-ball SGD method with random reshuffling, by showing that these methods approximate the noiseless SGD method. However, the results in \cite[Theorem 1]{josz2023lyapunov} and \cite{bolte2024inexact} rely on the noiseless evaluation of the subgradients of $f$. In addition, the results in \cite[Proposition 1-2]{josz2023lyapunov} require the objective function $f$ to be Clarke regular or differentiable, which fails for the loss functions of nonsmooth neural networks. Furthermore, given that these pioneering works \cite{bianchi2019constant,bianchi2022convergence,bolte2024inexact,josz2023global,josz2023lyapunov} are primarily designed for analyzing SGD and heavy-ball SGD methods, their results are not directly applicable to the convergence analysis of general stochastic subgradient methods. How to establish the global stability for general subgradient methods remains unexplored. 

Given the increasing number of stochastic subgradient methods, and the limited understanding of their global stability in nonsmooth nonconvex optimization, except for the SGD methods in \cite{bolte2024inexact,josz2023global}, we are motivated to ask the following question: 
\begin{quotation}
	\noindent{\it
		Can we establish the convergence properties, through a unified framework, for various stochastic subgradient methods with guaranteed global stability, especially in training nonsmooth neural networks? 
	}
\end{quotation}

In training neural networks, \textit{random reshuffling} (RR) is a popular choice of sampling techniques in the computation of stochastic subgradients, where the sequence of indices is sampled from $[N]$ without replacement across the iterations within an epoch. Here an {\it epoch} refers to a single pass through the training data, and different epochs may employ different orders \cite{de2020random,pauwels2021incremental}.  Conversely, the with-replacement sampling (WRS) method samples the indices and returns them to the pool before proceeding with another sampling round.

When training nonsmooth neural networks with these aforementioned sampling techniques, in each iteration, the stochastic subgradient is computed by differentiating the loss function $f_i$ for a sampled index $i$ from $[N]$. Automatic differentiation (AD) algorithms serve as essential tools for computing the differential of the loss function in neural networks and have become the default choice in numerous deep learning packages, such as TensorFlow, PyTorch, and JAX. When employed for differentiating a neural network, AD algorithms yield the differential of the loss function by recursively composing the Jacobians of each block of the neural network based on the chain rule. However, when applied to nonsmooth neural networks, the outputs of AD algorithms are usually not in the Clarke subdifferential of the network \cite{bolte2020mathematical}, since the chain rule may fail for general nonsmooth functions \cite{clarke1990optimization}. To address this issue, \cite{bolte2021conservative} introduced the concept of \textit{conservative field}, which is a generalization of the Clarke subdifferential for \textit{path-differentiable} functions  \cite[Definition 3]{bolte2021conservative}. Moreover, as demonstrated in \cite{bolte2021conservative}, the outputs of AD algorithms are always contained within the conservative field for a wide range of loss functions.  Therefore, the concept of the conservative field provides a way to interpret how AD algorithms work in differentiating nonsmooth neural networks. 

Furthermore, \cite{bolte2021conservative} shows that for any path-differentiable function, its Clarke subdifferential is a subset of its conservative field. However, for nonsmooth neural networks, the conservative fields associated with AD algorithms may introduce infinitely many spurious stationary points \cite{bianchi2022convergence,bolte2021nonsmooth}. Consequently, when developing subgradient methods that rely on the notion of conservative field, existing frameworks  \cite{benaim2005stochastic,benaim2006dynamics,bolte2021conservative,davis2020stochastic}  only guarantee the convergence to stationary points in the sense of the conservative field. To address the issue of convergence to possibly non-Clarke critical points, \cite{bianchi2022convergence,bolte2021nonsmooth} prove that, under certain conditions and with randomly chosen initial points and stepsizes, SGD converges to Clarke critical points almost surely. However, their analysis is limited to SGD without any momentum term, and how to extend their results to a wider range of SGD-type methods remains an open question.

\subsection{A general framework for stochastic subgradient methods}

In this paper, we consider a general framework for stochastic subgradient methods,
\begin{equation}
	\label{Eq_stable_iterate}
        \tag{SGM}
	\xkp \in \xk - \eta_k \left( \caH^{\delta_k}(\xk)  + \xi_{k+1}\right).
\end{equation}
Here $\caH: \Rn \rightrightarrows \Rn$ is a set-valued mapping that is graph-closed, convex-valued, and locally bounded. Moreover, the sequence $\{\eta_k\}$ denotes the stepsizes, and $\{\delta_k\}$ refers to the approximation parameters that characterize the inexact evaluations of $\caH$ (see Definition \ref{Defin_approximation_H} for detailed definition). Additionally, the sequence $\{\xi_k\}$ characterizes the evaluation noises of \eqref{Eq_stable_iterate}.

The flexibility in choosing the set-valued mapping $\caH$ in \eqref{Eq_stable_iterate} allows this framework to encompass a wide range of stochastic subgradient methods. Throughout this paper, we focus on two important schemes developed from \eqref{Eq_stable_iterate} by choosing different $\caH$. We first introduce the following scheme for developing SGD-type methods, 
\begin{equation}
	\label{Eq_Framework}
	\tag{GSGD}
	\left\{
	\begin{aligned}
		g_k \in{}& \D_{f_{i_k}}(\xk), \\
		\mkp ={}& \mk + \tau \eta_k( g_k - \mk),\\
		\xkp \in{}& \xk - \eta_k ( \partial \phi(\mkp ) +\rho g_k).
	\end{aligned}
	\right.
\end{equation}
Here $\D_{f_{i_k}}$ is a conservative field for $f_{i_k}$ that describes how $f_{i_k}$ is differentiated. 
Moreover, the sequence of indices $\{i_k\}$ is randomly chosen from $[N]$. In addition, $\{\eta_k\}$ is a prefixed positive sequence of stepsizes,  $\tau > 0$ is the momentum parameter, and $\rho \geq 0$ is a parameter associated with the Nesterov momentum \cite{nesterov1983method}. Furthermore, $\phi: \Rn \to \bb{R}$ is a possibly nonsmooth auxiliary function that characterizes how we regularize the update directions of $\{\xk\}$. Compared with existing schemes studied in previous works 
\cite{de2022born,francca2020conformal,maddison2018hamiltonian,muehlebach2021optimization,o2019hamiltonian,wang2024hamiltonian} that necessitate the differentiability of $\phi$, \eqref{Eq_Framework} offers enhanced generality and encompasses a wider range of SGD-type methods. As shown later in Appendix \ref{Section_developing_GSGD}, with specific choices of the auxiliary function $\phi$, \eqref{Eq_Framework} yields variants of several well-known SGD-type methods, including heavy-ball SGD \cite{polyak1964some}, SignSGD \cite{bernstein2018signsgd}, and normalized SGD \cite{you2017large,you2019large}.

Additionally, based on the frameworks in \cite{barakat2021convergence,ding2023adam,xiao2023adam}, we introduce the following scheme for ADAM-family methods,
\begin{equation}
	\label{Eq_Framework_ADAM}
	\tag{ADM}
	\left\{
	\begin{aligned}
		g_k \in{}& \D_{f_{i_k}}(\xk), \\
		\mkp ={}& \mk + \tau_1 \eta_k( g_k - \mk),\\
            \vkp \in{}& \vk + \tau_2 \eta_k( \ca{V}(\xk, \mkp) - \vk),\\
		\xkp \in{}& \xk - \eta_k  (\ca{P}_{+}(\vkp) + \varepsilon_0)^{-\frac{1}{2}} \odot (\mkp + \rho g_k).
	\end{aligned}
	\right.
\end{equation}
Here, $\odot$ and $(\cdot)^{\gamma}$ refer to taking element-wise multiplication and power,  respectively.  In addition, $\ca{V}: \Rn \times \Rn \rightrightarrows \Rn_+$ is a locally bounded and graph-closed set-valued mapping that determines how the estimator $\vk$ is updated, where $\Rn_+$ refers to the set of $n$-dimensional nonnegative vectors. Moreover, $\ca{P}_{+}$ refers to the orthogonal projection onto $\Rn_+$, and $\tau_1, \tau_2 > 0$ are the momentum parameters that control the stepsizes in updating $\{\mk\}$ and $\{\vk\}$, respectively. As demonstrated in Appendix \ref{Section_developing_ADAM}, with different choices of the set-valued mapping $\ca{V}$,  \eqref{Eq_Framework_ADAM} corresponds to various variants of ADAM-family methods, such as the ADAM \cite{kingma2014adam}, AdaBelief \cite{zhuang2020adabelief}, and NADAM \cite{dozat2016incorporating}.

Although several existing works \cite{ding2023adam,xiao2023adam} are developed to analyze the convergence properties of \eqref{Eq_Framework_ADAM} under general nonconvex nonsmooth cases, their analysis relies on  restrictive assumptions which are commonly employed in existing literature, including the uniform boundedness of the iterates $\{(\xk,\mk, \vk)\}$ and the diminishing stepsizes $\{\eta_k\}$. How to establish the global convergence properties for \eqref{Eq_Framework_ADAM} without these restrictive assumptions remains unexplored.

\subsection{Contributions}

The main contributions of our paper are summarized as follows.

\begin{itemize}
    \item {\bf Guaranteed global stability of the general framework \eqref{Eq_stable_iterate}} \\
    We establish convergence guarantees for \eqref{Eq_stable_iterate} by proving that, with a coercive Lyapunov function, the sequence of iterates generated by \eqref{Eq_stable_iterate} is uniformly bounded under mild conditions. Moreover, for any $\varepsilon > 0$, we can choose sufficiently small (possibly non-diminishing) stepsizes $\{\eta_k\}$ and approximation parameters $\{\delta_k\}$, coupled with sufficiently controlled evaluation noises $\{\xi_k\}$, such that the generated iterates $\{\xk\}$ will eventually stabilize within an $\varepsilon$-neighborhood of the stable set of the Lyapunov function. 

    Specifically, we show that when the evaluation noises $\{\xi_k\}$ are introduced by random reshuffling, the iterates $\{\xk\}$ of our proposed framework \eqref{Eq_stable_iterate} stabilize around the stable set almost surely with non-diminishing but sufficiently small stepsizes. Conversely, when the evaluation noises $\{\xi_k\}$ are introduced by with-replacement sampling and $\{\eta_k\}$ diminishes at the rate of $o(1/\log(k))$, we prove the high-probability convergence properties for the iterates $\{\xk\}$ generated by \eqref{Eq_stable_iterate}.
    
    In comparison to existing results on the stability of subgradient methods \cite{bianchi2022convergence,bolte2024inexact,josz2023global}, our results provide enhanced flexibility in the choice of update schemes, stepsizes, and sampling techniques. Consequently, these results can be applied to explain the convergence of a wide range of stochastic subgradient methods for minimizing nonconvex nonsmooth functions. A brief comparison between our results and existing frameworks on the global stability of stochastic subgradient methods is presented in Table \ref{Table_stability_framework}.

    \begin{table}[t]
    \scriptsize
    \centering
    \begin{tabular}{l|llllll}
    \hline
             & Choice of $\caH$ & Lyap. & Stepsizes & Noises & Conditions & Stability \\ \hline
    \cite{bianchi2019constant} & General $\caH$  & General $\Psi$ &  Constant &  General   &   \cite[(RM), (PH), (FL)]{bianchi2019constant}    &  H.p.   \\
    \cite{josz2023global} & $\partial f$ & $f$ &  Constant   &    RR    &     Coercive $f$     &     A.s.   \\
    \cite{bolte2024inexact} & $\partial f$ & $f$ &  Constant   &    RR    &     Coercive $f$     &     A.s.   \\
    Theorem \ref{Theo_abstract_stability_convergence_sequence} & General $\caH$  & General $\Psi$ &  Non-diminishing &  General   &   Coercive $\Psi$     &  A.s.   \\
    Theorem \ref{Theo_stable_reshuffling_convergence} & General $\caH$  & General $\Psi$ &  Non-diminishing &  RR   &   Coercive $\Psi$   &  A.s.   \\
    Theorem \ref{Theo_stable_stochastic_main} & General $\caH$  & General $\Psi$ &  Diminishing in $o(1/\log(k))$ &  WRS &   Coercive $\Psi$  &  H.p.   \\
    \hline
    \end{tabular}
    \caption{A brief comparison of existing results on various frameworks for global stability. Here ``Lyap.", ``A.s.'', ``H.p'', ``RR'', ``WRS'', and ``MDS'' are the abbreviations of ``Lyapunov function'', ``almost sure convergence'', ``high-probability convergence'', ``random reshuffling'', ``with replacement sampling'', and ``martingale difference sequence'' respectively.}
    \label{Table_stability_framework}
    \end{table}

    \item {\bf Convergence properties of \eqref{Eq_Framework}}\\
    Based on our developed results for global stability of \eqref{Eq_stable_iterate}, we establish the convergence properties for \eqref{Eq_Framework}, which encompasses a wide range of SGD-type methods. We prove that the coercivity of $f$ guarantees the coercivity of the corresponding Lyapunov function. Therefore, under mild conditions, we show that \eqref{Eq_Framework} fits into the framework \eqref{Eq_stable_iterate} and admits a coercive Lyapunov function. Then the global stability of \eqref{Eq_Framework} directly follows from our established results on the global stability of \eqref{Eq_stable_iterate}. In addition, we show that with random initialization,  almost surely, \eqref{Eq_Framework} can avoid the spurious stationary points introduced by the conservative fields $\{\D_{f_i}\}$. Therefore, the sequence $\{\xk\}$ asymptotically converges towards the Clarke critical points of $f$ almost surely, under mild conditions with random initialization.

    \item {\bf Convergence properties of \eqref{Eq_Framework_ADAM}} \\
    We then investigate the convergence properties of \eqref{Eq_Framework_ADAM}, where the coercivity of $f$ cannot guarantee the coercivity of its corresponding Lyapunov function \cite{barakat2021convergence,ding2023adam,xiao2023adam}. To circumvent this difficulty, we introduce an auxiliary update scheme to \eqref{Eq_Framework_ADAM} parameterized by $K > 0$, which corresponds to a coercive Lyapunov function for all positive $K$. Moreover, we prove that the auxiliary scheme coincides with \eqref{Eq_Framework_ADAM} for a specific choice of $K>0$. Therefore, the global stability and avoidance of spurious stationary points of \eqref{Eq_Framework_ADAM} directly follow from those properties of the proposed auxiliary scheme. A brief comparison between our results and the existing results on the stability of stochastic subgradient methods is presented in Table \ref{Table_stability_methods}. Under our  framework \eqref{Eq_Framework_ADAM}, we obtain the global stability of {several representative} popular
    (time independent) stochastic subgradient methods developed for the training of nonsmooth neural networks, including heavy-ball SGD, SignSGD, normalized SGD, ClipSGD, ADAM, and AdaBelief.

    \begin{table}[t]
    \scriptsize
    \centering
    \begin{tabular}{l|lllll}
    \hline
             & Method &  Conditions on $f$ & Stepsizes & Noises & Convergence to $\{x \in \Rn: 0\in \D_f(x)\}$ \\ \hline
    \cite{bianchi2022convergence} & SGD  & Differentiable outside a ball & Constant & WRS  & H.p. to $\varepsilon$-neighborhood  \\
    \cite[Alg. 1]{josz2023global} & GD   & Coercive & Constant & None  & A.s. to $\varepsilon$-neighborhood  \\
    \cite[Alg. 2]{josz2023global} & SGD  & Coercive \& Clarke regular & Constant & RR  & A.s. to $\varepsilon$-neighborhood  \\
    \cite{bolte2024inexact} & GD  & Coercive & Constant & None  & A.s. to $\varepsilon$-neighborhood  \\
    Corollary \ref{Coro_convergence_main}& \eqref{Eq_Framework}  & Coercive & Non-diminishing & RR  & A.s. to $\varepsilon$-neighborhood  \\
    {Corollary \ref{Coro_convergence_main_sto}}& \eqref{Eq_Framework}  & Coercive & Diminishing in $o(1/\log(k))$ & WRS  & H.p. to $\{x \in \Rn: 0\in \D_f(x)\}$  \\
    Corollary \ref{Coro_convergence_ADAM_reshuffling}& \eqref{Eq_Framework_ADAM} & Coercive & Non-diminishing & RR  & A.s. to $\varepsilon$-neighborhood \\
    Corollary \ref{Coro_convergence_ADAM_WRS}& \eqref{Eq_Framework_ADAM}  & Coercive & Diminishing in $o(1/\log(k))$ & WRS  & H.p. to $\{x \in \Rn: 0\in \D_f(x)\}$   \\
    \hline
    \end{tabular}
    \caption{A brief comparison between our results and existing results on the stability of stochastic subgradient methods. Here ``A.s.'' , ``H.p'' ,``WRS'' and ``RR''are the abbreviations of ``almost sure convergence'' , ``high-probability convergence'',  ``with replacement sampling''  and ``random reshuffling'' respectively.}
    \label{Table_stability_methods}
    \end{table}
\end{itemize}

\subsection{Organization}
The rest of this paper is organized as follows. In Section 2, we provide an overview of the notation used throughout the paper and present key concepts related to nonsmooth analysis and differential inclusion. In Section 3, we present a general framework for establishing the global stability of subgradient methods.  Section 4 focuses on the convergence properties of our proposed scheme \eqref{Eq_Framework}, in the aspect of its global stability, global convergence, and avoidance of spurious stationary points introduced by the conservative field. Moreover, Section 5 extends these results to the Adam-family methods in \eqref{Eq_Framework_ADAM}. Finally, we conclude the paper in the last section.

\section{Preliminaries}
\subsection{Notations}
\label{Subsection_notations}
We denote $\inner{\cdot, \cdot}$ as the standard inner product and $\norm{\cdot}$ as the $\ell_2$-norm of a vector or an operator, while $\norm{\cdot}_p$ refers to the $\ell_p$-norm of a vector or an operator.  $\bb{B}_{\delta}(x):= \{ \tilde{x} \in \Rn: \norm{\tilde{x} - x}^2 \leq \delta^2 \}$ refers to the ball
centered at $x$ with radius $\delta$. For a given set $\ca{Y}$, $\mathrm{dist}(x, \ca{Y})$ denotes the distance between $x$ and the set $\ca{Y}$, i.e. $\mathrm{dist}(x, \ca{Y}) := \mathop{\inf}_{y \in \ca{Y}} ~\norm{x-y}$,  $\conv \ca{Y}$ denotes the convex hull of $\ca{Y}$, and $\norm{\Y} := \sup_{y \in \Y} \norm{y}$. For any $x \in \Rn$, we denote $\inner{x, \ca{Y}} := \{\inner{x, y}: y \in \ca{Y}\}$. Moreover, for $\ca{Z} \subset \bb{R}$, we say $\ca{Z} \geq 0$ if $\inf_{z \in \ca{Z}} z\geq 0$. For any given sets $\ca{X}$ and $\ca{Y}$, $\ca{X} + \ca{Y}:= \{x+y: x\in \ca{X}, ~y\in \ca{Y}\}$. In addition,  the (Hausdorff) distance between the sets $\ca{X}$ and $\ca{Y}$ is defined as $$\mathrm{dist}\left( \ca{X}, \ca{Y} \right) := \max\left\{ \sup_{x \in \ca{X}}\mathrm{dist}(x, \ca{Y}),~ \sup_{y \in \ca{Y}}\mathrm{dist}(y, \ca{X}) \right\}.$$  
Furthermore, we define the set-valued mappings $\mathrm{sign}:\Rn \rightrightarrows \Rn$  and  $\mathrm{regu}:\Rn \rightrightarrows \Rn$ as follows: 
\begin{equation*}
	\left(\mathrm{sign}(x)\right)_i = \begin{cases}
		\{-1\} & x_i<0;\\
		[-1, 1] & x_i = 0; \\
		\{1\} & x_i > 0,
	\end{cases}
	\quad \text{and} \quad 
	\mathrm{regu}(m) := 
	\begin{cases}
		\left\{\frac{m}{\norm{m}} \right\} & m \neq 0;\\
		\{z \in \Rn: \norm{z} \leq 1 \} & m = 0.\\
	\end{cases}
\end{equation*}
Then it is easy to verify that $\mathrm{sign}(x)$ and $\mathrm{regu}(x)$ are the Clarke subdifferential of the nonsmooth functions $x \mapsto \norm{x}_1$ and  $x \mapsto \norm{x}$, respectively. For a prefixed constant $C>0$, we define the clipping mapping $\mathrm{clip}_C(x) := \{\min\{\max\{x, -C\}, C\} \}$.

In addition, we denote $\bb{R}_+$ as the set of all nonnegative real numbers, while {we use the convention that} $\bb{N}_+$ denotes the set of all nonnegative integers. For any $N> 0$, $[N]:= \{1,...,N\}$.  Moreover,  $\mu^d$ refers to the Lebesgue measure on $\bb{R}^d$, and when the dimension $d$ is clear from the context, we write the Lebesgue measure as $\mu$ for brevity. Furthermore, given any two sets $A \subseteq B \subseteq \bb{R}^d$, we say  $A$ is zero-measure if $\mu(A) = 0$, and $A$ is a full-measure subset of $B$ if $\mu(B \setminus A) = 0$.

\subsection{Nonsmooth analysis}
\label{Section_CF}

In this part, we briefly introduce the concept of conservative field, which can be applied to characterize how nonsmooth neural networks are differentiated by AD algorithms. We begin this part with the definition of the Clarke subdifferential. 
\begin{defin}
	\label{Defin_Clarke_Subdifferential}
	For any given locally Lipschitz continuous function $f: \Rn \to \bb{R}$, the Clarke subdifferential of $f$, denoted by $\partial f$, is defined as 
	\begin{equation*}
		\partial f(x) := \conv\left( \left\{ \lim_{k\to \infty} \nabla f(\zk): \text{$f$ is differentiable at $\{\zk\}$, and } \lim_{k\to \infty} \zk = x   \right\}  \right).
	\end{equation*}
\end{defin}

\begin{defin}
	A set-valued mapping $\caH: \Rn \rightrightarrows \bb{R}^s$ is a mapping from $\Rn$ to a collection of subsets of $\bb{R}^s$. $\caH$ is said to have closed graph or is graph-closed, if the graph of $\caH$, defined by
	\begin{equation*}
		\mathrm{graph}(\caH) := \left\{ (w,z) \in \Rn \times \bb{R}^s: w \in \bb{R}^n, z \in \caH(w) \right\},
	\end{equation*}
	is a closed subset of $\Rn \times \bb{R}^s$.
\end{defin}

\begin{defin}
	\label{Defin_locally_bounded}
	A set-valued mapping $\caH: \Rn \rightrightarrows \bb{R}^s$ is said to be locally bounded, if for any $x \in \Rn$, there is a neighborhood $V_x$ of $x$ such that $\cup_{y \in V_x}\caH(y)$ is a bounded subset of $\bb{R}^s$. Moreover, $\caH$ is said to be convex-valued, if for any $x \in \Rn$, $\caH(x)$ is a convex subset of $\bb{R}^s$. 
\end{defin}

The following lemma illustrates that the composition of two locally bounded closed-graph set-valued mappings is locally bounded and is graph-closed. The proof of Lemma \ref{Le_closed_graph_2} directly follows from Definition \ref{Defin_locally_bounded}, hence is omitted for simplicity. 
\begin{lem}
	\label{Le_closed_graph_2}
	Suppose both $\D_{g}$ and $\D_h$ are locally bounded set-valued graph-closed mappings. Then the set-valued mapping $\D_g\circ \D_h$ is a locally bounded set-valued graph-closed mapping.  
\end{lem}

In the following definitions, we present the definition of conservative field and its corresponding potential function. 
\begin{defin}
	An absolutely continuous curve is an absolutely continuous mapping $\gamma: \bb{R}_+ \to \Rn $ whose derivative $\gamma'$ exists almost everywhere in $\bb{R}$ and $\gamma(t) - \gamma(0)$ equals to the Lebesgue integral of $\gamma'$ between $0$ and $t$ for all $t \in \bb{R}_+$, i.e.,
	\begin{equation*}
		\gamma(t) = \gamma(0) + \int_{0}^t \gamma'(u) \mathrm{d} u, \qquad \text{for all $t \in \bb{R}_+$}.
	\end{equation*}
\end{defin}

\begin{defin}
	\label{Defin_conservative_field}
	Let $\ca{D}$ be a set-valued mapping from $\bb{R}^{n}$ to subsets of $\bb{R}^{n}$. Then we call $\ca{D}$ as a conservative field if $\D$ satisfies the following properties:
	\begin{enumerate}
		\item $\D$  is locally bounded, graph-closed, and compact-valued. 
		\item For any absolutely continuous curve $\gamma: [0,1] \to \bb{R}^{n} $ satisfying $\gamma(0) = \gamma(1)$, we have
		\begin{equation}
			\label{Eq_Defin_Conservative_mappping}
			\int_{0}^1 \max_{v_t \in \ca{D}(\gamma(t)) } \inner{\gamma'(t), v_t} \mathrm{d}t = 0, 
		\end{equation}
		where the integral is understood to be in the Lebesgue sense. 
	\end{enumerate}
\end{defin}

It is important to note that any conservative field is locally bounded, as illustrated in   \citep[Remark 3]{bolte2021conservative}. Moreover, \cite{bolte2021conservative} shows that any conservative field $\D$ uniquely determines a function through the path integral, as illustrated in the following definition.

\begin{defin}
	\label{Defin_conservative_field_path_int}
	Let $\ca{D}$ be a conservative field in $\bb{R}^{n}$. Then with any given $x_0 \in \bb{R}^{n}$, we can define a function $f: \Rn \to \bb{R}$ through the path integral
	\begin{equation}
		\label{Eq_Defin_CF}
		\begin{aligned}
			f(x) = &{} f(x_0) + \int_{0}^1 \max_{v_t \in \ca{D}(\gamma(t)) } \inner{\gamma'(t), v_t} \mathrm{d}t
			= f(x_0) + \int_{0}^1 \min_{v_t \in \ca{D}(\gamma(t)) } \inner{\gamma'(t), v_t} \mathrm{d}t\text{,}
		\end{aligned}
	\end{equation} 
	for any absolutely continuous curve $\gamma$ that satisfies $\gamma(0) = x_0$ and $\gamma(1) = x$.  Then $f$ is called a potential function for $\ca{D}$, or that $\ca{D}$ is a conservative field for $f$. Such function $f$ is also called path-differentiable. 
\end{defin}

To emphasize that $\D$ can be considered a generalization of the Clarke subdifferential, \cite{bolte2021conservative} presents the following lemma.
\begin{lem}[Theorem 1 and Corollary 1 in \cite{bolte2021conservative}]
	\label{Le_Defin_CF_inclusion}
	Let $f:\Rn \to \bb{R}$ be a path-differentiable function that admits $\D_f$ as its conservative field. Then $\D_f(x) = \{\nabla f(x)\}$ almost everywhere.  Moreover,  $\partial f$ is a conservative field for $f$, and for all $x \in \Rn$, it holds that 
	\begin{equation*}
		\partial f(x) \subseteq \conv(\D_f(x)). 
	\end{equation*} 
\end{lem}

Lemma \ref{Le_Defin_CF_inclusion} illustrates that for any Clarke stationary point $x^*$ of $f$, it holds that $0 \in \conv(\D_f(x^*))$. Therefore, we can employ the concept of conservative field to characterize the optimality for \eqref{Prob_Ori}, as illustrated in the following definition. 
\begin{defin}
	Let $f:\Rn \to \bb{R}$ be a path-differentiable function that admits $\D_f$ as its conservative field. Then we say $x$ is a $\D_f$-stationary point of $f$ if $0 \in \conv(\D_f(x))$. In particular, we say that $x$ is a $\partial f$-stationary point of $f$ if $0 \in \partial f(x)$. 
\end{defin}

The class of path-differentiable functions considered in this paper is general enough to encompass a wide range of objective functions encountered in real-world machine learning problems. As demonstrated in \cite[Section 5.1]{davis2020stochastic}, any Clarke regular function belongs to the category of path-differentiable functions. Consequently, all differentiable functions and weakly convex functions are potential functions. Another important class of functions is the so-called definable functions, which are functions whose graphs are definable in an $o$-minimal structure \cite[Definition 5.10]{davis2020stochastic}. As established in \cite{van1996geometric}, any definable function is also path-differentiable \cite{davis2020stochastic, bolte2021conservative}. As illustrated in the Tarski-Seidenberg theorem \cite{bierstone1988semianalytic}, any semi-algebraic function is definable. Moreover, \cite{wilkie1996model} demonstrates the existence of an $o$-minimal structure that contains both the graph of the exponential function and all semi-algebraic sets. 
Furthermore, it is worth mentioning that definability is preserved under finite summation and composition \cite{davis2020stochastic, bolte2021conservative}. Therefore, for any neural network constructed using definable building blocks, its loss function is definable and, consequently, a path-differentiable function. In addition, it is worth noting that the Clarke subdifferentials of definable functions are definable \cite{bolte2021conservative}. As a result, for any neural network constructed using definable building blocks, the conservative field corresponding to an automatic differentiation (AD) algorithm is also definable. 

The following proposition demonstrates that the definability of both the objective function $f$ and its conservative field $\mathcal{D}_f$ leads to the nonsmooth Morse-Sard property \cite{bolte2007clarke} for \eqref{Prob_Ori}. 
\begin{prop}
	\label{Prop_definable_regularity}
	Let $f$ be a path-differentiable function that admits $\D_f$ as its conservative field. Suppose both $f$ and $\D_f$ are definable over $\Rn$, then the set $\{f(x): 0\in \conv(\D_f(x))\}$ is finite. 
\end{prop}
\begin{proof}
	As illustrated in \cite[Theorem 4]{bolte2021conservative}, when both $f$ and $\D_f$ are definable over $\Rn$,  for any $r \in \bb{N}_+$, it holds that $(f, \D_f)$ has a $\ca{C}^r$ variational stratification \cite[Definition 5]{bolte2021conservative}. Then  $(f, \conv(\D_f))$ also has a $\ca{C}^r$ variational stratification. Therefore, together with \cite[Theorem 5]{bolte2021conservative}, we can conclude that the set $\{f(x): 0 \in \conv(\D_f(x))\}$ is a finite subset of $\bb{R}$.
\end{proof}

\subsection{Differential inclusion}
In this subsection, we introduce some fundamental concepts related to the concept of differential inclusion, which is essential in establishing the convergence properties of subgradient methods, as discussed in \cite{benaim2005stochastic,davis2020stochastic,bolte2021conservative}.   
\begin{defin}
	\label{Defin_DI}
	For any locally bounded set-valued mapping $\caH: \Rn \rightrightarrows \Rn$ that is nonempty convex-valued, and graph-closed, we say that the absolutely continuous mapping $\gamma:\bb{R}_+ \to \Rn$ is a solution (also called trajectory) of the differential inclusion 
	\begin{equation}
		\label{Eq_def_DI}
		\frac{\mathrm{d} x}{\mathrm{d}t} \in -\caH(x),
	\end{equation}
	with initial point $x_0$, if $\gamma(0) = x_0$ and $\dot{\gamma}(t) \in -\caH(\gamma(t))$ holds for almost every $t\geq 0$. 
\end{defin}

We now introduce the concept of Lyapunov function for the differential inclusion \eqref{Eq_def_DI} with a stable set $\A$. 
\begin{defin}
	\label{Defin_Lyapunov_function}
	Let $\A \subset \Rn$ be a closed set. A continuous function $\Psi:\Rn \to \bb{R}$ is referred to as a Lyapunov function for the differential inclusion \eqref{Eq_def_DI}, with the stable set $\A$, if it satisfies the following conditions:
	\begin{itemize}
		\item For any $\gamma$ that is a solution for \eqref{Eq_def_DI} with $\gamma(0) \notin \A$, it holds that $\Psi(\gamma(t)) < \Psi(\gamma(0))$ for any $t>0$.
		\item For any $\gamma$ that is a solution for \eqref{Eq_def_DI} with $\gamma(0) \in \A$, it holds that $\Psi(\gamma(t)) \leq \Psi(\gamma(0))$ for any $t\geq0$.
	\end{itemize}
\end{defin}

\begin{defin}
    \label{Defin_approximation_H}
	For any set-valued mapping $\caH :\Rn \rightrightarrows \Rn$ and any $\delta \geq 0$, we denote the set-valued mapping $\caH^{\delta}:\Rn \rightrightarrows \Rn$ as
	\begin{equation*}
		\caH^{\delta}(x) := \bigcup_{y \in \bb{B}_{\delta}(x)} (\caH(y) + \bb{B}_{\delta}( 0)). 
	\end{equation*}
	Furthermore, we define $\norm{\caH(x)} := \sup\{\norm{d}: d \in \caH(x)\}$. 
\end{defin}

Now consider the sequence $\{\xk\}$ generated by the  following update scheme,  
\begin{equation}
	\label{Eq_def_Iter}
	\xkp = \xk - \eta_kd_k,
\end{equation}
where $\{\eta_k\}$ is a diminishing positive sequence of real numbers. 
We first define the mappings $\lambda: \bb{N}_+ \to \bb{R}_+$ and $\Lambda: \bb{R}_+ \to \bb{N}_+$ as
\begin{equation}
	\label{Eq_def_lambda_Lambda}
	\lambda(0) := 0, \quad \lambda(i) := \sum_{k = 0}^{i-1} \eta_k, \quad \Lambda(t) := \sup  \{k \geq 0: t\geq \lambda(k)\}.
\end{equation}
Specifically, $\Lambda(t) = l$ if $\lambda(l) \leq t < \lambda(l+1)$ for all integer $l \geq 0$. 

Then we define the  (continuous-time) interpolated process of $\{\xk\}$ generated by \eqref{Eq_def_Iter} as follows. 
\begin{defin}
	The  (continuous-time) interpolated process of $\{\xk\}$ generated by \eqref{Eq_def_Iter} with respect to $\{\eta_k\}$ is the mapping $\hat{x}: \bb{R}_+ \to \Rn$ such that 
	\begin{equation*}
		\hat{x}(t) := x_{i} + \frac{t - \lambda(i)}{\eta_i} \left( x_{i+1} - x_{i} \right), \quad t\in[\lambda(i), \lambda(i+1)), \quad i = 0,1, \ldots . 
	\end{equation*}
\end{defin}

\section{Global Stability for General Framework}
In this section, we establish the global stability properties for stochastic subgradient methods. Section 3.1 establishes the global stability for \eqref{Eq_stable_iterate}, which extends the results from \cite{josz2023global,josz2023lyapunov}. Based on the proposed results, we establish the global stability for subgradient methods with reshuffling in Section 3.2. 
Additionally, similar to the settings in  \cite{bianchi2019constant,bianchi2022convergence}, Section 3.3 establishes our results on guaranteeing global stability of the update scheme outlined in \eqref{Eq_stable_random}, where the evaluation noises are introduced by with-replacement sampling. Consequently, our results can be applied to guarantee the global stability for stochastic subgradient methods that employ with-replacement sampling.  

\subsection{Basic assumptions and main results}
\label{Subsection_31}
 In this subsection, we aim to establish the global stability and convergence properties for the stochastic subgradient method in \eqref{Eq_stable_iterate}, with possibly non-diminishing stepsizes and general noises. Consider the corresponding continuous-time differential inclusion for \eqref{Eq_stable_iterate} below, 
\begin{equation}
	\label{Eq_stable_DI}
	\frac{\mathrm{d}\bar{x}}{\mathrm{d}t} \in -\mathrm{conv}\left(\caH(\bar{x}) \right),
\end{equation}
we make the following assumptions on the discrete scheme  \eqref{Eq_stable_iterate} and the differential inclusion \eqref{Eq_stable_DI}. 
\begin{assumpt}
	\label{Assumption_f_definable}
	\begin{enumerate}
		\item The set-valued mapping $\caH: \Rn \rightrightarrows \Rn$ is convex-valued, graph-closed, and locally bounded.
		\item The function $\Psi: \Rn \to \bb{R}$ is a Lyapunov function for the differential inclusion \eqref{Eq_stable_DI} with stable set $\ca{A}$. Moreover, $\Psi$ is coercive and locally Lipschitz continuous over $\Rn$. 
		\item The set $\{\Psi(x): x \in \ca{A}\}$ is a finite subset of $\bb{R}$. 
		\item The sequence $\{\eta_k\}$  is positive and satisfies $\sum_{k = 0}^{\infty} \eta_k = \infty$. Moreover, the sequence $\{\delta_k\}$ is nonnegative. 
	\end{enumerate}
\end{assumpt}

Here we make some comments on Assumption \ref{Assumption_f_definable}. Assumption \ref{Assumption_f_definable}(1)-(2) are mild conditions. For any $r \in \bb{R}$,  let the level set of $\Psi$ be defined as 
\begin{equation}
	\label{Eq_defin_level_set}
	\ca{L}_{r} := \{x \in \Rn: \Psi(x) \leq r\}. 
\end{equation}
Then for any $r \in \bb{R}$,  the coercivity of $\Psi$ in Assumption \ref{Assumption_f_definable}(2) implies the compactness of $\ca{L}_r$.

Moreover, Assumption \ref{Assumption_f_definable}(3) is referred to as the nonsmooth Morse-Sard property \cite[Corollary 9]{bolte2007clarke}, which is a frequently employed condition in various existing works \cite{benaim2005stochastic,borkar2009stochastic,davis2020stochastic,bolte2021conservative,josz2023lyapunov,josz2023convergence,bolte2024inexact}. As demonstrated in \cite{bolte2021conservative}, Assumption \ref{Assumption_f_definable}(3) holds whenever $f$ is a definable function, with the selection of $\caH = \partial f$ and $\Psi = f$. In addition, as shown later in Section 4,  the Lyapunov function corresponding to \eqref{Eq_Framework} satisfies Assumption \ref{Assumption_f_definable}(3) under mild conditions.

We begin our analysis with the definition of two important concepts: \textit{upper-boundedness}, and \textit{asymptotic upper-boundedness}. 
\begin{defin}
	For any non-negative sequence $\{\hat{\eta}_k\}$ and given constants $\alpha_{ub}, \alpha > 0$, we say that $\{\hat{\eta}_k\}$ is upper-bounded by $\alpha_{ub}$ or is $\alpha_{ub}$-upper-bounded, if $\sup_{k\geq 0} \hat{\eta}_k \leq \alpha_{ub}$. Moreover, we say that $\{\hat{\eta}_k\}$ is asymptotically upper-bounded by $(\alpha_{ub}, \alpha)$ or is $(\alpha_{ub}, \alpha)$-asymptotically-upper-bounded, if 
	\begin{equation}
		\sup_{k\geq 0} \hat{\eta}_k \leq \alpha_{ub}, \quad \text{and}\quad \limsup_{k\to \infty} \hat{\eta}_k \leq \alpha. 
	\end{equation}
\end{defin}

Next, we present the following definition of two important concepts for the sequence $\{\xi_k\}$ in \eqref{Eq_stable_iterate}: \textit{controlled} evaluation noises and \textit{asymptotically-controlled} evaluation noises. It is worth mentioning that a similar concept named asymptotic rate of change can be found in \cite[Page 137-138]{kushner2003} and \cite{benaim2005stochastic}.
\begin{defin}
	\label{Defin_controlled_noise}
	Given a sequence of positive numbers $\{\eta_k\}$ and constants $\alpha_{ub}, T > 0$, we say that a sequence of vectors $\{\xi_k\}$ is  $(\alpha_{ub}, T,  \{\eta_k\} )$-controlled, if for any $s \geq 0$, we have
	\begin{equation}
		\sup_{s \leq i\leq \Lambda(\lambda(s) + T) }  \norm{ \sum_{k = s}^i \eta_k \xi_{k+1}} \leq \alpha_{ub}. 
	\end{equation}
	Moreover, we say that a sequence of vectors $\{\xi_k\}$ is  $(\alpha_{ub}, \alpha,  T,  \{\eta_k\} )$-asymptotically-controlled, if it is $(\alpha_{ub}, T,  \{\eta_k\} )$-controlled and 
	\begin{equation}
		\limsup_{s \to \infty} \sup_{s \leq i\leq \Lambda(\lambda(s) + T) }  \norm{ \sum_{k = s}^i \eta_k \xi_{k+1}} \leq \alpha.  
	\end{equation}
\end{defin}

Intuitively, when the sequence of evaluation noise $\{\xi_k\}$ in \eqref{Eq_stable_iterate} is $(\alpha_{ub}, T,  \{\eta_k\} )$-controlled, the cumulative sum $\sum_{k=s}^{j} \eta_k \xi_{k+1}$ can be uniformly upper-bounded for any $s \geq 0$ and any $j \in[s, \Lambda(\lambda(s) + T)]$. Such a property is essential in characterizing the relationship between the discrete iterates $\{\xk\}$ in \eqref{Eq_stable_iterate} and the trajectories of the continuous-time differential inclusion \eqref{Eq_stable_DI}, as demonstrated later in Lemma \ref{Le_stability_close_DI_iter}.

Now we present Proposition \ref{Prop_abstract_descrease_lyapunov}, which estimates the uniform decrease of the Lyapunov function  $\Psi$ along the trajectories of the differential inclusion \eqref{Eq_stable_DI} within a certain time interval $T_{ub}$. Existing works \cite{josz2023global,bolte2024inexact} are restricted to the setting $\caH = \partial f$. As a result, their considered stochastic subgradient descent method corresponds to the following differential inclusion 
\begin{equation}
    \label{Eq_existing_DI_SGD}
    \frac{\mathrm{d}\bar{x}}{\mathrm{d}t} \in -\partial f(\bar{x}),
\end{equation}
which admits $f$ as its Lyapunov function with stable set $\{x \in \Rn: 0 \in \partial f(x)\}$. Then the uniform decrease of $f$ along the trajectories of \eqref{Eq_existing_DI_SGD} directly follows from \cite{davis2020stochastic}. Compared with these existing results in \cite{josz2023global,bolte2024inexact}, our presented Proposition \ref{Prop_abstract_descrease_lyapunov} admits great flexibility in choosing the set-valued mapping $\caH$ by only assuming the existence of a Lyapunov function $\Psi$, as demonstrated in Assumption \ref{Assumption_f_definable}(2). The detailed proof of Proposition \ref{Prop_abstract_descrease_lyapunov} is presented in Appendix \ref{Subsection_proof_00} for simplicity. 
\begin{prop}
	\label{Prop_abstract_descrease_lyapunov}
	Suppose Assumption \ref{Assumption_f_definable} holds, and let $\X_0$ be any compact subset of $\Rn$. Moreover, we assume that for some $T_{ub} > 0$, there exists $r > 0$ such that for any trajectory $\bar{x}$ of \eqref{Eq_stable_DI} with $\bar{x}(0) \in \X_0$, it holds that $\sup_{0\leq t\leq T_{ub}} \norm{\bar{x}(t)} \leq r$. 
	Then for any  $\varepsilon > 0$ and any $T\in (0, T_{ub}]$, there exists $\iota > 0$ such that for any trajectory $\bar{x}$ of the differential inclusion \eqref{Eq_stable_DI}, with 
	\begin{equation}
		\bar{x}(0) \in \X_0, \quad \mathrm{dist}\left( \bar{x}(0), \A  \right) \geq  \varepsilon,
	\end{equation}
	we have
	\begin{equation}
		\Psi(\bar{x}(T)) \leq \Psi(\bar{x}(0)) - \iota. 
	\end{equation}
	With such $\X_0$, $\varepsilon$, and $T$, we call $\iota$ as a descending coefficient for \eqref{Eq_stable_DI} with respect to $(\X_0, \varepsilon, T)$. 
\end{prop}

In the following, with the definitions on the mappings $\lambda$ and $\Lambda$ in \eqref{Eq_def_lambda_Lambda},  we introduce Lemma \ref{Le_stability_close_DI_iter} to characterize the relationship between the discrete iteration sequence in \eqref{Eq_stable_iterate} and the trajectories of the continuous-time differential inclusion \eqref{Eq_stable_DI}. Similar results on the relationship between two differential inclusions could be found in \cite[Corollary 2 in Section 8]{filippov2013differential}. It is worth mentioning that \cite[Lemma 1]{josz2023lyapunov} can be viewed as a special case of Lemma \ref{Le_stability_close_DI_iter} with $\xi_k = 0$ and $\ca{H} = \partial f$. As a result, Lemma \ref{Le_stability_close_DI_iter} extends \cite[Lemma 1]{josz2023lyapunov} to the cases with general $\caH$ and evaluation noises $\{\xi_{k}\}$. For a clear demonstration of our results, we present the detailed proof of Lemma \ref{Le_stability_close_DI_iter} in Appendix \ref{Subsection_proof_01}. 
\begin{lem}
	\label{Le_stability_close_DI_iter}
	Suppose Assumption \ref{Assumption_f_definable} holds, and let $\X_0$ be any compact subset of $\Rn$. Moreover, we assume that there exists $T > 0$,  and $\bar{\alpha}, r > 0$ (depending on $T$), such that for  any $x_0 \in \X_0$, any $\bar{\alpha}$-upper-bounded sequences $\{\eta_k\}$ and $\{\delta_k\}$, and any $(\bar{\alpha}, T, \{\eta_k\})$-controlled $\{\xi_k\}$, the sequence $\{\xk\}$ generated by \eqref{Eq_stable_iterate} satisfies 
	$\max\{\norm{\xk}: k\leq \Lambda(T)\} \leq r$. 
	
	Then for any $\varepsilon > 0$, there exists $\hat{\alpha}\in (0, \bar{\alpha}]$ such that for any $\hat{\alpha}$-upper-bounded sequences $\{\eta_k\}$ and $\{\delta_k\}$, any $(\hat{\alpha}, T, \{\eta_k\})$-controlled $\{\xi_k\}$, and any $\{\xk\}$ generated by \eqref{Eq_stable_iterate} with $x_0 \in \X_0$, there exists a trajectory $\bar{x}$ of the differential inclusion \eqref{Eq_stable_DI} such that 
	\begin{equation}
		\sup_{0\leq t\leq T}\norm{\bar{x}(t) - \hat{x}(t)} \leq \varepsilon.
	\end{equation}
	Here $\hat{x}$ is the interpolated process of $\{\xk\}$ with respect to the stepsizes $\{\eta_k\}$. 
\end{lem}

In the following theorem, we outline the main results on the global stability of \eqref{Eq_stable_iterate}. We first present Proposition \ref{Prop_abstract_stability} to demonstrate the global stability of the discrete update scheme \eqref{Eq_stable_iterate}. Proposition \ref{Prop_abstract_stability} establishes that any sequence $\{\xk\}$ generated by \eqref{Eq_stable_iterate} remains uniformly bounded,  whenever  $\{\eta_k\}$ and $\{\delta_k\}$ are $\alpha$-upper-bounded, and $\{\xi_k\}$ is $(\alpha, T, \{\eta_k\})$-controlled, for any sufficiently small $\alpha > 0$. The detailed proof of Proposition \ref{Prop_abstract_stability} is presented in Appendix \ref{Subsection_proof_1} for clarity. 
\begin{prop}
	\label{Prop_abstract_stability}
	Suppose Assumption \ref{Assumption_f_definable} holds, and let $\X_0$ be any compact subset of $\Rn$.   Then for any given $\tilde{r} > 4\max\{1, \sup_{x \in \X_0 \cup \A}\Psi(x)\}$, 
    there exist $\alpha >0$, $ T>0 $ such that for any $\alpha$-upper-bounded sequences $\{\eta_k\}$ and $\{\delta_k\}$, and any $(\alpha, T, \{\eta_k\})$-controlled sequence $\{\xi_{k}\}$, the sequence $\{\xk\}$ generated by \eqref{Eq_stable_iterate} with $x_0 \in \X_0$ is restricted in $\ca{L}_{\tilde{r}}$.  
\end{prop}

Based on Proposition \ref{Prop_abstract_stability}, we prove in the following theorem that with sufficiently small stepsizes $\{\eta_k\}$ and approximation parameters $\{\delta_k\}$, coupled with sufficiently controlled evaluation noises $\{\xi_k\}$, the sequence $\{\xk\}$ in \eqref{Eq_stable_iterate} eventually stabilizes in a neighborhood of the stable set $\A$. The detailed proof of Theorem \ref{Theo_abstract_stability_convergence_sequence} is presented in Appendix \ref{Subsection_proof_2}, for a clearer presentation of our results. 
\begin{theo}
	\label{Theo_abstract_stability_convergence_sequence}
	Suppose Assumption \ref{Assumption_f_definable} holds, and let $\X_0$ be any compact subset of $\Rn$. Then for any $\varepsilon > 0$ and any given $\tilde{r} \geq 1+ 4 \max\{1, \sup_{x \in \X_0 \cup \A}\Psi(x)\}$, there exist $\alpha > 0$ (depending on $\varepsilon$) and $\alpha_{ub}, T_{ub} > 0$ (independent of $\varepsilon$), such that for any $(\alpha_{ub}, \alpha)$-asymptotically-upper-bounded sequences $\{\eta_k\}$ and $\{\delta_k\}$ and any $(\alpha_{ub}, \alpha, T_{ub}, \{\eta_k\})$-asymptotically-controlled sequence $\{\xi_{k}\}$, the sequence $\{\xk\}$ generated by \eqref{Eq_stable_iterate} with $x_0 \in \X_0$ satisfies 
	\begin{equation}
		\{\xk\} \subseteq \ca{L}_{\tilde{r}}, ~\limsup_{k\to \infty} ~\mathrm{dist}\left(\xk, \A\right) \leq \varepsilon, \quad  \limsup_{k\to \infty} ~ \mathrm{dist}\left(\Psi(\xk),  \{\Psi(x): x \in \A\} \right) \leq \varepsilon.
	\end{equation}
\end{theo}

\begin{rmk}
	It is worth noting that \cite{josz2023global} exhibits similar results for those methods that are approximated by subgradient trajectories \cite[Definition 3]{josz2023global} with a constant stepsize. Based on the proposed framework in \cite{josz2023global}, the authors show the global stability of the SGD method and heavy-ball SGD method with constant stepsizes and random reshuffling. 
	
	Compared to the results presented in \cite{josz2023global,josz2023lyapunov}, Theorem \ref{Theo_abstract_stability_convergence_sequence} extends the results in \cite[Theorem 1, Corollary 1]{josz2023global}. Specifically, our results do not restrict $\caH$ to be chosen as $\partial f$, and allow the stepsizes ${\eta_k}$ to be dynamically changed throughout the iterations. Consequently, Theorem \ref{Theo_abstract_stability_convergence_sequence} facilitates the establishment of global stability across a broader range of subgradient methods and enables a flexible choice of stepsizes in their implementations.
\end{rmk}

\subsection{General subgradient methods using reshuffling}
\label{Subsection_32}
In this subsection, with a given collection $\{\ca{U}_i: 1\leq i\leq N\}$ of locally bounded graph-closed set-valued mappings, we consider the following discrete scheme, 
\begin{equation}
	\label{Eq_stable_reshuffling}
	\xkp \in \xk - \eta_k \ca{U}_{i_k}^{ p(\xk)\eta_k}(\xk). 
\end{equation}
Here we make the following assumptions on \eqref{Eq_stable_reshuffling}.
\begin{assumpt}
	\label{Assumption_stable_reshuffling}
	\begin{enumerate}
		\item The sequence of stepsizes $\{\eta_k\}$ is positive, $\sum_{k=0}^{\infty} \eta_k = \infty$. Moreover, for any $j_1, j_2 \in \bb{N}_+$,  $\eta_{j_1} = \eta_{j_2}$ holds whenever $\lfloor \frac{j_1}{N} \rfloor = \lfloor \frac{j_2}{N} \rfloor$. 
		\item $p: \Rn \to \bb{R}_+$ is a locally bounded function. 
		\item For any $x \in \Rn$, it holds that 
		\begin{equation*}
			\conv\left( \frac{1}{N} \sum_{i = 1}^N \ca{U}_i(x) \right) \subseteq \caH(x). 
		\end{equation*}
		\item The sequence of indices $\{i_k\}$ is chosen from $[N]$ by reshuffling. That is, $\{i_k: jN\leq k < (j+1)N\} = [N]$ holds for any $j \in \bb{N}_+$. 
	\end{enumerate}
\end{assumpt}

Assumption \ref{Assumption_stable_reshuffling}(1) illustrates that the sequence of stepsizes $\{\eta_k\}$ remains unchanged in each epoch, which is a mild condition in practice. Moreover,  Assumption \ref{Assumption_stable_reshuffling}(2) addresses the inexactness in the evaluation of $\ca{U}_{i_k}$, while Assumption \ref{Assumption_stable_reshuffling}(3) shows that the scheme \eqref{Eq_stable_reshuffling} corresponds to the differential inclusion \eqref{Eq_stable_DI}. Furthermore, Assumption \ref{Assumption_stable_reshuffling}(4) illustrates that the sequence of indices $\{i_k\}$ is drawn by reshuffling, in the sense that each index is selected exactly once in each epoch. Such a condition can be regarded as a generalization of random reshuffling.

{Under Assumption \ref{Assumption_stable_reshuffling}, the update scheme \eqref{Eq_stable_reshuffling} can be viewed as a special case of \eqref{Eq_stable_iterate}, as demonstrated in Lemma \ref{Le_stable_reshuffling_xkN_scheme}. Therefore, the global stability of \eqref{Eq_stable_iterate} in Proposition \ref{Prop_abstract_stability} and Theorem \ref{Theo_abstract_stability_convergence_sequence} can be applied to analyze the global stability of \eqref{Eq_stable_reshuffling}.}

Based on Proposition \ref{Prop_abstract_stability} and Theorem \ref{Theo_abstract_stability_convergence_sequence}, we present Proposition \ref{Prop_stable_reshuffling_UB} and  Theorem \ref{Theo_stable_reshuffling_convergence} to illustrate the convergence properties of the sequence $\{\xk\}$ in \eqref{Eq_stable_reshuffling}. For a more concise presentation, the proofs of Proposition \ref{Prop_stable_reshuffling_UB} and Theorem 
\ref{Theo_stable_reshuffling_convergence}
are presented in Appendix \ref{Subsection_appendix_reshuffling}. 

We first present Proposition \ref{Prop_stable_reshuffling_UB} to illustrate that with sufficiently small stepsizes $\{\eta_k\}$, the sequence $\{\xk\}$ generated by \eqref{Eq_stable_reshuffling} is uniformly bounded. 
\begin{prop}
	\label{Prop_stable_reshuffling_UB}
	Suppose Assumption \ref{Assumption_f_definable} and Assumption \ref{Assumption_stable_reshuffling} hold, and let $\X_0$ be any compact subset of $\Rn$. Then for any $\tilde{r} > 8 \max\{1, \sup_{x \in \X_0 {\cup} \A} \Psi(x) \}$, there exists $\alpha_{ub} > 0$ such that for any $\alpha_{ub}$-upper-bounded sequence $\{\eta_k\}$, the sequence $\{\xk\}$ in \eqref{Eq_stable_reshuffling} with $x_0 \in \X_0$ is restricted in $\ca{L}_{\tilde{r}}$. 
\end{prop}

Next, Theorem \ref{Theo_stable_reshuffling_convergence} shows that the sequence $\{\xk\}$ generated by \eqref{Eq_stable_reshuffling} eventually stabilizes in a neighborhood of the stable set $\A$. 
\begin{theo}
	\label{Theo_stable_reshuffling_convergence}
	Suppose Assumption \ref{Assumption_f_definable} and Assumption \ref{Assumption_stable_reshuffling} hold, and let $\X_0$ be any compact subset of $\Rn$. Then for any $\varepsilon > 0$, there exists $\alpha > 0$ (depending on $\varepsilon$) and $\alpha_{ub} > 0$ (independent of $\varepsilon$), such that for any $(\alpha_{ub}, \alpha)$-asymptotically-upper-bounded sequence $\{\eta_k\}$, the sequence $\{\xk\}$ generated by \eqref{Eq_stable_reshuffling} with $x_0 \in \X_0$ satisfies 
	\begin{equation*}
		\mathop{\lim\sup}_{k\to \infty} \mathrm{dist}\left(\xk, \A\right) \leq \varepsilon, \quad  \mathop{\lim\sup}_{k\to \infty} ~ \mathrm{dist}\left(\Psi(\xk),  \{\Psi(x): x \in \A\} \right) \leq \varepsilon.
	\end{equation*}
\end{theo}

\subsection{General subgradient methods using with-replacement sampling}
\label{Subsection_33}
In this subsection, we aim to establish the global stability for a wide range of stochastic subgradient methods, where the evaluation noises are introduced by with-replacement sampling. Let $(\Omega, \ca{F}, \bb{P})$ be the probability space, we consider the following subgradient methods, 
\begin{equation}
	\label{Eq_stable_random}
	\xkp \in \xk - c\eta_k\left( \caH^{\delta_k}(\xk) + \chi(\xk, \zeta_{k+1})  \right). 
\end{equation}
Here $\caH: \Rn \rightrightarrows \Rn$ is a locally bounded and graph-closed set-valued mapping. Moreover,  $\chi: \Rn \times \Omega \to \Rn$ characterizes the evaluation noises in the stochastic subgradient methods, while $\{\zeta_k\}$ is a sequence of elements in $\Omega$. Furthermore, $c > 0$ is a scaling parameter to the stepsizes.

The scheme \eqref{Eq_stable_random} includes \eqref{Eq_stable_reshuffling} using with-replacement sampling. That is, the indices $\{i_k\}$ are independently selected in each iteration, thus allowing any given index to be selected multiple times within a single epoch. 

We begin with the following assumptions on \eqref{Eq_stable_random}. 
\begin{assumpt}
	\label{Assumption_stochastic}
	\begin{enumerate}
		\item For any $x \in \Rn$, it holds that $\bb{E}_{\zeta \sim \bb{P}}[\chi(x, \zeta)] = 0$. 
		Moreover, there exists a locally bounded function $\tilde{q}: \Rn \to \bb{R}_+$, such that for any $x \in \Rn$, $\norm{\chi(x, \zeta)} \leq \tilde{q}(x)$ holds for almost every $\zeta \in \Omega$. 
		\item The sequence $\{\zeta_k\}$ is drawn randomly and independently. 
		\item The sequence of stepsizes $\{\eta_k\}$ is a prefixed sequence of positive numbers, and satisfies
		\begin{equation}
			\sum_{k = 0}^{\infty} \eta_k = \infty, \quad \lim_{k\to\infty} \eta_k \log(k) = 0. 
		\end{equation}
		\item The sequence $\{\delta_k\}$ is nonnegative and satisfies $\lim_{k\to \infty} \delta_k = 0$. 
	\end{enumerate}
\end{assumpt}

In the following theorem, we prove the global stability for the scheme \eqref{Eq_stable_random}, in the sense that with sufficiently small $c$ and $\{\delta_k\}$, all cluster points of the sequence $\{\xk\}$ lie in $\A$ with high probability. Notice that the stepsizes in \eqref{Eq_stable_random} are scaled by the scaling parameter $c > 0$. Consequently, a sufficiently small scaling parameter $c$ results in sufficiently small stepsizes in \eqref{Eq_stable_random}. The detailed proof of Theorem \ref{Theo_stable_stochastic_main} is presented in Section \ref{Subsection_proof_sto} for clarity. 
\begin{theo}
	\label{Theo_stable_stochastic_main}
	Suppose Assumption \ref{Assumption_f_definable} and Assumption \ref{Assumption_stochastic} hold. Let $\X_0$ be any compact subset of $\Rn$ and $x_0 \in \X_0$ in \eqref{Eq_stable_random}.  Then for any $\varepsilon \in (0,1)$, there exists $\alpha > 0$ such that for any $c \in (0, \alpha)$ and any $\{\delta_k\}$ upper-bounded by $\alpha$, we have
	\begin{equation}
		\bb{P}\left(\lim_{k\to \infty} \mathrm{dist} \left(\xk, \A\right)  = 0\right) \geq 1-\varepsilon.
	\end{equation}
\end{theo}
It is worth mentioning that as the set $\{\Psi(x): x \in \A\}$ is finite, we can conclude that for any sequence $\{\xk\}$ satisfying $\lim_{k\to \infty} \mathrm{dist} \left(\xk, \A\right)  = 0$, the sequence $\{\Psi(\xk)\}$ converges.

\begin{rmk}
	To establish the global stability for subgradient methods, existing works \cite{josz2023global,josz2023lyapunov} and the results in Section \ref{Subsection_32} impose strong assumptions on the stochasticity of these subgradient methods, in the sense that the indices are selected by random reshuffling. Besides,  \cite{bianchi2019constant,bianchi2022convergence} establish the global stability for subgradient methods based on the Markov chain technique, where the stochasticity is characterized by a mapping from $\Rn \times \Omega$ to $\Rn$. The analysis in \cite{bianchi2019constant,bianchi2022convergence} can be applied to analyze the global stability of the SGD method that uses the with-replacement sampling techniques. 
	
	Similar to the subgradient methods analyzed in \cite{bianchi2019constant,bianchi2022convergence}, in Assumption \ref{Assumption_stochastic}, the stochasticity of the scheme \eqref{Eq_stable_random} is characterized by the function $\chi: \Rn \times \Omega \to \Rn$. Therefore, our analysis can be applied to establish the global stability for a wider range of SGD-type methods that employ the with-replacement sampling techniques.   
    
   As a result,  our results in Theorem \ref{Theo_stable_stochastic_main} offer a complementary perspective to those presented by \cite{bianchi2019constant,bianchi2022convergence}, thereby enhancing the understanding of convergence behaviors of SGD-type methods.
	
\end{rmk}

\section{Convergence Properties of SGD-type Methods}
\label{Section_GSGD}

In this section, we analyze the convergence properties of \eqref{Eq_Framework} based on the framework developed in Section 3. To establish the convergence properties of \eqref{Eq_Framework}, we impose the following assumptions on the objective function $f$.
\begin{assumpt}
	\label{Assumption_f}
	\begin{enumerate}
		\item For each $i \in [N]$, $f_i$ is a definable locally Lipschitz continuous function that admits a definable conservative field $\D_{f_i}$. 
		\item The objective function $f$ is coercive. 
	\end{enumerate}
\end{assumpt}
As discussed in Section \ref{Section_CF}, the class of definable functions is general enough to include a wide range of commonly employed neural networks, thus making Assumption \ref{Assumption_f}(1) mild in practical scenarios. Moreover,  we define the set-valued mapping $\D_f$ as follows,  
\begin{equation}
	\D_f (x) = \conv\left( \frac{1}{N} \sum_{i = 1}^N \D_{f_i}(x)  \right).
\end{equation}
From the calculus rules of conservative field established in \cite[Corollary 4]{bolte2021conservative}, $\D_f$ is a conservative field for $f$. 
In addition, from Proposition \ref{Prop_definable_regularity} and the definability of $f$, we can conclude that the set $\{f(x): 0 \in \D_f(x)\}$ is a finite subset of $\bb{R}$.

Then we make the following assumptions on \eqref{Eq_Framework}.
\begin{assumpt}
	\label{Assumption_framework}
	\begin{enumerate}
		\item The auxiliary function $\phi:\Rn \to \bb{R}$ is  locally Lipschitz continuous, definable, coercive, and satisfies $\phi(0) = 0$. Moreover,  $\inner{m, \partial \phi(m) } > 0$ holds for any  $m \in \Rn \setminus \{0\}$. 
		\item There exists a compact subset $\X \subseteq \Rn \times \Rn$ such that $(x_0, m_0) \in \X$. 
		\item The parameters $\rho$ and $\tau$ satisfy $\rho  \geq 0$ and $\tau > 0$. 
	\end{enumerate}
\end{assumpt}

Assumption \ref{Assumption_framework}(1) imposes regularity conditions on the auxiliary function $\phi$. From Assumption \ref{Assumption_framework}(1), we can conclude that $\phi$ admits a unique critical point at $0$. In addition, it can be easily verified that this assumption holds whenever $\phi$ is a convex function with a unique minimizer at $0$. Consequently, a wide range of possibly nonsmooth functions can be used for $\phi$, such as $\phi(m) =  \norm{m}_1$ and $\phi(m) = \norm{m}$. Moreover, Assumption \ref{Assumption_framework}(2) requires the initial point $(x_0, m_0)$ of \eqref{Eq_Framework} to lie in a prefixed compact region $\X$, which is mild in practice. Furthermore, Assumption \ref{Assumption_framework}(3) is a standard assumption for SGD-type methods.

\subsection{Basic properties}
In this subsection, we analyze the basic properties of \eqref{Eq_Framework} under Assumption \ref{Assumption_f} and Assumption \ref{Assumption_framework}, based on the results in Theorem \ref{Theo_abstract_stability_convergence_sequence}.  In particular, we show the corresponding differential inclusion of the scheme \eqref{Eq_Framework}, together with its Lyapunov function and stable set.

With Assumption \ref{Assumption_f}, we introduce the function $h$ as follows: 
\begin{equation}
	h(x, m) := f(x)   + \frac{1}{\tau } \phi(m ).
\end{equation}
Then the next lemma shows that $h(x,m)$ is a path-differentiable function and provides the expression for its conservative field $\D_h$.
\begin{lem}
	\label{Le_poten}
	Suppose Assumption \ref{Assumption_f} and Assumption \ref{Assumption_framework} hold. Then $h(x, m)$ is a path-differentiable function that corresponds to the following convex-valued conservative field  
	\begin{equation}
		\D_{h}(x, m) := 
		\left[\begin{matrix}
			\D_f(x) \\
			\frac{1}{\tau } \partial \phi(m)
		\end{matrix}\right].
	\end{equation}
\end{lem}
\begin{proof}
	Notice that $f$ and $\phi$ are potential functions that admit $\D_f$  and $\partial \phi$ as their conservative field, respectively. Then by the chain rule of conservative field \cite{bolte2021conservative}, we can conclude that $h$ is a path-differentiable function that admits $\D_h$ as its conservative field. Moreover, as $\D_f$ and $\partial \phi$ are convex-valued over $\Rn$, it holds that $\D_h$ is convex-valued over $\Rn \times \Rn$.  This completes the proof. 
\end{proof}

The following proposition shows that the function $h$ is coercive and satisfies the nonsmooth Morse-Sard property \cite{bolte2007clarke}. 
\begin{prop}
	\label{Prop_lyapunov_regularity}
	Suppose Assumption \ref{Assumption_f} and Assumption \ref{Assumption_framework} hold. Then $h$ is a definable coercive function, and the set $\{h(x,m): 0 \in \D_f(x), m = 0\}$ is finite.
\end{prop}
\begin{proof}
	The coercivity of $h$ directly follows the coercivity of $f$ and $\phi$. Moreover, from the definability of $f$ and $\D_f$, we can conclude that the set $\{f(x): 0 \in \D_f(x) \}$ is a finite subset of $\bb{R}$. Furthermore, notice that 
	\begin{equation}
		\{h(x,m): 0 \in \D_f(x), m = 0\} = \left\{f(x) +\frac{1}{\tau}\phi(0): 0 \in \D_f(x) \right\} = \{f(x): 0 \in \D_f(x) \},
	\end{equation}
	we can conclude that the set $\{h(x,m): 0 \in \D_f(x), m = 0\}$ is finite. This completes the proof. 
\end{proof}

Next, let the set-valued mapping $\ca{G}:\Rn \times \Rn \rightrightarrows \Rn \times \Rn$ be defined as, 
\begin{equation}
	\label{Eq_Mapping_ST}
	\begin{aligned}
		\ca{G}(x, m) = 
		\left\{ \left[
		\begin{matrix}
			\partial \phi(m) + \rho d\\
			\tau m  -  \tau d\\
		\end{matrix}
		\right]: d \in \D_f(x) \right\}.
	\end{aligned}
\end{equation}
Moreover, for any $i \in [N]$, we define the set-valued mapping $\ca{G}_i:\Rn \times \Rn \rightrightarrows \Rn \times \Rn$ as 
\begin{equation}
	\begin{aligned}
		\ca{G}_i(x, m) = 
		\left\{ \left[
		\begin{matrix}
			\partial \phi(m) + \rho d\\
			\tau m  -  \tau d\\
		\end{matrix}
		\right]: d \in \D_{f_i}(x) \right\}.
	\end{aligned}
\end{equation}
Then it holds directly from Lemma \ref{Le_closed_graph_2} that $\ca{G}$ is convex-valued, locally bounded, and graph-closed. The following proposition illustrates that $h$ can be regarded as the Lyapunov function for the differential inclusion
\begin{equation}
	\label{Eq_DI_ST}
	\left( \frac{\mathrm{d} x}{\mathrm{d} t}, \frac{\mathrm{d} m}{\mathrm{d} t} \right) \in -\ca{G}(x, m). 
\end{equation} 
The detailed proof of Proposition \ref{Prop_Lyapunov} is presented in Appendix \ref{Subsection_appendix_Prop_Lyapunov} for clarity.
\begin{prop}
	\label{Prop_Lyapunov}
	Suppose Assumption \ref{Assumption_f} and Assumption \ref{Assumption_framework} hold, then $h(x, m)$ is a Lyapunov function for the differential inclusion \eqref{Eq_DI_ST}
	with the stable set $ \A_h := \{(x, m)\in \Rn \times \Rn: 0\in \D_f(x), m  = 0\}$.
\end{prop}

\subsection{Global stability}

In this subsection, based on the results in Theorem \ref{Theo_stable_reshuffling_convergence} and Theorem \ref{Theo_stable_stochastic_main}, we investigate the global stability of the scheme \eqref{Eq_Framework}.

We first consider the cases where the indices are drawn by reshuffling. Following Assumption \ref{Assumption_stable_reshuffling}, we stipulate the following assumptions.
\begin{assumpt}
	\label{Assumption_reshuffling}
	\begin{enumerate}
		\item The sequence of stepsizes $\{\eta_k\}$ is positive, $\sum_{k = 0}^{\infty} \eta_k = \infty$. Moreover, for any $k_1, k_2 \in \bb{N}_+$,  $\eta_{k_1} = \eta_{k_2}$ holds  whenever $\lfloor \frac{k_1}{N} \rfloor = \lfloor \frac{k_2}{N} \rfloor$. 
		\item The sequence of indices $\{i_k\}$ is chosen from $[N]$ by reshuffling. That is, $\{i_k: jN\leq k < (j+1)N\} = [N]$ holds for any $j \in \bb{N}_+$. 
	\end{enumerate}
\end{assumpt}

The following corollary shows that with sufficiently small stepsizes $\{\eta_k\}$, we can prove the convergence for the framework \eqref{Eq_Framework} without any assumption on the uniform boundedness of $\{(\xk, \mk)\}$. The detailed proof of Corollary \ref{Coro_convergence_main} is presented in Appendix \ref{Subsection_appendix_Coro_convergence_main}, for a clearer presentation of our results.
\begin{coro}
	\label{Coro_convergence_main}
	Suppose Assumption \ref{Assumption_f}, Assumption \ref{Assumption_framework}, and Assumption \ref{Assumption_reshuffling} hold. Then for any $\varepsilon > 0$, there exists $\alpha > 0$ (depending on $\varepsilon$) and $\alpha_{ub} > 0$ (independent of $\varepsilon$), such that for any $(\alpha_{ub}, \alpha)$-asymptotically-upper-bounded sequence $\{\eta_k\}$, we have 
	\begin{equation}
		\limsup_{k\to \infty} \mathrm{dist}\left( \xk, \{x \in \Rn: 0\in \D_f(x)\}  \right) \leq \varepsilon, \quad \limsup_{k\to \infty} \mathrm{dist}\left( f(\xk), \{f(x): 0\in \D_f(x)\}  \right) \leq \varepsilon. 
	\end{equation}
\end{coro}

As a direct corollary of Corollary \ref{Coro_convergence_main}, we have the following corollary showing the convergence properties of \eqref{Eq_Framework} with diminishing stepsizes. We omit its proof for simplicity. 
\begin{coro}
	\label{Coro_convergence_diminishing}
	Suppose Assumption \ref{Assumption_f}, Assumption \ref{Assumption_framework}, and Assumption \ref{Assumption_reshuffling} hold, and $\lim_{k\to \infty} \eta_k = 0$. Then there exists $\alpha_{ub} > 0$ such that whenever the sequence $\{\eta_k\}$ satisfies $\sup_{k\geq 0} \eta_k \leq \alpha_{ub}$, any cluster point of $\{\xk\}$ lies in $\{x \in \Rn: 0\in \D_f(x)\}$ and the sequence of function values $\{f(\xk)\}$ converges. 
\end{coro}

Based on the results from Theorem \ref{Theo_stable_stochastic_main},  we next establish the global stability of the scheme \eqref{Eq_Framework} when the indices are drawn by with-replacement sampling. Consequently, we make the following assumptions on the stepsizes $\{\eta_k\}$ and the with-replacement sampling strategy.
\begin{assumpt}
	\label{Assumption_fully_random}
	\begin{enumerate}
		\item The sequence of indices $\{i_k\}$ is chosen from $[N]$ uniformly and independently. 
		\item Let $\{\hat{\eta}_k\}$ be a prefixed sequence such that $\sum_{k = 0}^{\infty} \hat{\eta}_k = \infty$, $\lim_{k\to \infty} \hat{\eta}_k \log(k) = 0$, and for any scalar $c > 0$, we set $\eta_k = c\hat{\eta}_k$ for all $k\geq 0$. 
	\end{enumerate}
\end{assumpt}

The following {corollary} demonstrates that, given a sufficiently small scaling parameter $c$, the sequence of iterates ${\xk}$ remains uniformly bounded with high probability, consequently converging to the stationary points of $f$. The detailed proof of Corollary \ref{Coro_convergence_main_sto} directly follows {from} Theorem \ref{Theo_stable_stochastic_main}, hence is omitted for simplicity. 
\begin{coro}
	\label{Coro_convergence_main_sto}
	Suppose Assumption \ref{Assumption_f}, Assumption \ref{Assumption_framework} and Assumption \ref{Assumption_fully_random} hold. Then for any $\varepsilon \in (0,1)$, there exists $\alpha > 0$ such that for any $c \in (0, \alpha)$, with probability at least $1-\varepsilon$, any cluster point of $\{\xk\}$ lies in $\{x \in \Rn: 0 \in \D_f(x)\}$, and the sequence of function values $\{f(\xk)\}$ converges. 
\end{coro}

\subsection{Avoidance of spurious stationary points}

In this subsection, we aim to establish the avoidance of spurious stationary points for the scheme \eqref{Eq_Framework}. Our analysis is developed based on the techniques from \cite{bolte2021nonsmooth,bianchi2022convergence,xiao2023adam}. We begin our analysis with the following definition of the concept of \textit{almost everywhere $\ca{C}^1$} set-valued mappings. 
\begin{defin}
	\label{Defin_AS_C1}
	A measurable mapping $q: \bb{R}^d \to \bb{R}^d$ is almost everywhere $\ca{C}^1$ if for almost every $z \in \bb{R}^d$, $q$ is locally continuously differentiable in 
	a neighborhood of $z$. 
	
	Moreover, a set-valued mapping $\ca{Q} : \bb{R}^d \rightrightarrows \bb{R}^d$ is almost everywhere $\ca{C}^1$ if there exists an almost everywhere $\ca{C}^1$ mapping  $q:\bb{R}^d \to \bb{R}^d$ such that for almost every $z \in \bb{R}^d$, $\ca{Q}(z) = \{q(z)\}$. 
\end{defin}

Based on Definition \ref{Defin_AS_C1}, we present the following proposition to demonstrate that for almost all stepsizes, the iterates $z_{k+1} \in z_k - cQ_k(z)$ can avoid any zero-measure subset of $\Rn$ with almost every initial point. The proof of Proposition \ref{Prop_AS_multiples_step} directly follows from \cite[Proposition 3.10]{xiao2023adam}, hence is omitted for simplicity. 
\begin{prop}
	\label{Prop_AS_multiples_step}
	For any sequence of locally bounded set-valued mappings $\{\ca{Q}_k\}$ that are almost everywhere $\ca{C}^1$, there exists a full-measure subset $\ca{S}\subseteq \bb{R}_+$ such that for any $c \in \ca{S}$ and any zero-measure subset $\ca{Y}$ of $\Rn$, there exists a full-measure subset $\ca{Z}$ of $\Rn$ such that for any sequence $\{z_k\}$ generated by the update scheme
	\begin{equation}
		z_{k+1} \in z_k - cQ_k(z_k),\qquad z_0 \in \ca{Z},
	\end{equation}
	it holds that {$\{z_k\}\cap \ca{Y}=\emptyset$}. 
\end{prop}

Next, we present Corollary \ref{Coro_AS_avoid} to illustrate that the update scheme of \eqref{Eq_Framework} can be characterized by the composition of almost everywhere $\ca{C}^1$ set-valued mappings. Together with Proposition \ref{Prop_AS_multiples_step}, we prove that for almost every initial point and stepsize, the sequence $\{\xk\}$ generated by \eqref{Eq_Framework} can stabilize around the Clarke stationary points, regardless of the chosen conservative field. 
\begin{coro}
	\label{Coro_AS_avoid}
	For the framework \eqref{Eq_Framework}, suppose  
	\begin{enumerate}
		\item Assumption \ref{Assumption_f}, Assumption \ref{Assumption_framework}, and Assumption \ref{Assumption_reshuffling} hold;
		\item let $\{\hat{\eta}_k\}$ be a prefixed sequence such that $\sum_{k = 0}^{\infty} \hat{\eta}_k = \infty$, {$\sup_{k\geq 0}\hat{\eta}_k<\infty$}, and for any scalar $c > 0$, we set $\eta_k = c\hat{\eta}_k$ for all $k\geq 0$. 
	\end{enumerate}
	Then for any $\varepsilon> 0$, there exists $\alpha_{c} >0$, a full-measure subset $\ca{S}$ of $(0, \alpha_c)$ and a full-measure subset $\ca{K}$ of $\X$, such that for any $c \in \ca{S}$ and $(x_0, m_0) \in \ca{K}$, the sequence $\{\xk\}$ generated by \eqref{Eq_Framework} satisfies 
	\begin{equation}
		\limsup_{k\to \infty} \mathrm{dist}\left( \xk, \{x \in \Rn: 0\in \partial f(x)\}  \right) \leq \varepsilon, \quad \limsup_{k\to \infty} \mathrm{dist}\left( f(\xk), \{f(x): 0\in \partial f(x)\}  \right) \leq \varepsilon. 
	\end{equation}
\end{coro}

\begin{proof}
	Let the set-valued mappings $\{\ca{U}_{1,i}\}$ and $\{\ca{U}_{2,i}\}$ be defined as 
	\begin{equation}
		\begin{aligned}
			\ca{U}_{1,i}(x, m) :={}& (0, \tau m - \tau \D_{f_i}(x))\\
			\ca{U}_{2,i}(x, m):={}& (\partial \phi(m ) + \rho\D_{f_i}(x), 0).
		\end{aligned}
	\end{equation}
	Notice that $(f,\D_f)$ has a $\ca{C}^r$ variational stratification, hence the mappings $\{\ca{U}_{1,i}\}$ and $\{\ca{U}_{2,i}\}$ are almost everywhere $\ca{C}^1$ over {$\Rn\times\Rn$}. Furthermore, let 
	\begin{equation}
		\ca{Y}:= \left\{(x, m) \in \Rn \times \Rn: \D_{f_i}(x) \neq \{\nabla f_i(x)\}, \text{ for some } i\in [N]  \right\}.
	\end{equation}
	From \cite{bolte2021conservative}, $\ca{Y}$ is a zero-measure subset of $\Rn \times \Rn$. 
	
	Then  the sequence $\{(\xk, \mk)\}$ generated by \eqref{Eq_Framework} satisfies
	\begin{equation}
		\left\{
		\begin{aligned}
			\left(x_{k+\frac{1}{2}}, m_{k + \frac{1}{2}}  \right) \in{}& (\xk, \mk) - c \hat{\eta}_k\ca{U}_{1, i_k}(\xk, \mk),\\
			\left(\xkp, \mkp  \right) \in{}&\left(x_{k+\frac{1}{2}}, m_{k + \frac{1}{2}}  \right) -  c \hat{\eta}_k\ca{U}_{2, i_k}\left(x_{k+\frac{1}{2}}, m_{k + \frac{1}{2}}  \right).\\
		\end{aligned}
		\right.
	\end{equation}
	By choosing $\{\ca{Q}_k\} = \{\hat{\eta}_k\ca{U}_{1, i_k}, \hat{\eta}_k\ca{U}_{2, i_k}: k\geq 0\}$ in Proposition \ref{Prop_AS_multiples_step}, there exists a full-measure subset $\hat{\ca{S}} \subseteq \bb{R}_+$, such that for any $c \in \hat{\ca{S}}$, there exists a full-measure subset $\hat{\ca{K}}$ of $\Rn \times \Rn$ with the property that, whenever we choose $(x_0, m_0) \in \hat{\ca{K}}$, any $\{(\xk, \mk)\}$ generated by \eqref{Eq_Framework} satisfies $\{(\xk, \mk)\} \cap \ca{Y} = \emptyset$. For {such a sequence} $(\xk, \mk)$, the update scheme in \eqref{Eq_Framework} can be reformulated as 
	\begin{equation}
		\left\{
		\begin{aligned}
			g_k ={}& \nabla {f_{i_k}}(\xk), \\
			\mkp ={}& \mk + \tau \eta_k( g_k - \mk),\\
			\xkp \in{}& \xk - \eta_k ( \partial \phi(\mkp) +\rho g_k).
		\end{aligned}
		\right.
	\end{equation}
	Therefore, by applying Corollary \ref{Coro_convergence_main} with $\D_f = \partial f$, we can conclude that for any $\varepsilon > 0$, we can choose $\alpha >0$ (depending on $\varepsilon$) and $\alpha_{ub}>0$ (independent of $\varepsilon$), such that whenever $c \in \hat{\ca{S}}$ and $\{c \hat{\eta}_k\}$ is $(\alpha_{ub}, \alpha)$-asymptotically-upper-bounded, any sequence $\{\xk\}$ with $(x_0, m_0) \in \hat{\ca{K}} \cap \X$ satisfies 
	\begin{equation}
		\limsup_{k\to \infty} \mathrm{dist}\left( \xk, \{x \in \Rn: 0\in \partial f(x)\}  \right) \leq \varepsilon, \quad \limsup_{k\to \infty} \mathrm{dist}\left( f(\xk), \{f(x): 0\in \partial f(x)\}  \right) \leq \varepsilon. 
	\end{equation}
	Then by choosing {$\alpha_c = \min\left\{\frac{\alpha_{ub} }{\sup_{k\geq 0} \hat{\eta}_k}, \frac{\alpha}{1+\limsup_{k\to \infty} \hat{\eta}_k}\right\}$}, $\ca{S} = \hat{\ca{S}} \cap (0,\alpha_c)$ and $\ca{K} = \hat{\ca{K}} \cap \X$, we complete the proof. 
\end{proof}

Similar to Corollary \ref{Coro_convergence_diminishing}, we present the following theorem to show the convergence of \eqref{Eq_Framework} with random initialization and diminishing stepsizes. {The proof of Corollary \ref{Coro_AS_avoid_diminishing} follows from the same avoidance argument in Corollary \ref{Coro_AS_avoid}, together with Corollary \ref{Coro_convergence_main_sto} and \cite[Theorem 3.11]{xiao2023adam}; hence it is omitted for simplicity.}
\begin{coro}
	\label{Coro_AS_avoid_diminishing}
	Suppose Assumption \ref{Assumption_f}, Assumption \ref{Assumption_framework} and Assumption \ref{Assumption_fully_random} hold. Then for any $\varepsilon \in (0,1)$, there exists $\alpha_{c} >0$, a full-measure subset $\ca{S}$ of $(0, \alpha_c)$ and a full-measure subset $\ca{K}$ of $\X$, such that for any $c \in \ca{S}$ and $(x_0, m_0) \in \ca{K}$, with probability at least $1-\varepsilon$, any cluster point of $\{\xk\}$ lies in $\{x \in \Rn:  0\in \partial f(x)\}$ and the sequence $\{f(\xk)\}$ converges.
\end{coro}

\begin{rmk}
    It is worth mentioning that by selecting an appropriate auxiliary function $\phi$, the framework \eqref{Eq_Framework} can be specialized to several well-known SGD-type methods, including heavy-ball SGD, Lion, SignSGD, and normalized SGD. More importantly, following the results on the global stability of the framework \eqref{Eq_Framework},  our developed SGD-type methods directly enjoy convergence guarantees in the training of nonsmooth neural networks. A detailed discussion of these SGD-type methods is provided in Appendix \ref{Section_developing_GSGD} for interested readers. 

    Furthermore, to demonstrate the numerical potential of our proposed framework, we show that the framework \eqref{Eq_Framework} could yield efficient SGD-type methods. We set $\phi(m) = \norm{m}_1$ in \eqref{Eq_Framework} to formulate the Sign-Map Regularization Momentum SGD (SRSGD), which is a variant of the SignSGD method. Then our established results provide guarantees for its convergence in the training of nonsmooth neural networks. As illustrated in the preliminary numerical experiments in Appendix \ref{Section_numerical}, SRSGD exhibits comparable performance to the state-of-the-art methods, such as Lion \cite{chen2023symbolic}, Adam \cite{kingma2014adam}, and AdamW \cite{loshchilov2017decoupled}.  
\end{rmk}

\section{Convergence Properties of ADAM-family Methods}
\label{Section_ADAM}
In this section, we focus on the convergence properties of the framework for ADAM-family methods with the following assumptions on \eqref{Eq_Framework_ADAM}. 
\begin{assumpt}
    \label{Assumption_framework_ADAM}
    \begin{enumerate}
        \item There exists a compact subset $\X \subseteq \Rn \times \Rn \times \Rn$ such that $(x_0, m_0, v_0) \in \X$. 
        \item The parameters $\varepsilon_0$, $\tau_1$ and $\tau_2$ are positive and satisfy $4\tau_1 \geq \tau_2$. 
        \item The set-valued mapping $\ca{V}: \Rn \times \Rn  \rightrightarrows \Rn_+$ is  nonempty convex-valued, compact-valued, locally bounded, and graph-closed.
    \end{enumerate}
\end{assumpt}

As demonstrated in \cite{ding2023adam}, \eqref{Eq_Framework_ADAM} corresponds to the following differential inclusion,
\begin{equation}
    \label{Eq_DI_ADAM}
    \left( \frac{\mathrm{d}\bar{x}}{\mathrm{d}t}, \frac{\mathrm{d}\bar{m}}{\mathrm{d}t}, \frac{\mathrm{d}\bar{v}}{\mathrm{d}t}  \right) \in -\W(\bar{x}, \bar{m}, \bar{v}), 
\end{equation}
where 
\begin{equation}
	\label{Eq_mapping_G}
	\W(x, m, v) := 
	 \left\{  \left[\begin{matrix}
		(\ca{P}_{+}(v) + \varepsilon_0 )^{-\frac{1}{2}} \odot \left(   m + \rho d \right)\\
		\tau_1 m - \tau_1 d\\
		\tau_2 v - \tau_2  u\\
	\end{matrix}\right]  : d \in \D_f(x), ~u \in \ca{V}(x,m)
        \right\}.
\end{equation}
As shown in \cite[Proposition 4.6]{ding2023adam}, the differential inclusion \eqref{Eq_DI_ADAM} corresponds to the Lyapunov function 
\begin{equation}
    \label{Eq_Lyapunov_func_ADAM}
    h(x,m,v) = f(x) + \frac{1}{2\tau_1} \inner{m, (\ca{P}_+(v)  + \varepsilon_0 )^{-\frac{1}{2}} \odot m},
\end{equation}
with stable set $\A = \{(x, m, v) \in \Rn \times \Rn \times \Rn: 0 \in \D_f(x),~ m = 0\}$. However, the Lyapunov function in \eqref{Eq_Lyapunov_func_ADAM} is {\it not coercive} over $\Rn \times \Rn \times \Rn$, hence the results in Proposition \ref{Prop_abstract_stability} and Theorem \ref{Theo_abstract_stability_convergence_sequence} cannot be directly applied to establish the global stability of \eqref{Eq_Framework_ADAM}.

\subsection{Global stability}

In this subsection, we establish the global stability for \eqref{Eq_Framework_ADAM}. We begin our proof by introducing an auxiliary update scheme to \eqref{Eq_Framework_ADAM} parameterized by $K > 0$, which corresponds to a coercive Lyapunov function for all positive $K$. Then from the great flexibility of the choices of $\caH$ in \eqref{Eq_stable_iterate}, we establish the global stability of our introduced auxiliary scheme. Finally, we establish the global stability for \eqref{Eq_Framework_ADAM} by showing that the auxiliary scheme coincides with \eqref{Eq_Framework_ADAM} for a sufficiently large $K$.

For any $K\geq 0$, let $\mathbbm{I}_{K}: \Rn \rightrightarrows [0, 1]$ denote the {set-valued indicator mapping} of the subset $\{x \in \Rn: \norm{x} \leq K\}$, {in the sense that $\mathbbm{I}_{K}(x) = 1$ for any $\norm{x} < K$, $\mathbbm{I}_{K}(x) = 0$ for any $\norm{x} > K$, and $\mathbbm{I}_{K}(x) = [0,1]$ for any $\norm{x} = K$}. Then we consider the following auxiliary update scheme parameterized by $K>0$, 
\begin{equation}
    \label{Eq_Framework_ADAM_auxiliary}
	\tag{ADM-$K$}
	\left\{
	\begin{aligned}
		g_k \in{}& \D_{f_{i_k}}(\xk), \\
		\mkp ={}& \mk + \tau_1 \eta_k( g_k - \mk),\\
            \vkp \in{}& \vk + \tau_2 \eta_k( \ca{V}(\xk, \mkp) \cdot \mathbbm{I}_{K}(\vk) - \vk),\\
		\xkp \in{}& \xk - \eta_k  (\ca{P}_{+}(\vkp) + \varepsilon_0)^{-\frac{1}{2}} \odot (\mkp + \rho g_k).
	\end{aligned}
	\right.
\end{equation}
Let the set-valued mapping $\W_K: \Rn \times \Rn \times \Rn  \rightrightarrows \Rn \times \Rn \times \Rn$ be defined as 
\begin{equation}
	\label{Eq_mapping_WK}
	\W_K(x, m, v) := 
	 \left\{  \left[\begin{matrix}
		(\ca{P}_{+}(v) + \varepsilon_0 )^{-\frac{1}{2}} \odot \left(   m + \rho d \right)\\
		\tau_1 m - \tau_1 d\\
		\tau_2 v - \tau_2  u\\
	\end{matrix}\right]  : d \in \D_f(x), ~u \in \ca{V}(x, m)\cdot \mathbbm{I}_{K}(v)
        \right\}.
\end{equation}
Moreover, for any $i \in [N]$, we define the set-valued mapping $\W_{K, i}: \Rn \times \Rn \times \Rn \rightrightarrows \Rn \times \Rn \times \Rn$ as 
    \begin{equation}
	\W_{K, i}(x, m, v) := 
	 \left\{  \left[\begin{matrix}
		(\ca{P}_{+}(v) + \varepsilon_0 )^{-\frac{1}{2}} \odot \left(   m + \rho d \right)\\
		\tau_1 m - \tau_1 d\\
		\tau_2 v - \tau_2  u\\
	\end{matrix}\right]  : d \in \D_{f_i}(x), ~u \in \ca{V}(x, m)\cdot \mathbbm{I}_{K}(v)
        \right\}.
\end{equation}
Then it directly follows from Lemma \ref{Le_closed_graph_2} that $\{\W_{K,i}\}$ are convex-valued, locally bounded, and graph-closed.

Now we consider the differential inclusion 
\begin{equation}
    \label{Eq_DI_ADAM_WK}
    \left( \frac{\mathrm{d}\bar{x}}{\mathrm{d}t}, \frac{\mathrm{d}\bar{m}}{\mathrm{d}t}, \frac{\mathrm{d}\bar{v}}{\mathrm{d}t}  \right) \in -\W_K(\bar{x}, \bar{m}, \bar{v}).
\end{equation}
Let $h_K: \Rn \times \Rn \times \Rn \to \bb{R}$ be defined as 
\begin{equation}
    \label{Eq_Lyapunov_func_ADAM_WK}
    h_K(x,m,v):= f(x) + \frac{1}{2\tau_1} \inner{m, (\ca{P}_+(v)  + \varepsilon_0 )^{-\frac{1}{2}} \odot m} + K \max\{0, \norm{v} - K\}.
\end{equation}
The following lemma illustrates the conservative field of $h_K$. 
\begin{lem}
    \label{Le_basic_property_hK}
    Suppose Assumption \ref{Assumption_f} and Assumption \ref{Assumption_framework_ADAM} hold. Then for any $K> 0$, $h_K$ is coercive over $\Rn \times \Rn \times \Rn$ and admits 
    \begin{equation}
        \D_{h_K}(x, m, v) := 
        \left[
        \begin{smallmatrix}
            \D_f(x)\\
            \frac{1}{\tau_1}(\ca{P}_+(v)  + \varepsilon_0 )^{-\frac{1}{2}} \odot m\\
            -\frac{1}{4\tau_1}m^{2} \odot (\ca{P}_+(v)+\varepsilon_0)^{-\frac{3}{2}} \odot \partial \ca{P}_+(v) + K (1-\mathbbm{I}_{K}(v)) \cdot \mathrm{regu}(v)
        \end{smallmatrix}
        \right]
    \end{equation}  
    as its conservative field.
\end{lem}

Next we present the following proposition to illustrate that the differential inclusion \eqref{Eq_DI_ADAM_WK} admits $h_K$ as its Lyapunov function with stable set $\A_K = \{(x, m, v) \in \Rn \times \Rn \times \Rn: 0 \in \D_f(x),~ m = 0, ~ \norm{v} \leq K\}$. For a clearer presentation of our results, we present the proof of Proposition \ref{Prop_Lyapunov_ADAM_WK} in Appendix \ref{Subsection_appendix_Prop_Lyapunov_ADAM_WK}.
\begin{prop}
    \label{Prop_Lyapunov_ADAM_WK}
    Suppose Assumption \ref{Assumption_f} and Assumption \ref{Assumption_framework_ADAM} hold. Then for any $K > 0$, the differential inclusion \eqref{Eq_DI_ADAM_WK} admits $h_K$ as its Lyapunov function with stable set $\A_K$. 
\end{prop}

By directly applying Proposition \ref{Prop_abstract_stability}, we present the following proposition illustrating the global stability of $\{(\xk, \mk, \vk)\}$ generated by the auxiliary scheme \eqref{Eq_Framework_ADAM_auxiliary}. 
\begin{prop}
    \label{Prop_Stability_ADAM_WK}
    Suppose Assumption \ref{Assumption_f}, Assumption \ref{Assumption_framework_ADAM}, and Assumption \ref{Assumption_reshuffling} hold. Then there exists $K_{ub} >0$, such that for any $K\geq K_{ub}$, we can choose $\alpha_{ub} {\in (0, \min\{1/\tau_1, 1/\tau_2\})}$ (depending on $K$) such that for any $\alpha_{ub}$-upper-bounded sequence $\{\eta_k\}$, the sequence $\{(\xk, \mk, \vk)\}$ generated by \eqref{Eq_Framework_ADAM_auxiliary} satisfies
    \begin{equation}
         h_K(\xk, \mk, \vk)\leq \tilde{r}, \quad \sup_{k\geq 0}\norm{\vk} < K_{ub}, 
    \end{equation}
    with $\tilde{r} = 8 \max\{1, \sup_{(x, m, v) \in \X \cup \A } h(x, m, v)\} + 1$. 
\end{prop}
\begin{proof}
    We begin our proof by verifying the validity of Assumption \ref{Assumption_f_definable} and Assumption \ref{Assumption_stable_reshuffling} with $\caH = \W_K$, $\{\ca{U}_i\} = \{\W_{K,i}\}$, $\Psi = h_K$ and $\A = \A_K$. Lemma \ref{Le_basic_property_hK} and Proposition \ref{Prop_Lyapunov_ADAM_WK} illustrate that the differential inclusion \eqref{Eq_DI_ADAM_WK} admits $h_K$ as its coercive Lyapunov function with stable set $\A_K$, which verifies Assumption \ref{Assumption_f_definable}(2). Moreover Assumption \ref{Assumption_f_definable}(3) directly follows from the definability of $f$. 

    Furthermore, from the local boundedness of $\{\W_{K,i}: i\in [N]\}$, the update scheme \eqref{Eq_Framework_ADAM_auxiliary} fits into the following scheme:
    \begin{equation}
        (\xkp, \mkp, \vkp) \in (\xk, \mk, \vk) - \eta_k\W_{K, i_k}^{\eta_k \tilde{p}(\xk, \mk, \vk)}(\xk, \mk, \vk), 
    \end{equation}
    where 
    \begin{equation}
        \label{Eq_Prop_Stability_ADAM_WK_1}
        \tilde{p}(x, m, v) = \tau_1 \norm{\D_f(x) - m} + \tau_2 \norm{\ca{V}(x, m) - v}.
    \end{equation}
    It is easy to prove that $\tilde{p}$ is a locally bounded function over $\Rn \times \Rn \times \Rn$, hence verifying Assumption \ref{Assumption_stable_reshuffling}(2). In addition, Assumption \ref{Assumption_stable_reshuffling}(1)(3)(4) directly follows from Assumption \ref{Assumption_reshuffling} and the expression of $\W_K$ and $\{\W_{K, i}\}$. 

    Therefore, letting $K_{0} = \sup_{(x, m, v) \in \X} \norm{v}$, then from the expression of $h_K$, it holds for any $K\geq K_{0}$ that $\sup_{(x, m, v) \in \X} h_K(x, m, v) = \sup_{(x, m, v) \in \X} h(x, m, v)$. Then it follows from Proposition \ref{Prop_stable_reshuffling_UB} that there exists $\alpha_{ub} > 0$, such that for any $K > K_{0}$ and any $\alpha_{ub}$-upper-bounded sequence $\{\eta_k\}$ the sequence $\{(\xk, \mk, \vk)\}$ generated by \eqref{Eq_Framework_ADAM_auxiliary} satisfies  
    \begin{equation}
        \label{Eq_Prop_Stability_ADAM_WK_0}
        h(\xk, \mk, \vk) \leq  h_K(\xk, \mk, \vk)\leq \tilde{r}, 
    \end{equation}
    with $\tilde{r} = 8 \max\{1, \sup_{(x, m, v) \in \X \cup \A  } h(x, m, v)\} + 1$. 

    Let {$M_{ub,1}= \sup_{\{x: f(x) \leq \tilde{r}\}}  \norm{x}$}. From the coercivity of $f$ in Assumption \ref{Assumption_f}, we can conclude that $M_{ub,1}$ is finite. Together with \eqref{Eq_Prop_Stability_ADAM_WK_0}, it holds for any $K > K_{0}$ that $\sup_{k\geq 0}\norm{\xk} \leq M_{ub,1}$. Moreover, based on the definition of $M_{ub,1}$, we denote 
    \begin{equation}
        \label{Eq_Prop_Stability_ADAM_WK_2}
        M_{ub,2}:= \max\left\{ \sup_{(x, m, v) \in \X} \norm{m},  {\sup_{ \norm{x} \leq 
        M_{ub,1},~1\leq i\leq N }  \norm{\D_{f_i}(x)}}\right\}. 
    \end{equation}
    Then we choose $K_{ub}$ such that $K_{ub} > \max\{K_{0}, 2\sup_{\norm{x}\leq M_{ub,1}, \norm{m} \leq M_{ub,2}} \norm{\ca{V}(x, m)}\}$. {By decreasing $\alpha_{ub}$ if necessary, we assume that $\tau_1\eta_k\leq 1$ and $\tau_2\eta_k\leq 1$ for all $k\geq0$. We prove by induction that $\norm{\mk}\leq M_{ub,2}$ and $\norm{\vk}<K_{ub}$. The assertion is true at $k=0$ by the definitions of $M_{ub,2}$ and $K_0$. If it holds at iteration $k$, then $\norm{g_k}\leq M_{ub,2}$ and}
    \begin{equation*}
        {\norm{\mkp}\leq (1-\tau_1\eta_k)\norm{\mk}+\tau_1\eta_k\norm{g_k}\leq M_{ub,2}.}
    \end{equation*}
    {Moreover, since $\norm{\vk}<K_{ub}\leq K$, then $\mathbbm{I}_{K}(\vk)$ equals $1$, and every $u_k\in\ca{V}(\xk,\mkp)$ satisfies $\norm{u_k}<K_{ub}/2$. Hence}
    \begin{equation*}
        {\norm{\vkp}\leq (1-\tau_2\eta_k)\norm{\vk}+\tau_2\eta_k\norm{u_k}<K_{ub}.}
    \end{equation*}
    {Therefore, $\sup_{k\geq0}\norm{\vk}<K_{ub}$ for all $K\geq K_{ub}$.} This completes the proof. 
\end{proof}

Furthermore, based on similar proof techniques as in Corollary \ref{Coro_convergence_main}, we have the following theorem illustrating the convergence properties of \eqref{Eq_Framework_ADAM}. The detailed proof of Corollary \ref{Coro_convergence_ADAM_reshuffling} is presented in Appendix \ref{Section_appendix_Coro_convergence_ADAM_reshuffling}.
\begin{coro}
    \label{Coro_convergence_ADAM_reshuffling}
    Suppose Assumption \ref{Assumption_f}, Assumption \ref{Assumption_reshuffling}, and Assumption \ref{Assumption_framework_ADAM} hold. Then for any $\varepsilon > 0$, there exists $\alpha > 0$ (depending on $\varepsilon$) and $\alpha_{ub} > 0$ (independent of $\varepsilon$), such that for any $(\alpha_{ub}, \alpha)$-asymptotically-upper-bounded sequence $\{\eta_k\}$, the sequence $\{\xk\}$ generated by \eqref{Eq_Framework_ADAM} satisfies 
	\begin{equation}
		\limsup_{k\to \infty} \mathrm{dist}\left( \xk, \{x \in \Rn: 0\in \D_f(x)\}  \right) \leq \varepsilon, \quad \limsup_{k\to \infty} \mathrm{dist}\left( f(\xk), \{f(x): 0\in \D_f(x)\}  \right) \leq \varepsilon. 
	\end{equation}
\end{coro}

The following theorem demonstrates that, given a sufficiently small scaling parameter $c$, the sequence of iterates $\{\xk\}$ remains uniformly bounded with high probability, consequently converging to the stationary points of $f$. By the same proof technique as 
in Corollary \ref{Coro_convergence_ADAM_reshuffling}, the proof of Corollary \ref{Coro_convergence_ADAM_WRS} directly follows from Theorem \ref{Theo_stable_stochastic_main}, hence is omitted for simplicity. 
\begin{coro}
    \label{Coro_convergence_ADAM_WRS}
    Suppose Assumption \ref{Assumption_f}, Assumption \ref{Assumption_fully_random}, and Assumption \ref{Assumption_framework_ADAM} hold. Then for any $\varepsilon \in (0, 1)$, there exists $\alpha > 0$ such that for any $c \in (0, \alpha)$, with probability at least $1-\varepsilon$, any cluster point of $\{\xk\}$ in \eqref{Eq_Framework_ADAM} lies in $\{x \in \Rn: 0 \in \D_f(x)\}$, and the sequence of function values $\{f(\xk)\}$ converges. 
\end{coro}

\subsection{Avoidance of spurious stationary points}
In this subsection, we establish the avoidance of spurious stationary points for the scheme \eqref{Eq_Framework_ADAM} in the following two theorems. 
\begin{coro}
\label{Coro_AS_avoid_ADAM}
For the scheme \eqref{Eq_Framework_ADAM}, suppose  
\begin{enumerate}
    \item Assumption \ref{Assumption_f}, Assumption \ref{Assumption_reshuffling}, and Assumption \ref{Assumption_framework_ADAM} hold;
    \item let $\{\hat{\eta}_k\}$ be a prefixed sequence such that $\sum_{k = 0}^{\infty} \hat{\eta}_k = \infty$, {$\sup_{k\geq0}\hat{\eta}_k<\infty$}, and for any scalar $c > 0$, we set $\eta_k = c\hat{\eta}_k$ for all $k\geq 0$;
    \item $\ca{V}$ is almost everywhere $\ca{C}^1$ over $\Rn \times \Rn$. That is, for almost every $(x, m) \in \Rn \times \Rn$, $\ca{V}$ is a singleton and continuously differentiable within a neighborhood of $(x, m)$. 
\end{enumerate}
{Then for any $\varepsilon>0$,} there exists $\alpha_{c} >0$, a full-measure subset $\ca{S}$ of $(0, \alpha_c)$ and a full-measure subset $\ca{K}$ of $\X$, such that for any $c \in \ca{S}$ and $(x_0, m_0, v_0) \in \ca{K}$, the sequence $\{\xk\}$ generated by \eqref{Eq_Framework_ADAM} satisfies 
\begin{equation}
    \limsup_{k\to \infty} \mathrm{dist}\left( \xk, \{x \in \Rn: 0\in \partial f(x)\}  \right) \leq \varepsilon, \quad \limsup_{k\to \infty} \mathrm{dist}\left( f(\xk), \{f(x): 0\in \partial f(x)\}  \right) \leq \varepsilon. 
\end{equation}
\end{coro}
\begin{proof}
    Let the set-valued mappings $\{\ca{U}_{1,i}\}$, $\{\ca{U}_{2,i}\}$ and $\{\ca{U}_{3,i}\}$ be defined as, 
    \begin{equation*}
        \small
        \ca{U}_{1,i}(x, m, v) = 
        \left[\begin{matrix}
            0\\
            \tau_1 m - \tau_1  \D_{f_i}(x)\\
            0\\
        \end{matrix}\right], 
        \quad 
        \ca{U}_{2,i}(x, m, v) = 
        \left[\begin{matrix}
            0\\
            0\\
             \tau_2 (v - \ca{V}(x, m) ) 
            \\
        \end{matrix}\right],
    \end{equation*}
    and 
    \begin{equation*}
        \ca{U}_{3,i}(x, m, v) = 
        \left[\begin{matrix}
            (\ca{P}_{+}(v) + \varepsilon_0)^{-\frac{1}{2}} \odot (m + \rho \D_{f_i}(x))\\
            0\\
            0\\
        \end{matrix}\right].
    \end{equation*}
    Notice that $(f,\D_f)$ has a $\ca{C}^r$ variational stratification, hence the mappings $\{\ca{U}_{1,i}\}$, $\{\ca{U}_{2,i}\}$ and $\{\ca{U}_{3,i}\}$ are almost everywhere $\ca{C}^1$ over {$\Rn\times\Rn\times\Rn$}. Furthermore, let 
	\begin{equation}
		\ca{Y}:= \left\{(x, m, v) \in \Rn \times \Rn \times \Rn: \D_{f_i}(x) \neq \{\nabla f_i(x)\}, \text{ for some } i\in [N]  \right\}.
	\end{equation}
	From \cite{bolte2021conservative}, $\ca{Y}$ is a zero-measure subset of $\Rn \times \Rn \times \Rn$. Furthermore, the scheme \eqref{Eq_Framework_ADAM} can be expressed as 
 \begin{equation*}
     \begin{aligned}
         &\left(x_{k+\frac{1}{3}}, m_{k+\frac{1}{3}}, v_{k+\frac{1}{3}} \right) \in (\xk, \mk, \vk) - c\hat{\eta}_k \ca{U}_{1,i_k}(\xk, \mk, \vk),\\
         &\left(x_{k+\frac{2}{3}}, m_{k+\frac{2}{3}}, v_{k+\frac{2}{3}} \right) 
         \in (x_{k+\frac{1}{3}}, m_{k+\frac{1}{3}}, v_{k+\frac{1}{3}}) - c\hat{\eta}_k \ca{U}_{2,i_k}\left(x_{k+\frac{1}{3}}, m_{k+\frac{1}{3}}, v_{k+\frac{1}{3}}\right),\\
         &(\xkp, \mkp, \vkp) \in \left(x_{k+\frac{2}{3}}, m_{k+\frac{2}{3}}, v_{k+\frac{2}{3}}\right) - c\hat{\eta}_k \ca{U}_{3,i_k}\left(x_{k+\frac{2}{3}}, m_{k+\frac{2}{3}}, v_{k+\frac{2}{3}}\right). 
     \end{aligned} 
 \end{equation*}
 By choosing $\{\ca{Q}_k\} = \{\hat{\eta}_k\ca{U}_{1, i_k}, \hat{\eta}_k\ca{U}_{2, i_k}, \hat{\eta}_k\ca{U}_{3, i_k}: k\geq 0\}$ in Proposition \ref{Prop_AS_multiples_step}, there exists a full-measure subset $\hat{\ca{S}} \subseteq \bb{R}_+$, such that for any $c \in \hat{\ca{S}}$, there exists a full-measure subset $\hat{\ca{K}}$ of $\Rn \times \Rn \times \Rn$ with the property that, whenever we choose $(x_0, m_0, v_0) \in \hat{\ca{K}}$, any $\{(\xk, \mk, \vk)\}$ generated by \eqref{Eq_Framework_ADAM} satisfies $\{(\xk, \mk, \vk)\} \cap \ca{Y} = \emptyset$. For such a sequence $\{(\xk, \mk, \vk)\}$, the update scheme in \eqref{Eq_Framework_ADAM} can be reformulated as 
	\begin{equation}
		\left\{
    	\begin{aligned}
    		g_k ={}& \nabla f_{i_k}(\xk), \\
    		\mkp ={}& \mk + \tau_1 \eta_k( g_k - \mk),\\
                \vkp \in{}& \vk + \tau_2 \eta_k( \ca{V}(\xk, \mkp) - \vk),\\
    		\xkp \in{}& \xk - \eta_k  (\ca{P}_{+}(\vkp) + \varepsilon_0)^{-\frac{1}{2}} \odot (\mkp + \rho g_k).
    	\end{aligned}
    	\right.
	\end{equation}
    Therefore, by applying Corollary \ref{Coro_convergence_ADAM_reshuffling} with $\D_f = \partial f$, and choosing $\alpha, \alpha_{ub}> 0$ as defined in Corollary \ref{Coro_convergence_ADAM_reshuffling}, we proceed with $\alpha_c = \min\left\{\frac{\alpha_{ub} }{\sup_{k\geq 0} \hat{\eta}_k}, \frac{\alpha}{1+\limsup_{k\to +\infty} \hat{\eta}_k}\right\}$, $\ca{S} = \hat{\ca{S}} \cap (0,\alpha_c)$ and $\ca{K} = \hat{\ca{K}} \cap \X$, hence completing the proof of this theorem. 
\end{proof}

Similarly, we present the following theorem to show the convergence of \eqref{Eq_Framework_ADAM} with random initialization and with-replacement sampling technique. {The proof of Corollary \ref{Coro_AS_avoid_diminishing_ADAM} follows from the same avoidance argument in Corollary \ref{Coro_AS_avoid_ADAM}, together with Corollary \ref{Coro_convergence_ADAM_WRS}; hence it is omitted for simplicity.}
\begin{coro}
	\label{Coro_AS_avoid_diminishing_ADAM}
	Suppose Assumption \ref{Assumption_f}, Assumption \ref{Assumption_fully_random}, and Assumption \ref{Assumption_framework_ADAM} hold, and $\ca{V}$ is almost everywhere $\ca{C}^1$. Then for any $\varepsilon \in (0,1)$, there exists $\alpha_{c} >0$, a full-measure subset $\ca{S}$ of $(0, \alpha_c)$ and a full-measure subset $\ca{K}$ of $\X$, such that for any $c \in \ca{S}$ and $(x_0, m_0, v_0) \in \ca{K}$, with probability at least $1-\varepsilon$, any cluster point of $\{\xk\}$ lies in $\{x \in \Rn:  0\in \partial f(x)\}$ and the sequence $\{f(\xk)\}$ converges.
\end{coro}

\begin{rmk}
    It is worth mentioning that by selecting a specific set-valued mapping $\ca{V}$, the framework \eqref{Eq_Framework_ADAM} yields several well-known ADAM-family methods, including the ADAM, NADAM, and AdaBelief. From our established results on global stability, these ADAM-family methods inherently enjoy the convergence guarantees in the context of training nonsmooth neural networks. For readers interested in a more comprehensive discussion of these ADAM-family methods, detailed discussions are presented in Appendix \ref{Section_developing_ADAM}. 
\end{rmk}

\section{Conclusion}
In this paper, we introduce a general framework \eqref{Eq_stable_iterate} for stochastic subgradient methods in nonsmooth nonconvex optimization. Our proposed framework \eqref{Eq_stable_iterate} allows for flexible choices of the set-valued mapping $\caH$ and the evaluation noises $\{\xi_k\}$, thus encompassing a wide variety of stochastic subgradient methods.

We establish the global stability of our proposed framework under mild conditions, particularly the coercivity of the corresponding Lyapunov function. Specifically, for any given $\varepsilon > 0$, we demonstrate that by selecting sufficiently small (possibly non-diminishing) stepsizes and approximation parameters, coupled with sufficiently controlled evaluation noises, the iterates ${\xk}$ remain uniformly bounded and asymptotically stabilize within an $\varepsilon$-neighborhood of the stable set. These results provide theoretical convergence guarantees for a wide range of stochastic subgradient methods with potentially non-diminishing stepsizes, which improves existing results in \cite{josz2023lyapunov,bolte2024inexact}.

To demonstrate the global stability of existing stochastic subgradient methods, we introduce the scheme \eqref{Eq_Framework} for developing SGD-type methods, including variants of heavy-ball SGD, SignSGD, ClipSGD, and normalized SGD. We show that \eqref{Eq_Framework} fits within our proposed framework \eqref{Eq_stable_iterate}, where the coercivity of its Lyapunov function directly follows from the coercivity of the objective function $f$. Therefore, from our established results on the global stability of \eqref{Eq_stable_iterate}, we establish the global convergence, global stability, and avoidance of spurious stationary points for the scheme \eqref{Eq_Framework}.

Furthermore, we extend our analysis to the scheme \eqref{Eq_Framework_ADAM}, which encompasses variants of ADAM, AdaBelief, and NADAM. Although the Lyapunov function corresponding to \eqref{Eq_Framework_ADAM} is non-coercive even with a coercive objective function $f$, we develop an improved analysis by introducing an auxiliary update scheme that admits a coercive Lyapunov function. By establishing the equivalence between this auxiliary scheme and \eqref{Eq_Framework_ADAM}, we demonstrate the global convergence, stability, and avoidance of spurious stationary points for \eqref{Eq_Framework_ADAM}. These examples further demonstrate the promising potential and broad applicability of our proposed framework \eqref{Eq_stable_iterate} and our established results on its global stability.

It is important to note that bias correction terms are excluded in the update scheme \eqref{Eq_Framework_ADAM}, which aligns with the schemes analyzed in \cite{zaheer2018adaptive,guo2021novel,shi2021rmsprop,zhang2022adam,zhou2024towards}. However, as shown in \cite{barakat2021convergence}, incorporating the bias correction terms in the ADAM method results in a time-dependent update scheme, hence corresponding to a time-dependent differential inclusion. Currently, there are no established approaches for analyzing the global stability of stochastic subgradient method with a time-dependent update scheme. Therefore, we leave the analysis of the global stability of ADAM with bias correction terms for future investigation.

\section*{Acknowledgements}

The authors wish to express their gratitude to the Associate Editor and
reviewers for their valuable
comments and suggestions that have greatly helped
to improve the presentation of the paper.

\appendix

\section{Proofs of the main results}

\subsection{Proofs for Section \ref{Subsection_31}}

\subsubsection{Proof of Proposition \ref{Prop_abstract_descrease_lyapunov}}
\label{Subsection_proof_00}

\begin{proof}
	We prove this proposition by contradiction. That is, we assume that there exists $\varepsilon > 0$ and $T> 0$, such that for any $\iota > 0$, there exists a trajectory $\bar{x}$ of the differential inclusion \eqref{Eq_stable_DI}, that satisfies $\bar{x}(0) \in \X_0$, $\mathrm{dist}\left( \bar{x}(0), \A  \right) \geq  \varepsilon$,  and $\Psi(\bar{x}(T)) > \Psi(\bar{x}(0)) - \iota$. Then for any $i \in \bb{N}_+$, there exists a trajectory $\bar{x}^{(i)}: [0, T] \to \Rn$ of the differential inclusion \eqref{Eq_stable_DI} satisfying  
	\begin{equation}
		\label{Eq_Prop_abstract_descrease_lyapunov_0}
		\bar{x}^{(i)}(0) \in \X_0, \quad \mathrm{dist}\left( \bar{x}^{(i)}(0), \A  \right) \geq \varepsilon, \quad \Psi(\bar{x}^{(i)}(T)) > \Psi(\bar{x}^{(i)}(0)) - \frac{1}{i+1}.
	\end{equation}
	
	Notice that $\sup_{i\geq 0, ~ 0\leq t\leq T} \norm{\bar{x}^{(i)}(t)} \leq r$, 
    we can conclude that $\{\bar{x}^{(i)}\}$  is uniformly bounded and equicontinuous on $[0,T]$.  Here $L^2[0, T]$ denotes the Hilbert space of square-integrable mappings from $[0, T]$ to $\Rn$. By successively employing the Arzel\`a-Ascoli theorem and Banach-Alaoglu theorem,   there exists a subsequence $\{i_k\}$ and an absolutely continuous function $\hat{x}: [0, T] \to \Rn$ such that  $\{\bar{x}^{(i_k)}\}$  converges uniformly to $\hat{x}$ and $\{(\bar{x}^{(i_k)})'\}$  converges weakly to $\hat{x}'$, respectively in $L^2[0, T]$. 
	
	Then from Banach-Saks theorem, there exists a subsequence $\{i^*_{k}\} \subseteq \{i_k\}$ such that 
	\begin{equation}
		\label{Eq_Prop_abstract_descrease_lyapunov_1}
		\lim_{j\to \infty} \frac{1}{j}\sum_{k = 1}^j\left(\bar{x}^{(i^*_k)}\right)'(t)  = \hat{x}'(t) \quad \text{for almost every} \; t\in [0,T].
	\end{equation}
	{Since $\caH$ is graph-closed, locally bounded, and convex-valued, its restriction to the compact set containing all trajectories is outer semicontinuous. Hence, for almost every $t\in[0,T]$ and any $\varepsilon_H>0$, there exists $k_0(t)$ such that $-\big(\bar{x}^{(i^*_k)}\big)'(t)\in \caH(\hat{x}(t))+\bb{B}_{\varepsilon_H}(0)$ for all $k\geq k_0(t)$. The finitely many terms with $k<k_0(t)$  vanish on the average as $j \to +\infty$, and the convexity of $\caH(\hat{x}(t))+\bb{B}_{\varepsilon_H}(0)$ then yields $-\hat{x}'(t)\in \caH(\hat{x}(t))+\bb{B}_{\varepsilon_H}(0)$. Letting $\varepsilon_H\downarrow0$, and using the closedness of $\caH(\hat{x}(t))$, gives} for almost every $t \in [0, T]$,
	\begin{equation}
		\hat{x}'(t) = \lim_{j\to \infty} \frac{1}{j}\sum_{k = 1}^j\left(\bar{x}^{(i^*_k)}\right)'(t) \in \lim_{j\to \infty}\frac{1}{j}\sum_{k = 1}^j -\caH(\bar{x}^{(i^*_k)}(t)) \subseteq -\caH(\hat{x}(t)). 
	\end{equation}
	As a result, $\hat{x}$ is a trajectory of the differential inclusion \eqref{Eq_stable_DI}, and Definition \ref{Defin_Lyapunov_function} illustrates that $\Psi(\hat{x}(T)) < \Psi(\hat{x}(0))$. On the other hand, \eqref{Eq_Prop_abstract_descrease_lyapunov_0} and \eqref{Eq_Prop_abstract_descrease_lyapunov_1} imply that
	\begin{equation}
		\hat{x}(0) \in \X_0, \quad \mathrm{dist}\left( \hat{x}(0), \A  \right) \geq  \varepsilon, \quad \Psi(\hat{x}(T)) \geq \Psi(\hat{x}(0)),
	\end{equation}
	and we get a contradiction. This completes the proof. 
	
\end{proof}

\subsubsection{Proof of Lemma \ref{Le_stability_close_DI_iter}}
\label{Subsection_proof_01}
\begin{proof}
	We prove this lemma by contradiction. Suppose there exists $\varepsilon > 0$ such that for any $\hat{\alpha}\in (0, \bar{\alpha}]$, there exist two sequences $\{\eta_k\}$ and $\{\delta_k\}$ upper-bounded by $\hat{\alpha}$, where the 
	corresponding sequence $\{\xk\}$ generated by \eqref{Eq_stable_iterate} with $x_0 \in \X_0$ has the property that for any trajectory $\bar{x}$ of the differential inclusion \eqref{Eq_stable_DI}, we have
	\begin{equation}
		\sup_{0\leq t\leq T}\norm{\bar{x}(t) - \hat{x}(t)} \geq  \varepsilon.
	\end{equation}
    Here $\hat{x}$ is the interpolated process of $\{\xk\}$ with respect to the stepsizes $\{\eta_k\}$.
	Then for any $i \geq 0$, we can choose two sequences $\{\eta_{k}^{(i)}\}$ and $\{\delta_{k}^{(i)}\}$ upper-bounded by $\frac{\bar{\alpha}}{i+1}$, and select $\{\xi_{k+1}^{(i)}\}$ that is controlled by $(\frac{\bar{\alpha}}{i+1}, T, \{\eta_k^{(i)}\})$, such that there exists $\{x_{k}^{(i)}\}$  satisfying
	\begin{equation}
		\label{Eq_Le_stability_close_DI_iter_2}
		x_{k+1}^{(i)} = x_{k}^{(i)} - \eta_k^{(i)} (d_k^{(i)} + \xi_{k+1}^{(i)}), \quad d_k^{(i)} \in \ca{H}^{\delta_k^{(i)}}(x_{k}^{(i)}),
	\end{equation}
	and  the following inequality 
	\begin{equation}
		\label{Eq_Le_stability_close_DI_iter_0}
		\sup_{0\leq t\leq T}\norm{\bar{x}(t) - \hat{x}^{(i)}(t)} \geq  \varepsilon,
	\end{equation}
	holds for any trajectory $\bar{x}$ of the differential inclusion \eqref{Eq_stable_DI}. 
	Here $\hat{x}^{(i)}$ is an interpolated process of $\{x_{k}^{(i)}\}$ with respect to the stepsizes $\{\eta_{k}^{(i)}\}$. 
	
	Then let the sequence $\{y^{(i)}_k\}$ be generated by $y^{(i)}_{k+1} = y^{(i)}_k - \eta_k^{(i)} d_k^{(i)}$ with $y^{(i)}_0 = x^{(i)}_0$, where $\{d_k^{(i)}\}$ are from \eqref{Eq_Le_stability_close_DI_iter_2}. From the fact that $\{\xi^{(i)}_{k}\}$ is controlled by $(\frac{\bar{\alpha}}{i+1}, T, \{\eta_k^{(i)}\})$, we can conclude that 
	\begin{equation}
		\sup_{0\leq t\leq T} \norm{\hat{x}^{(i)}(t) - \hat{y}^{(i)}(t)} \leq \frac{\bar{\alpha}}{i+1}. 
	\end{equation}
	Here $\hat{y}^{(i)}$ is the interpolated process of $\{y_{k}^{(i)}\}$ with respect to the stepsizes $\{\eta_{k}^{(i)}\}$. As a result, we can conclude that for any trajectory $\bar{x}$ of the differential inclusion \eqref{Eq_stable_DI}, 
	\begin{equation}
		\label{Eq_Le_stability_close_DI_iter_1}
		\sup_{0\leq t\leq T}\norm{\bar{x}(t) - \hat{y}^{(i)}(t)} \geq  \varepsilon. 
	\end{equation}

	Then similar to the proof techniques as in Proposition \ref{Prop_abstract_descrease_lyapunov}, by successively employing the Arzel\`a-Ascoli theorem, Banach-Alaoglu theorem, and Banach-Saks theorem, there exists a subsequence $\{i^*_{k}\}$ such that  $\{\hat{y}^{(i^*_k)}\}$ uniformly converges to $\hat{x}$, and 
	\begin{equation}
		\lim_{j\to \infty} \frac{1}{j}\sum_{k = 1}^j\left(\hat{y}^{(i^*_k)}\right)'(t)  = \hat{x}'(t) \quad \mbox{for almost every}\; t\in [0,T].
	\end{equation}
	Then {by the same outer-semicontinuity and convexity argument as in the proof of Proposition \ref{Prop_abstract_descrease_lyapunov},} from the fact that $\caH$ is convex-valued, graph-closed, and locally bounded, we can conclude that  $\hat{x}'(t) \in -\caH(\hat{x}(t))$ holds almost everywhere for $t\in [0, T]$. 
	Therefore, $\hat{x}$ is a trajectory of the differential inclusion \eqref{Eq_stable_DI} and there exists a trajectory $\bar{x}$ of the differential inclusion \eqref{Eq_stable_DI} such that 
	\begin{equation}
		\liminf_{i\to \infty} \sup_{0\leq t\leq T}\norm{\bar{x}(t) - \hat{y}^{(i)}(t)} = 0,  
	\end{equation}
	which contradicts \eqref{Eq_Le_stability_close_DI_iter_1}. This completes the proof. 
\end{proof}

\begin{figure}[htbp]
	\centering
	\tikzset{every picture/.style={line width=0.75pt}}
	
	\subfigure[Illustration of \eqref{Eq_Prop_abstract_stability_0}]  
	{  
		\resizebox{0.45\textwidth}{!}{%
			\begin{tikzpicture}[x=0.75pt,y=0.75pt,yscale=-1,xscale=1]
				
				\draw   (129,146.6) .. controls (151,116.6) and (212,136.6) .. (211,156.6) .. controls (210,176.6) and (217,196.6) .. (237,226.6) .. controls (257,256.6) and (190,268.6) .. (170,238.6) .. controls (150,208.6) and (107,176.6) .. (129,146.6) -- cycle ;
				\draw   (269,38.6) .. controls (274,36.1) and (382,133.6) .. (362,153.6) .. controls (342,173.6) and (307,209.6) .. (327,239.6) .. controls (347,269.6) and (150,304.6) .. (130,274.6) .. controls (110,244.6) and (42.5,172.1) .. (69,138.6) .. controls (95.5,105.1) and (131,91.6) .. (137,50.6) .. controls (143,9.6) and (264,41.1) .. (269,38.6) -- cycle ;
				\draw    (168,192.6) .. controls (216,182.6) and (200,217.6) .. (232,175.6) .. controls (263.68,134.02) and (217.93,149.29) .. (273.29,117.58) ;
				\draw [shift={(275,116.6)}, rotate = 150.36] [color={rgb, 255:red, 0; green, 0; blue, 0 }  ][line width=0.75]    (10.93,-3.29) .. controls (6.95,-1.4) and (3.31,-0.3) .. (0,0) .. controls (3.31,0.3) and (6.95,1.4) .. (10.93,3.29)   ;
				\draw [shift={(168,192.6)}, rotate = 348.23] [color={rgb, 255:red, 0; green, 0; blue, 0 }  ][fill={rgb, 255:red, 0; green, 0; blue, 0 }  ][line width=0.75]      (0, 0) circle [x radius= 3.35, y radius= 3.35]   ;
				\draw [color={rgb, 255:red, 0; green, 0; blue, 0 }  ,draw opacity=1 ][line width=0.75]  [dash pattern={on 0.84pt off 2.51pt}]  (267,112.6) -- (236,141.6) -- (253,162.6) -- (245,173.6) -- (227,187.6) -- (210,187.6) -- (170,180.6) ;
				\draw [shift={(170,180.6)}, rotate = 189.93] [color={rgb, 255:red, 0; green, 0; blue, 0 }  ,draw opacity=1 ][fill={rgb, 255:red, 0; green, 0; blue, 0 }  ,fill opacity=1 ][line width=0.75]      (0, 0) circle [x radius= 1.34, y radius= 1.34]   ;
				\draw [shift={(267,112.6)}, rotate = 136.91] [color={rgb, 255:red, 0; green, 0; blue, 0 }  ,draw opacity=1 ][fill={rgb, 255:red, 0; green, 0; blue, 0 }  ,fill opacity=1 ][line width=0.75]      (0, 0) circle [x radius= 1.34, y radius= 1.34]   ;
				
				\draw (40,110) node [anchor=north west][inner sep=0.75pt]   [align=left] {$\displaystyle \mathcal{L}_{3r_{0}}$};
				\draw (148,243) node [anchor=north west][inner sep=0.75pt]   [align=left] {$\displaystyle \mathcal{L}_{2r_{0}}$};
				\draw (154,202) node [anchor=north west][inner sep=0.75pt]  [font=\footnotesize] [align=left] {$\displaystyle \overline{x}^{\star }( 0)$};
				\draw (160,168) node [anchor=north west][inner sep=0.75pt]  [font=\footnotesize] [align=left] {$\displaystyle x_{i}$};
				\draw (264,127) node [anchor=north west][inner sep=0.75pt]  [font=\footnotesize] [align=left] {$\displaystyle \overline{x}^{\star }( T)$};
				\draw (241,98) node [anchor=north west][inner sep=0.75pt]  [font=\footnotesize] [align=left] {$\displaystyle x_{\Lambda ( \lambda ( i) +T)}$};

			\end{tikzpicture}
		}%
	}
	\subfigure[Illustration of \eqref{Eq_Prop_abstract_stability_1}]  
	{  
		\resizebox{0.45\textwidth}{!}{%
			\begin{tikzpicture}[x=0.75pt,y=0.75pt,yscale=-1,xscale=1]
				
				\draw   (117,152.6) .. controls (139,126.6) and (240,126.6) .. (239,146.6) .. controls (238,166.6) and (258,188.6) .. (278,218.6) .. controls (298,248.6) and (165,234.6) .. (145,204.6) .. controls (125,174.6) and (95,178.6) .. (117,152.6) -- cycle ;
				\draw   (269,38.6) .. controls (274,36.1) and (382,133.6) .. (362,153.6) .. controls (342,173.6) and (307,209.6) .. (327,239.6) .. controls (347,269.6) and (150,304.6) .. (130,274.6) .. controls (110,244.6) and (42.5,172.1) .. (69,138.6) .. controls (95.5,105.1) and (131,91.6) .. (137,50.6) .. controls (143,9.6) and (264,41.1) .. (269,38.6) -- cycle ;
				\draw    (175,94.6) .. controls (223,84.6) and (259,124.6) .. (264,149.6) .. controls (268.95,174.35) and (239.6,159.9) .. (206.02,190.65) ;
				\draw [shift={(205,191.6)}, rotate = 316.74] [color={rgb, 255:red, 0; green, 0; blue, 0 }  ][line width=0.75]    (10.93,-3.29) .. controls (6.95,-1.4) and (3.31,-0.3) .. (0,0) .. controls (3.31,0.3) and (6.95,1.4) .. (10.93,3.29)   ;
				\draw [shift={(175,94.6)}, rotate = 348.23] [color={rgb, 255:red, 0; green, 0; blue, 0 }  ][fill={rgb, 255:red, 0; green, 0; blue, 0 }  ][line width=0.75]      (0, 0) circle [x radius= 3.35, y radius= 3.35]   ;
				\draw [color={rgb, 255:red, 0; green, 0; blue, 0 }  ,draw opacity=1 ][line width=0.75]  [dash pattern={on 0.84pt off 2.51pt}]  (178,78.6) -- (243,96.6) -- (230,124.6) -- (274,130.6) -- (272.95,155.27) -- (242,178.6) -- (205,180.6) ;
				\draw [shift={(205,180.6)}, rotate = 182.01] [color={rgb, 255:red, 0; green, 0; blue, 0 }  ,draw opacity=1 ][fill={rgb, 255:red, 0; green, 0; blue, 0 }  ,fill opacity=1 ][line width=0.75]      (0, 0) circle [x radius= 1.34, y radius= 1.34]   ;
				\draw [shift={(178,78.6)}, rotate = 15.48] [color={rgb, 255:red, 0; green, 0; blue, 0 }  ,draw opacity=1 ][fill={rgb, 255:red, 0; green, 0; blue, 0 }  ,fill opacity=1 ][line width=0.75]      (0, 0) circle [x radius= 1.34, y radius= 1.34]   ;
				
				\draw (40,110) node [anchor=north west][inner sep=0.75pt]   [align=left] {$\displaystyle \mathcal{L}_{3r_{0}}$};
				\draw (148,230) node [anchor=north west][inner sep=0.75pt]   [align=left] {$\displaystyle \mathcal{L}_{3r_{0} - \frac{\iota}{2}}$};
				\draw (159,100) node [anchor=north west][inner sep=0.75pt]  [font=\footnotesize] [align=left] {$\displaystyle \overline{x}^{\star }( 0)$};
				\draw (173,68) node [anchor=north west][inner sep=0.75pt]  [font=\footnotesize] [align=left] {$\displaystyle x_{i}$};
				\draw (194,193) node [anchor=north west][inner sep=0.75pt]  [font=\footnotesize] [align=left] {$\displaystyle \overline{x}^{\star }( T)$};
				\draw (183,163) node [anchor=north west][inner sep=0.75pt]  [font=\footnotesize] [align=left] {$\displaystyle x_{\Lambda ( \lambda ( i) +T)}$};

			\end{tikzpicture}
		}%
	}
	
	\caption{Illustrations of the proof technique in Proposition \ref{Prop_abstract_stability}. Here the dash line refers to the interpolated process of $\{x_k: i\leq k\leq \Lambda(\lambda(i)+T)\}$, and the arrowed curve refers to a trajectory $\bar{x}^{\star}$ of the differential inclusion \eqref{Eq_stable_DI}. }
	\label{Fig_illustration_stab_0}
\end{figure}

\subsubsection{Proof of Proposition \ref{Prop_abstract_stability}}

\label{Subsection_proof_1}
We now present the proof of Proposition \ref{Prop_abstract_stability}. Inspired by the pioneering work  \cite{josz2023lyapunov}, Proposition \ref{Prop_abstract_stability} generalizes the global stability results from \cite{josz2023convergence}, extending them from the vanilla SGD to a wider range of subgradient methods.  
Based on the nonsmooth Morse-Sard property in Assumption \ref{Assumption_f_definable}(3), we can choose $r_0>0$ sufficiently large such that $\ca{L}_{r_0}$ encloses the stable set $\A$ of the differential inclusion \eqref{Eq_stable_DI}. Moreover,  together with Proposition \ref{Prop_abstract_descrease_lyapunov}, Lemma \ref{Le_stability_close_DI_iter} provides a technique to estimate the decrease on the function values $\{\Psi(\xk) \,:\,k\geq i\}$ when $x_i \in \ca{L}_{2r_0}$. More precisely, for  any $T>0$ and any sequence $\{\xk\}$ generated by \eqref{Eq_stable_iterate} using sufficiently small stepsizes $\{\eta_k\}$, whenever $x_i \in \ca{L}_{2r_0}$, there exists $\iota > 0$ such that $\Psi(x_{\Lambda(\lambda(i) + T)})- \Psi(x_i) \leq \frac{\iota}{2}$. Then the results of Proposition \ref{Prop_abstract_stability} follow from the coercivity of $\Psi$.

The proof of Proposition \ref{Prop_abstract_stability} begins with induction to ensure that the sequence remains bounded in  $\mathcal{L}_{4r_0}$ in a short time window. Then we utilize {the approximation technique for} the differential inclusion trajectories (Lemma \ref{Le_stability_close_DI_iter}) and descent properties (Proposition \ref{Prop_abstract_descrease_lyapunov}) to estimate the changes in $\{\Psi(\xk)\}$, hence proving that {$\Psi(\xk) \leq 4r_0$} for all $k\geq 0$.

{
\begin{proof}[Proof of Proposition \ref{Prop_abstract_stability}]
For convenience, we let 
	\begin{equation}
		r_0 :=  \max\left\{1, \sup_{x \in \X_0 \cup \A}\Psi(x)\right\}. 
	\end{equation}
	From the fact  that $\{\Psi(x): x \in \ca{A}\}$ is a finite set by Assumption \ref{Assumption_f_definable}(3) and that $\X_0$ is a compact set, we have that $r_0 \geq 1$ is a finite constant.

	We denote $M_D := \sup_{x \in \ca{L}_{5r_0}} \norm{\caH(x)}$, and set $M_{\Psi}$ to be a Lipschitz constant of $\Psi$ over the level set $\ca{L}_{5r_0}$.  The definition of the level set can be found in \eqref{Eq_defin_level_set}. Additionally, it directly follows from the choice of $M_{\Psi}$ that 
    \begin{equation}
        \label{Eq_Prop_abstract_stability_4}
        \mathrm{dist}(\ca{L}_{2r_0}, \Rn \setminus\ca{L}_{3r_0}) \geq \frac{r_0}{M_{\Psi}}, \quad \text{and} \quad \mathrm{dist}(\ca{L}_{3r_0}, \Rn \setminus\ca{L}_{4r_0}) \geq \frac{r_0}{M_{\Psi}}. 
    \end{equation}
    
	Now we first present the following claim demonstrating the boundedness of $\{\xk\}$ within a sufficiently small time interval. 
	\begin{myclaim}
    \label{Claim_Prop_abstract_stability_0}
	    For any $k\geq 0$, any $x_k \in \ca{L}_{3r_0}$, any $T \in \left(0,  \frac{r_0}{4(M_D +1)(M_{\Psi} +1)} \right)$, and any $\hat{\alpha}_0 \in \left(0, \frac{r_0}{2(M_{\Psi} + M_D + r_0 + 1)}\right)$,  
        with any $\hat{\alpha}_0$-upper-bounded $\{\eta_k\}$ and $\{\delta_k\}$ and any $(\hat{\alpha}_0, T, \{\eta_k\})$-controlled noises $\{\xi_k\}$, it holds that $\sup_{k \leq j \leq \Lambda(\lambda(k) + T)}\norm{x_j - x_k} \leq \frac{r_0}{M_{\Psi} + 1}$. 
	\end{myclaim}
    \begin{proof} ({\em Proof of Claim \ref{Claim_Prop_abstract_stability_0}})
        Firstly, for any $x \in \ca{L}_{4r_0}$ and any $\delta \leq \hat{\alpha}_0$, it follows from the choice of $\hat{\alpha}_0$ that $\bb{B}_{\delta}(x) \subset \ca{L}_{5r_0}$. Then we have 
        \begin{equation}
        \label{Eq_Prop_abstract_stability_7}
            \norm{\caH^{\delta}(x)} \leq \delta + \sup_{y \in \bb{B}_{\delta}(x)} \norm{\ca{H}(y)} \leq \hat{\alpha}_0 + M_D.  
        \end{equation}
        Now
         we employ the induction technique to prove this claim. Consider the induction statement: for $k\leq i \leq \Lambda(\lambda(k) + T)$, we have  $\sup_{k \leq j \leq i}\norm{x_j - x_k} \leq \frac{r_0}{M_{\Psi} + 1}$. 

        Observe that the induction statement at the initial step $i = k$ trivially holds. 

        Suppose the induction statement holds for $i \geq k$, then we will prove that the statement holds for $i+1$. From \eqref{Eq_Prop_abstract_stability_4}, it holds that $\{x_j: k\leq j\leq i\} \subset \ca{L}_{4r_0}$. Then from the definition of $M_D$ and \eqref{Eq_Prop_abstract_stability_7}, it holds that 
    \begin{equation}
        \label{Eq_Prop_abstract_stability_6}
        \norm{x_{i+1} - x_k } \leq \norm{\sum_{l = k}^{i} \eta_l \caH^{\delta_l}(x_l)} + \norm{\sum_{l = k}^{i} \eta_l \xi_{l+1}} \leq \left(\sum_{l = k}^{i} \eta_l (M_D + \hat{\alpha}_0) \right) + \hat{\alpha}_0 \leq T(M_D + \hat{\alpha}_0) + \hat{\alpha}_0 < \frac{r_0}{M_{\Psi} + 1}. 
    \end{equation}
    Therefore the induction statement holds for step $i+1.$
    As a result, by induction, we can conclude that $\sup_{k \leq j \leq \Lambda(\lambda(k) + T)}\norm{x_j - x_k} \leq \frac{r_0}{M_{\Psi} + 1}$. This completes the proof of this claim. 
    \end{proof}

	Additionally, we have the following claim illustrating the boundedness of any trajectory $\bar{x}$ of the differential inclusion \eqref{Eq_stable_DI} within a sufficiently small time interval. 
    \begin{myclaim}
    \label{Claim_Prop_abstract_stability_1}
        For any $T \in \left(0,  \frac{r_0}{4(M_D +1)(M_{\Psi} +1)} \right)$ and  any trajectory $\bar{x}$ of the differential inclusion \eqref{Eq_stable_DI} with $\bar{x}(0) \in \ca{L}_{3r_0}$, it holds that $\sup_{0\leq t\leq T}\norm{\bar{x}(t) - \bar{x}(0)} \leq \frac{r_0}{M_{\Psi} +1}$. 
    \end{myclaim}
    \begin{proof} ({\em Proof of Claim \ref{Claim_Prop_abstract_stability_1}})
        We prove this claim by contradiction. That is, for any $T \in \left(0,  \frac{r_0}{4(M_D +1)(M_{\Psi} +1)} \right)$ and  any trajectory $\bar{x}$ of the differential inclusion \eqref{Eq_stable_DI} with $\bar{x}(0) \in \ca{L}_{3r_0}$, we assume that $\sup_{0\leq t\leq T}\norm{\bar{x}(t) - \bar{x}(0)} > \frac{r_0}{M_{\Psi} +1}$, and we denote $\hat{T} := \inf \{t \in [0, T]: \norm{\bar{x}(t) - \bar{x}(0)} \geq \frac{r_0}{M_{\Psi} +1} \} $. 

        From the definition of $\hat{T}$, it holds that $\sup_{0\leq t\leq \hat{T}}\norm{\bar{x}(t) - \bar{x}(0)} \leq \frac{r_0}{M_{\Psi} +1}$. Then together with \eqref{Eq_Prop_abstract_stability_4}, it holds that $\{\bar{x}(t): 0\leq t\leq \hat{T}\} \subseteq \ca{L}_{4r_0}$. Thus it follows from the definition of $M_D$ that 
        \begin{equation*}
            \sup_{0\leq t\leq \hat{T}}\norm{\bar{x}(t) - \bar{x}(0)} \leq M_D \hat{T} \leq M_D \cdot  \frac{r_0}{4(M_D +1)(M_{\Psi} +1)} \leq \frac{r_0}{4(M_{\Psi} +1)} < \frac{r_0}{M_{\Psi} +1}. 
        \end{equation*}
        As a result, we conclude that $\norm{\bar{x}(\hat{T}) - \bar{x}(0)} < \frac{r_0}{M_{\Psi} +1}$, which contradicts the choice of $\hat{T}$. Therefore, we can conclude that $\sup_{0\leq t\leq T}\norm{\bar{x}(t) - \bar{x}(0)} \leq \frac{r_0}{M_{\Psi} +1}$. This completes the proof of this claim. 
    \end{proof}

    Now, based on Claim \ref{Claim_Prop_abstract_stability_1} and Proposition \ref{Prop_abstract_descrease_lyapunov},  we fix $T \in \left(0, \frac{r_0}{4(M_D +1)(M_{\Psi} +1)}\right)$ and choose $\iota$ as the descending coefficient of \eqref{Eq_stable_DI} with respect to $\big(\ca{L}_{5r_0}, \frac{r_0}{4(M_{\Psi}+1)}, T\big)$; see Proposition \ref{Prop_abstract_descrease_lyapunov}.  
    Moreover, we choose  
	\begin{equation}   
		\label{Eq_Prop_abstract_stability_3}
		\hat{\varepsilon} =  \min \left\{\frac{\min\{1, \iota  \}}{8(M_{\Psi}+1)(\hat{\alpha}_0 + M_D+ 1)}, \;\frac{r_0}{8(M_D +1)(M_{\Psi} +1)}, \frac{T}{2}, \hat{\alpha}_0 \right\}. 
	\end{equation}
	From Claim \ref{Claim_Prop_abstract_stability_0} and Lemma \ref{Le_stability_close_DI_iter},  there exists a constant $\alpha \in (0, \hat{\varepsilon})$, such that for any $\alpha$-upper-bounded $\{\eta_k\}$ and $\{\delta_k\}$, any $(\alpha, T, \{\eta_k\})$-controlled $\{\xi_k\}$, any $i > 0$, and any $\{\xk\}$ generated by \eqref{Eq_stable_iterate} with $x_i \in \ca{L}_{3r_0}$,  there exists a trajectory $\bar{x}$ of the differential inclusion \eqref{Eq_stable_DI} with $\norm{\bar{x}(0) - x_i} \leq \hat{\varepsilon}$
    such that $\sup_{0\leq t\leq T}\norm{\bar{x}(t) - \hat{x}(t)} \leq \hat{\varepsilon}$. 
	Here $\hat{x}$ is the interpolated process of $\{\xk: i\leq k \leq \Lambda(\lambda(i) + T)\}$ with respect to the stepsizes $\{\eta_k: i\leq k \leq \Lambda(\lambda(i) + T)\}$. 
	
	With the above preparations, we now proceed to prove that the sequence of iterates $\{\xk\}$ lies within $\ca{L}_{4r_0}$ by considering four steps.
	
	{\bf Step 1}. With the  particular choices of $\alpha$ and $T$ in the penultimate paragraph, for any $\alpha$-upper-bounded $\{\eta_k\}$ and $\{\delta_k\}$, any $(\alpha, T, \{\eta_k\})$-controlled $\{\xi_k\}$, and any $i \geq 0$, when $x_i \in \ca{L}_{2r_0}$, Lemma \ref{Le_stability_close_DI_iter} illustrates that there exists a trajectory $\bar{x}^\star$ of \eqref{Eq_stable_DI} with $\norm{\bar{x}^\star(0) - x_i} \leq \hat{\varepsilon}$ such that $ \sup_{i \leq j \leq  \Lambda(\lambda(i)+T)}\norm{x_j - \bar{x}^\star(\lambda(j) - \lambda(i))} \leq \hat{\varepsilon}$. Together with Claim \ref{Claim_Prop_abstract_stability_0},
    it holds that $\sup_{i\leq j\leq \Lambda(\lambda(i) + T)} \norm{x_j - x_i} \leq \frac{r_0}{M_{\Psi} +1}$. Therefore, it follows from \eqref{Eq_Prop_abstract_stability_4} that 
    \begin{equation}
		\label{Eq_Prop_abstract_stability_0}
        \{x_j: i\leq j \leq \Lambda(\lambda(i) + T)\} \subseteq \ca{L}_{3r_0}. 
    \end{equation}

	
	{\bf Step 2}. On the other hand, when  $x_i \in \ca{L}_{3r_0} \setminus \ca{L}_{2r_0}$, Claim \ref{Claim_Prop_abstract_stability_0} and Lemma \ref{Le_stability_close_DI_iter} illustrate that there exists a trajectory $\bar{x}^\star$ of \eqref{Eq_stable_DI} such that $\norm{\bar{x}^\star(0) - x_i} \leq \hat{\varepsilon}$ and $\norm{x_{\Lambda(\lambda(i) + T)} - \bar{x}^\star\big(\lambda\big(\Lambda(\lambda(i) + T)\big) - \lambda(i)\big)} \leq \hat{\varepsilon}$. Then it follows from Claim \ref{Claim_Prop_abstract_stability_0} and   \eqref{Eq_Prop_abstract_stability_4}  that $\{x_j: i\leq j \leq \Lambda(\lambda(i) + T)\} \subset \ca{L}_{4r_0}$. Moreover, together with Claim \ref{Claim_Prop_abstract_stability_1},  \eqref{Eq_Prop_abstract_stability_4} and the definition of $\hat{\varepsilon}$, we have $\bar{x}^\star(t) \in \ca{L}_{4r_0}$ for any $t \in [0,T]$. Additionally, from the choice of $T$ and the definition of $\iota$, Proposition \ref{Prop_abstract_descrease_lyapunov} illustrates that $\Psi(\bar{x}^\star(T)) \leq \Psi(\bar{x}^\star(0)) - \iota$. 
	As a result, we have 
	\begin{equation}
		\label{Eq_Prop_abstract_stability_1}
		\begin{aligned}
			&\Psi\left(x_{\Lambda(\lambda(i) + T)}\right) \leq \Psi(\bar{x}^\star(T)) + \hat{\varepsilon} M_{\Psi} + \left|\Psi\Big(\bar{x}^\star(T)\Big) - \Psi\Big(\bar{x}^\star\Big(\lambda\big(\Lambda(\lambda(i) + T)\big) - \lambda(i)\Big)\Big) \right|
            \\
            \leq{}& \Psi(\bar{x}^\star(T)) + \hat{\varepsilon} M_{\Psi} + M_{\Psi} \eta_{\Lambda(\lambda(i) + T)}   \norm{\caH^{\delta_{\Lambda(\lambda(i) + T)}}\left(x_{\Lambda(\lambda(i) + T)}\right)}\\
            \leq{}& \Psi(\bar{x}^\star(T)) + \hat{\varepsilon} M_{\Psi} + \hat{\varepsilon}M_{\Psi}(\hat{\alpha}_0 + M_D)
            \\
			={}& \Psi(\bar{x}^\star(0)) + \hat{\varepsilon}M_{\Psi}(1 +\hat{\alpha}_0 + M_D) + \left(\Psi(\bar{x}^\star(T)) - \Psi(\bar{x}^\star(0))\right) \\
			\leq{}&  \Psi(\bar{x}^\star(0)) + \hat{\varepsilon}M_{\Psi} (1 +\hat{\alpha}_0 + M_D) - \iota  \leq \Psi(x_i) + 2\hat{\varepsilon} M_{\Psi}(1 +\hat{\alpha}_0 + M_D) - \iota \\
			<{}&  \Psi(x_i) - \frac{1}{2}\iota \leq 3r_0- \frac{1}{2}\iota . 
		\end{aligned}
	\end{equation}
    Here the second inequality uses \eqref{Eq_Prop_abstract_stability_7} and the fact that $|\lambda(i) + T - \lambda\big(\Lambda(\lambda(i) + T)\big)| \leq \eta_{\Lambda(\lambda(i) + T)} \leq \hat{\varepsilon}$. Moreover, Figure \ref{Fig_illustration_stab_0} illustrates the results in \eqref{Eq_Prop_abstract_stability_0} and \eqref{Eq_Prop_abstract_stability_1}.

    {\bf Step 3}. Furthermore, for any $i \geq 0$ such that $x_i \in \ca{L}_{3r_0}$,  it follows from Claim \ref{Claim_Prop_abstract_stability_0} that $\{x_j: i\leq j \leq \Lambda(\lambda(i) + T)\} \subset \ca{L}_{4r_0}$, and thus 
	\begin{equation}
		\label{Eq_Prop_abstract_stability_2}
		\begin{aligned}
		    &\sup_{i\leq j \leq \Lambda(\lambda(i) + T)} |\Psi(x_j) - \Psi(x_i)| \leq M_\Psi\sup_{i\leq j \leq \Lambda(\lambda(i) + T)}\| x_j-x_i\|
            \\
            \leq{}& M_{\Psi}\sup_{i\leq j \leq \Lambda(\lambda(i) + T)-1} \norm{\sum_{l = i}^j \eta_l \xi_{l+1}} + M_{\Psi}\sum_{i\leq l \leq \Lambda(\lambda(i) + T)-1}\eta_l\norm{\caH^{\delta_l}(x_l)}\\
            \leq{}& M_{\Psi} \hat{\alpha}_0 + M_{\Psi} T (M_D + \hat{\alpha}_0)   \leq M_{\Psi} \hat{\alpha}_0 + M_{\Psi} T (M_D + 1)  \leq \frac{r_0}{4}. 
		\end{aligned}
	\end{equation}

    {\bf Step 4}. Finally, let $\{k_j\}$ be the sequence such that $k_0 = 0$ and $k_{j+1} = \Lambda(\lambda(k_j) + T)$ for any $j \geq 0$. Then by combining \eqref{Eq_Prop_abstract_stability_0} and \eqref{Eq_Prop_abstract_stability_1} together, we claim that $\sup_{j\geq 0}\Psi(x_{k_j}) \leq 3r_0$ and prove it by contradiction. That is, we assume there exists $j^* \geq 0$ such that $\Psi\big(x_{k_{j^*}} \big) \leq 3r_0$ (hence $x_{k_{j^*}}\in\ca{L}_{3r_0}$) and $\Psi\big(x_{k_{j^*+1}} \big) > 3r_0$.
    However, when $x_{k_{j^*}} \in \ca{L}_{2r_0}$, it follows from \eqref{Eq_Prop_abstract_stability_0} that $\Psi\big(x_{k_{j^*+1}} \big) \leq 3r_0$. Also, when $x_{k_{j^*}} \in \ca{L}_{3r_0}\setminus \ca{L}_{2r_0}$, it follows from \eqref{Eq_Prop_abstract_stability_1} that $\Psi\big(x_{k_{j^*+1}} \big) \leq 3r_0$. Therefore, we can conclude that $\Psi\big(x_{k_{j^*+1}} \big) \leq 3r_0$.  This leads to a contradiction and thus verified our claim that $\sup_{j\geq 0}\Psi(x_{k_j}) \leq 3r_0$.

      Together with \eqref{Eq_Prop_abstract_stability_2}, we have 
	\begin{equation*}
		\sup_{k\geq 0} \Psi(\xk) \leq \sup_{j\geq 0} \Psi(x_{k_j}) + \frac{1}{4}r_0  < 4r_0. 
	\end{equation*}
	Therefore, we can conclude that the sequence $\{\xk\}$ is restricted in $\ca{L}_{4r_0}$. 
    Thus for any $\tilde{r} > 4 r_0$, the sequence $\{x_k\}$ 
    generated by (SGM) with $x_0 \in\X_0$ is contained in 
    $\ca{L}_{4r_0}\subset \ca{L}_{\tilde{r}}$. This completes the proof.

\end{proof}

}

\subsubsection{Proof of Theorem \ref{Theo_abstract_stability_convergence_sequence}}
\label{Subsection_proof_2}
In this subsection, we present the proof of Theorem \ref{Theo_abstract_stability_convergence_sequence}. We begin our theoretical analysis with Theorem \ref{Theo_abstract_stability_converge}, which extends the results in \cite[Theorem 1]{josz2023global} and exhibits the convergence properties of the sequence of function values $\{\Psi(\xk)\}$. Compared with those results in \cite[Theorem 1]{josz2023global},  Theorem \ref{Theo_abstract_stability_converge} admits a broader range of Lyapunov functions $\Psi$ and offers more flexibility in the choice of the sequence $\{\eta_k\}$.  

The proof of Theorem \ref{Theo_abstract_stability_converge} utilizes Proposition \ref{Prop_abstract_descrease_lyapunov} and Lemma \ref{Le_stability_close_DI_iter} to demonstrate that the interpolated discrete iterates $\{\xk\}$ generated by \eqref{Eq_stable_iterate} accurately track the trajectories of the differential inclusion \eqref{Eq_stable_DI} over finite time intervals. Then by using the descent property of the Lyapunov function $\Psi$ along these trajectories, together with a contradiction argument on the oscillation of $\{\Psi(\xk)\}$, we prove that the sequence $\{\Psi(\xk)\}$ stabilizes within a neighborhood of the set $\{\Psi(x): x \in \A\}$.

\begin{theo}
	\label{Theo_abstract_stability_converge}
	Suppose Assumption \ref{Assumption_f_definable} holds, and let $\X_0$ be any compact subset of $\Rn$. Then for any $\varepsilon > 0$, there exists $\alpha > 0$ (depending on $\varepsilon$) and $\alpha_{ub}, T_{ub} > 0$ (independent of $\varepsilon$), such that for any $(\alpha_{ub}, \alpha)$-asymptotically-upper-bounded sequences $\{\eta_k\}$ and $\{\delta_k\}$, and any {$(\alpha_{ub}, \alpha, T_{ub}, \{\eta_k\})$}-asymptotically-controlled sequence $\{\xi_k\}$, the sequence $\{\xk\}$ generated by \eqref{Eq_stable_iterate} with $x_0 \in \X_0$ satisfies 
	\begin{equation}
		\mathop{\lim\sup}_{k\to \infty} ~ \mathrm{dist}\left(\Psi(\xk),  \{\Psi(x): x \in \A\} \right) \leq \varepsilon.
	\end{equation}
\end{theo}
\begin{proof}
	From Proposition \ref{Prop_abstract_stability}, there exist $\alpha_1 > 0$, $T_{ub} > 0$, $r_1 > 4\max\{1, \sup_{\X_0 \cup \A} \Psi(x)\}$ such that  for any $\alpha_1$-upper-bounded $\{\eta_k\}$ and $\{\delta_k\}$ and any $(\alpha_1, T_{ub}, \{\eta_k\})$-controlled  $\{\xi_k\}$, the sequence $\{\xk\}$ with $x_0 \in \X_0$ lies in $\ca{L}_{\frac{r_1}{2} }$. 
	Then from the coercivity of $\Psi$, we can conclude that the sequence $\{\Psi(\xk)\}$ is uniformly bounded.

	Denote 
    ${M_H} := \sup_{x \in \ca{L}_{r_1} + \bb{B}_1(0)} \norm{\caH(x)}$, 
    $M_{\Psi}$ as the Lipschitz constant of $\Psi$ over $\ca{L}_{r_1}$,  $\{l_1,..., l_{\hat{N}}\} := \{\Psi(x): x \in \A\}$ such that $l_1 < l_2<\ldots < l_{\hat{N}}$, and  $\nu_{\Psi} := \frac{\min_{2\leq i\leq \hat{N}} |l_{i} - l_{i-1}|}{16(1+M_{\Psi})(1+{M_H})}$. In addition, for any $s > 0$, we denote 
	\begin{equation}
		\ca{R}_{s} := \{x \in \ca{L}_{r_1} \,:\, \mathrm{dist}(\Psi(x), \{\Psi(x): x \in \A\} )\geq s\}. 
	\end{equation}

	Now  for any $\varepsilon \in (0, \min\{\nu_{\Psi}, 1\})$, we choose $T < \min\left\{T_{ub}, \frac{\varepsilon}{16({M_H} + 1)(M_{\Psi} + 1)}\right\}$ and define $\iota$ as the descending coefficient for \eqref{Eq_stable_DI} with respect to $\left(\ca{L}_{r_1}, \frac{\varepsilon}{16({M_H} + 1)(M_{\Psi} + 1)}, T\right)$. Then from Lemma \ref{Le_stability_close_DI_iter}, there exists 
	$\alpha_0 \in \left(0, \min\left\{\alpha_1,  \frac{\min\{\varepsilon, r_1, \iota \}}{8({M_H} + 1)(M_{\Psi} + 1)} \right\} \right)$, such that for any $\alpha_0$-upper-bounded sequences $\{\eta_k\}$ and $\{\delta_k\}$, and any $(\alpha_0, T_{ub}, \{\eta_k\})$-controlled sequence $\{\xi_k\}$, for any $\{\xk\}$ generated by \eqref{Eq_stable_iterate} and any $i \geq 0$ with  $x_i \in \ca{R}_{\frac{\varepsilon}{2}}$, there exists a trajectory $\bar{x}^\star$ of the differential inclusion \eqref{Eq_stable_DI} such that $\bar{x}^\star(0) \in \ca{R}_{\frac{\varepsilon}{2}}$ and 
	\begin{equation}
		\label{Eq_Prop_abstract_stability_converge_1}
		\sup_{0\leq t\leq T} \norm{\bar{x}^\star(t) - \hat{x}(t)} \leq \min\left\{\frac{\iota }{4(M_{\Psi} + 1)}, \frac{\varepsilon}{16 (M_{\Psi} + 1)} \right\}. 
	\end{equation}
	Here $\hat{x}$ is the interpolated process of the sequence $\{x_k: i\leq k\leq \Lambda(\lambda(i) + T)\}$ with respect to $\{\eta_k: i\leq k\leq \Lambda(\lambda(i) + T)\}$. 
	
	Since $\bar{x}^\star(0) \in \ca{R}_{\frac{\varepsilon}{2}}$, by the Lipschitz continuity of $\Psi$ over $\ca{L}_{r_1}$, we have $\mathrm{dist}(\Psi(\bar{x}^\star(0)), \{\Psi(x): x \in \A\}) \leq M_{\Psi}\mathrm{dist}(\bar{x}^\star(0), \A)$. Then it follows from the choice of $\varepsilon$ that $\mathrm{dist}(\bar{x}^\star(0), \A) \geq \frac{\varepsilon}{2(M_{\Psi}+1)} > \frac{\varepsilon}{16({M_H} + 1)(M_{\Psi} + 1)}$. Then it follows from \eqref{Eq_Prop_abstract_stability_converge_1} and Proposition \ref{Prop_abstract_descrease_lyapunov} that 
	\begin{equation}
		\label{Eq_Prop_abstract_stability_converge_0}
		\begin{aligned}
			&\Psi(x_{\Lambda(\lambda(i) + T)}) \leq \Psi(\bar{x}^\star(T)) + \frac{\iota }{4(M_{\Psi} + 1)}  \cdot M_{\Psi} + \alpha_0 {M_H}\\
			={}& \Psi(\bar{x}^\star(0)) + \left( \Psi(\bar{x}^\star(T)) - \Psi(\bar{x}^\star(0)) \right) + \frac{\iota }{4(M_{\Psi} + 1)}  \cdot M_{\Psi} + \alpha_0 {M_H}\leq  \Psi(\bar{x}^\star(0)) -  \iota + \frac{\iota }{4(M_{\Psi} + 1)}  \cdot M_{\Psi} + \alpha_0 {M_H}\\
			\leq{}&  \Psi(x_i) -  \iota + \frac{\iota }{4(M_{\Psi} + 1)}  \cdot M_{\Psi} + \frac{\iota}{8({M_H} + 1)} \cdot {M_H}\leq  \Psi(x_i) - \frac{ \iota}{2}. 
		\end{aligned}
	\end{equation}

	Then based on \eqref{Eq_Prop_abstract_stability_converge_0}, the following elementary observation illustrates that $\mathop{\lim\inf}\limits_{k\to \infty}\Psi(\xk)$ is close to  $\{\Psi(x): x \in \A\}$. 
	\begin{myclaim}
		\label{Claim_Theo_abstract_stability_converge_0}
		For any $(\alpha_1, \frac{\alpha_0}{4})$-asymptotically-upper-bounded sequences $\{\eta_k\}$ and $\{\delta_k\}$ and any $(\alpha_1, \frac{\alpha_0}{4}, T_{ub}, \{\eta_k\})$-asymptotically-controlled sequence $\{\xi_k\}$,  we have  
		\begin{equation*}
			\mathrm{dist}\left( \liminf\limits_{k\to \infty}\Psi(\xk),  \{\Psi(x): x \in \A\} \right) \leq \frac{\varepsilon}{2}.
		\end{equation*}
	\end{myclaim}
	\begin{proof}[Proof of Claim \ref{Claim_Theo_abstract_stability_converge_0}]
    
		We prove this claim by contradiction. Suppose $\mathrm{dist}\big( \mathop{\lim\inf}\limits_{k\to \infty}\Psi(\xk),  \{\Psi(x): x \in \A\} \big) > \frac{\varepsilon}{2}$, then there exists a sequence of indices $\{ k_j  \}$ such that 
		\begin{equation*}
			\{x_{k_j} \} \subset \ca{R}_{\frac{\varepsilon}{2}}, \quad \lim_{j\to \infty} k_j = \infty, \quad \text{and} \quad \lim\limits_{j\to \infty}\Psi(x_{k_j}) = \mathop{\lim\inf}\limits_{k\to \infty}\Psi(\xk).
		\end{equation*}

		Notice that  $T \leq T_{ub}$, thus $\{\xi_{k}\}$ is $(\alpha_1, \frac{\alpha_0}{4}, T, \{\eta_k\})$-asymptotically-controlled. Together with the fact that the sequences $\{\eta_k\}$ and $\{\delta_k\}$  are $(\alpha_{1}, \frac{\alpha_0}{4})$-asymptotically-upper-bounded, there exists a sufficiently large $K>0$ such that 
		\begin{equation*}
			\sup_{k\geq K} \eta_k +\delta_k \leq \alpha_0, \quad \text{and} \quad \sup_{s\geq K,\, s \leq i\leq \Lambda(\lambda(s) + T) }  \norm{ \sum_{k = s}^i \eta_k \xi_{k+1}} \leq \alpha_0.
		\end{equation*}
		Then from \eqref{Eq_Prop_abstract_stability_converge_0}, for any $j\geq  0$ such that  $k_j\geq K$, we have $\Psi(x_{ \Lambda( \lambda(k_j) + T)}) \leq \Psi(x_{k_j}) - \frac{\iota}{2}$. Then it holds that 
		\begin{equation}
			\mathop{\lim\inf}\limits_{k\to \infty}\Psi(\xk) \leq \mathop{\lim\inf}_{j\to \infty} \Psi(x_{ \Lambda( \lambda(k_j) + T) }) \leq \lim\limits_{j\to 
				\infty}\Psi(x_{k_j}) - \frac{\iota}{2} = \mathop{\lim\inf}\limits_{k\to \infty}\Psi(\xk) - \frac{\iota}{2}.
		\end{equation}
		This leads to a contradiction. Therefore, we can conclude that 
		$\mathrm{dist}\big( \mathop{\lim\inf}\limits_{k\to \infty}\Psi(\xk),  \{\Psi(x): x \in \A\} \big) \leq \frac{\varepsilon}{2}$. This completes the proof of Claim \ref{Claim_Theo_abstract_stability_converge_0}. 
	\end{proof}

	Next, we present an elementary observation for estimating $\limsup\limits_{k\to \infty}\Psi(\xk)$, based on the same proof techniques as those in Claim \ref{Claim_Theo_abstract_stability_converge_0}.
	\begin{myclaim}
		\label{Claim_Theo_abstract_stability_converge_1}
		For any $(\alpha_1, \frac{\alpha_0}{4})$-asymptotically-upper-bounded sequences $\{\eta_k\}$ and $\{\delta_k\}$, and any $(\alpha_1, \frac{\alpha_0}{4}, T_{ub}, \{\eta_k\})$-asymptotically-controlled sequence $\{\xi_k\}$, we have $$\mathrm{dist}\left( \mathop{\lim\sup}\limits_{k\to \infty}\Psi(\xk),  \{\Psi(x): x \in \A\} \right) \leq \varepsilon.$$ 
	\end{myclaim}
	\begin{proof}[Proof of Claim \ref{Claim_Theo_abstract_stability_converge_1}]
		
		We prove this claim by contradiction. That is, we assume that there exist $(\alpha_1, \frac{\alpha_0}{4})$-asymptotically-upper-bounded sequences $\{\eta_k\}$ and $\{\delta_k\}$, together with a  $(\alpha_1, \frac{\alpha_0}{4}, T_{ub}, \{\eta_k\})$-asymptotically-controlled sequence $\{\xi_k\}$,  such that we have  $\mathrm{dist}\Big( \mathop{\lim\sup}\limits_{k\to \infty}\Psi(\xk),  \{\Psi(x): x \in \A\} \Big) > \varepsilon$. 
		Notice that  $T \leq T_{ub}$, so $\{\xi_{k}\}$ is $(\alpha_1, \frac{\alpha_0}{4}, T, \{\eta_k\})$-asymptotically-controlled. 
		Therefore, there exists a sufficiently large $K>0$ such that $\sup_{k\geq K} \eta_k  \leq \frac{\alpha_0}{2}$,  $\sup_{k\geq K} \delta_k  \leq \frac{\alpha_0}{2}$, and 
		\begin{equation*}
			\sup_{s\geq K, s \leq i\leq \Lambda(\lambda(s) + T) }  \norm{ \sum_{k = s}^i \eta_k \xi_{k+1}} \leq \frac{\alpha_0}{2}. 
		\end{equation*}
		 Then there exist infinitely many indices $\{k_{j,1}\}$ such that for any $j\geq 0$, it holds that   
		\begin{equation}
			\label{Eq_Prop_abstract_stability_converge_2}
			k_{j,1} \geq K, \quad x_{k_{j,1}} \in \ca{R}_{\varepsilon}, \quad \sup_{j\geq 0} \left|\Psi(x_{k_{j,1}}) - \limsup_{k\to \infty} \Psi(\xk)\right| \leq \frac{\iota}{8}.
		\end{equation}
		For any $j\geq 0$, let $k_{j,0} = \inf \{l \geq 0: \Lambda(\lambda(l) + T) \geq k_{j,1}\}$. Then from the choice of $k_{j,0}$, we can conclude that when $k_{j,0} > K$, it holds that $\lambda({k_{j,1}}) - \lambda({k_{j,0}}) \geq T-\alpha_0$. 
		Therefore, we have  $\norm{x_{k_{j,1}} - x_{\Lambda(\lambda(k_{j,0}) + T)}} \leq ({M_H}+1)\alpha_0$. As a result, when $k_{j,0} \geq K$, from \eqref{Eq_Prop_abstract_stability_converge_0} and the choice of $\alpha_0$,  it holds that  $x_{k_{j,0}} \in \ca{R}_{\frac{\varepsilon}{2} }$, hence    
		\begin{equation}
			\Psi(x_{k_{j,1}}) \leq \Psi(x_{\Lambda(\lambda(k_{j,0}) + T)})  + M_{\Psi}({M_H}+1)\alpha_0 \leq \Psi(x_{k_{j,0}}) - \frac{\iota}{2} +  M_{\Psi}({M_H}+1)\alpha_0 \leq \Psi(x_{k_{j,0}}) - \frac{\iota}{4}. 
		\end{equation}
		As a result, there exist infinitely many indices $\{k_{j,0}\}$ such that 
		\begin{equation}
			\begin{aligned}
				&\limsup_{k\to \infty} \Psi(\xk) \geq \limsup_{j\to \infty}\Psi(x_{k_{j,0}}) \geq \limsup_{j\to \infty}\Psi(x_{k_{j,1}})+ \frac{\iota}{4} \\
				\geq{}& \limsup_{k\to \infty} \Psi(\xk) + \frac{\iota}{8} > \limsup_{k\to \infty} \Psi(\xk).
			\end{aligned}
		\end{equation}
		Here the second last inequality follows from \eqref{Eq_Prop_abstract_stability_converge_2}. 
		This leads to the contradiction and completes the proof of  Claim \ref{Claim_Theo_abstract_stability_converge_1}.

	\end{proof}

	Therefore, by combining Claim \ref{Claim_Theo_abstract_stability_converge_0} and Claim \ref{Claim_Theo_abstract_stability_converge_1} together, we can conclude that 
	\begin{equation}
		\begin{aligned}
			\mathrm{dist}\left( \mathop{\lim\inf}\limits_{k\to \infty}\Psi(\xk),  \{\Psi(x): x \in \A\} \right) \leq  \varepsilon, \quad \mathrm{dist}\left( \mathop{\lim\sup}\limits_{k\to \infty}\Psi(\xk),  \{\Psi(x): x \in \A\} \right) \leq  \varepsilon.
		\end{aligned}
	\end{equation}
	Then together with the fact that $\{\eta_k\}$ and $\{\delta_k\}$ are asymptotically upper-bounded by $(\alpha_1, \frac{\alpha_0}{4})$, we can conclude that 
	\begin{equation}
		\limsup_{k \to +\infty} |\Psi(\xkp) - \Psi(\xk)| \leq M_{\Psi} \norm{\xkp - \xk}  \leq \alpha_0 ({M_H} + \alpha_0 ) M_{\Psi} + \alpha_0 M_{\Psi} \leq \frac{1}{2} \min_{2\leq i\leq \hat{N}} |l_{i} - l_{i-1}|. 
	\end{equation}
	As a result, there exists $i' \in [\hat{N}]$, such that 
	\begin{equation}
		|\mathop{\lim\inf}\limits_{k\to \infty}\Psi(\xk) - l_{i'} | \leq  \varepsilon, \quad |\mathop{\lim\sup}\limits_{k\to \infty}\Psi(\xk) - l_{i'} | \leq  \varepsilon.
	\end{equation}
	Then by choosing $\alpha_{ub} = \alpha_1$ and $\alpha = \frac{\alpha_0}{4}$, we complete the proof.

\end{proof}

Finally, based on Theorem \ref{Theo_abstract_stability_converge}, we present the proof of Theorem \ref{Theo_abstract_stability_convergence_sequence}. We prove it by contradiction. That is, we demonstrate that if the iterates $\{\xk\}$   remain bounded away from the stable set $\A$, then the descent property in Proposition \ref{Prop_abstract_descrease_lyapunov} would force a significant oscillation in the function values $\{\Psi(x_k)\}$, contradicting the asymptotic stability established in Theorem \ref{Theo_abstract_stability_converge}.

\begin{proof}[Proof of Theorem \ref{Theo_abstract_stability_convergence_sequence}]
	
	From Proposition \ref{Prop_abstract_stability}, there exist $\alpha_1, T_0 > 0$, $r_1 > 0$ such that  for any   $\alpha_1$-upper-bounded sequences $\{\eta_k\}$ and $\{\delta_k\}$, and any $(\alpha_1, T_0, \{\eta_k\})$-controlled sequence $\{\xi_k\}$, any sequence $\{\xk\}$ in \eqref{Eq_stable_iterate} lies in $\ca{L}_{\frac{r_1}{2}}$. Denote  $M_D := \sup_{x \in \ca{L}_{r_1}} \norm{\caH(x)}$, and $M_{\Psi}$ as the Lipschitz constant of $\Psi$ over $\ca{L}_{r_1}$. Moreover, for any $\varepsilon > 0$,  we can choose   $T \in \left(0, \min\left\{T_0, \frac{\varepsilon}{16(M_D + 1)(M_{\Psi} + 1)} \right\}\right)$, and denote $\iota$ as the descending coefficient for the differential inclusion \eqref{Eq_stable_DI} with respect to $(\ca{L}_{r_1}, \frac{\varepsilon}{4}, T)$.

	From Lemma \ref{Le_stability_close_DI_iter}, there exists $\alpha_0 > 0$ such that    for any   $\alpha_0$-upper-bounded $\{\eta_k\}$ and $\{\delta_k\}$, any $(\alpha_0, T_0, \{\eta_k\})$-controlled sequence $\{\xi_k\}$, and any $\{\xk\}$ generated by \eqref{Eq_stable_iterate} with  $x_0 \in \X_0$, there exists a trajectory $\bar{x}^\star$ of the differential inclusion \eqref{Eq_stable_DI} satisfying
	\begin{equation}
		\label{Eq_Prop_abstract_stability_convergence_sequence_1}
		\sup_{0\leq t\leq T} \norm{\bar{x}^\star(t) - \hat{x}(t)} \leq \min\left\{\frac{\iota }{4(M_{\Psi} + 1)}, \frac{\varepsilon}{16 (M_{\Psi} + 1)} \right\}.
	\end{equation}
	Here $\hat{x}$ is the interpolated process of the sequence $\{\xk\}$ with respect to $\{\eta_k\}$. 
	
	In addition, from Theorem \ref{Theo_abstract_stability_converge}, there exists $\alpha_{ub} \in (0, \alpha_1)$ and $\alpha_2 \in (0, \alpha_0)$ such that for any  $(\alpha_{ub}, \alpha_2)$-asymptotically-upper-bounded $\{\eta_k\}$ and $\{\delta_k\}$, and any  $(\alpha_{ub}, \alpha_2, T, \{\eta_k\})$-asymptotically-controlled $\{\xi_k\}$, 
	we have  that
	the sequence $\{x_k\}$ generated by \eqref{Eq_stable_iterate} with $x_0 \in \X_0$ satisfies 
	\begin{equation}
    \label{Eq_Prop_abstract_stability_convergence_sequence_0}
		\mathop{\lim\sup}_{k\to \infty} ~ \mathrm{dist}(\Psi(\xk),  \{\Psi(x): x \in \A\} ) \leq \min\left\{\frac{\iota }{16(M_{\Psi} + 1)}, \frac{\varepsilon}{2} \right\}.
	\end{equation}

	Now we prove this theorem by contradiction. Suppose $\mathop{\lim\sup}_{k\to \infty} \mathrm{dist}(\xk, \A) > \varepsilon$. Then there exists a sequence of indices $\{k_j\}$ such that $\mathrm{dist}(x_{k_j}, \A) > \varepsilon$, and $\lim_{j\to \infty} x_{k_j} = \tilde{x}$. For each $j\geq 0$, it directly follows from \eqref{Eq_Prop_abstract_stability_convergence_sequence_1} and Proposition \ref{Prop_abstract_descrease_lyapunov} that there exists trajectories $\bar{x}^{\star, (j)}$ of the differential inclusion \eqref{Eq_stable_DI} satisfying
	\begin{equation}
		\label{Eq_Prop_abstract_stability_convergence_sequence_11}
		\sup_{0\leq t\leq T} \norm{\bar{x}^{\star, (j)}(t) - \hat{x}^{(j)}(t)} \leq \min\left\{\frac{\iota }{4(M_{\Psi} + 1)}, \frac{\varepsilon}{16 (M_{\Psi} + 1)} \right\}.
	\end{equation}
    Here $\hat{x}^{(j)}$ is the interpolated process of the sequence $\{\xk: k\geq k_j\}$ with respect to $\{\eta_k\}$. Then it holds that 
	\begin{equation}
		\begin{aligned}
			&\Psi(x_{ \Lambda(\lambda(k_j) + T )}) \leq \Psi(\bar{x}^{\star, (j)}(T)) + M_{\Psi} \cdot \frac{\iota}{4 (1+M_{\Psi})} \\
			={}&\Psi(\bar{x}^{\star, (j)}(0)) + M_{\Psi} \cdot \frac{\iota}{4 (1+M_{\Psi})} + \left(  \Psi(\bar{x}^{\star, (j)}(T)) - \Psi(\bar{x}^{\star, (j)}(0))\right) \leq \Psi(\bar{x}^{\star, (j)}(0)) + M_{\Psi} \cdot \frac{\iota}{4 (1+M_{\Psi})} - \iota\\
			\leq{}& \Psi(x_{k_j})+ 2M_\Psi \cdot \frac{\iota}{4M_{\Psi}}  - \iota  \leq \Psi(x_{k_j}) - \frac{\iota}{2}. 
		\end{aligned}
	\end{equation}
	Therefore, it holds that $\limsup_{k\to \infty} \Psi(\xk) - \liminf_{k\to \infty} \Psi(\xk) \geq \frac{\iota }{2}$. 
	Then the triangle inequality implies 
	\begin{equation}
		\begin{aligned}
			\mathop{\lim\sup}_{k\to \infty} ~ \mathrm{dist}(\Psi(\xk),  \{\Psi(x): x \in \A\} ) \geq{}& \frac{1}{2} \Big( \limsup_{k\to \infty} \Psi(\xk) - \liminf_{k\to \infty} \Psi(\xk) \Big)\\
			\geq{}& \frac{\iota }{4} > \frac{\iota }{16(M_{\Psi} + 1)},
		\end{aligned}
	\end{equation}
	which contradicts  \eqref{Eq_Prop_abstract_stability_convergence_sequence_0}. Therefore, we have that $\mathop{\lim\sup}_{k\to \infty} \mathrm{dist}(\xk, \A) \leq \varepsilon$. This completes the proof. 
\end{proof}

\subsection{Proofs for Section \ref{Subsection_32} and Section \ref{Subsection_33}}
\label{Subsection_appendix_reshuffling}

\subsubsection{Proof of Theorem \ref{Theo_stable_reshuffling_convergence}}
We first present the following lemma characterizing the update scheme of $\{x_{kN}\}$. 
\begin{lem}
	\label{Le_stable_reshuffling_xkN_scheme}
	Suppose Assumption \ref{Assumption_f_definable} and Assumption \ref{Assumption_stable_reshuffling} hold, and let $\Y$ be any compact subset of $\Rn$. Then for any $\tilde{\varepsilon} > 0$, there exists a constant $\alpha > 0$ such that for any $k\geq 0$, any $\eta_{kN} \in (0, \alpha]$, {any $x_{kN} \in \Y$, and any iterates $\{x_j: kN\leq j\leq (k+1)N\}$ generated by \eqref{Eq_stable_reshuffling}}, we have 
	\begin{equation}
        \label{Eq_Le_stable_reshuffling_xkN_scheme_0}
		x_{(k+1)N} \in x_{kN} - N\eta_{kN} \conv\left( \caH^{\tilde{\varepsilon}}(x_{kN})\right).  
	\end{equation}
\end{lem}
\begin{proof}
	Let $M := \sup\left\{ \max_{i\in [N]}\{ \norm{\ca{U}_i(x)} \} + p(x): \mathrm{dist}\left(x, \Y\right) \leq 2 \right\}$. From the outer-semicontinuity of $\{\ca{U}_i\}$ and $\caH$, and the compactness of $\Y$, there exists $\hat{\varepsilon} \in (0, 1)$ such that $\frac{1}{N} \sum_{l = 1}^N \ca{U}_l^{\hat{\varepsilon}}(x) \subseteq \caH^{\tilde{\varepsilon}}(x)$ 
    holds for any $x \in \{x \in \Rn: \mathrm{dist}(x, \Y) \leq 2\}$. 
    
    Next we use the induction technique to prove that for any $\eta_{kN} \leq \frac{\min\{1, \hat{\varepsilon}\}}{N (1+M)}$,  it holds that $ \sup_{kN \leq j \leq (k+1)N} \mathrm{dist}\left(x_{j}, \Y \right) \leq 1 $. {
    The induction statement clearly holds for the initial step
    $j=kN$ since $x_{kN} \in \ca{Y}.$
    Now
    suppose for a fixed $j \in \{kN, kN+1, \ldots, (k+1)N-1\}$, we have $\sup_{kN \leq l\leq j}\mathrm{dist} \left(x_{l}, \Y \right) \leq 1 $. Then for any $kN \leq l\leq j$ we have $\eta_l p(x_l) \leq 1$ and thus $\bb{B}_{\eta_l p(x_l)}(x_l) \subseteq \{x \in \Rn: \mathrm{dist}(x, \Y) \leq 2\}$.} Then it follows from the (reshuffling) update scheme of \eqref{Eq_stable_reshuffling} and Assumption \ref{Assumption_stable_reshuffling}(1) that 
    {\begin{equation}
        \begin{aligned}
            &\mathrm{dist}\left( x_{j+1}, \Y  \right) \leq \norm{x_{j+1} - x_{kN}} \leq \sum_{l = kN}^{j} \eta_l \norm{\ca{U}_{i_l}^{\eta_l p(x_l)}(x_l)} \\
            \leq{}& \sum_{l = kN}^{j} \eta_l \left( \norm{\bigcup_{y \in \bb{B}_{\eta_l p(x_l)}(x_l)} \ca{U}_{i_l}(y)} + \eta_l p(x_l)  \right) \leq \sum_{l = kN}^{j} \eta_l (M +1)   = (j - kN + 1)\eta_{kN} (M +1) \leq N \eta_{kN} (M +1) \leq 1,
        \end{aligned}
    \end{equation}}
    {which completes the proof of the induction step for $j+1$.}
	Thus, by induction, we have $ \sup_{kN\leq j\leq (k+1)N} \mathrm{dist}\left(x_{j}, \Y  \right) \leq 1 $. 

     As a result, from the update scheme \eqref{Eq_stable_reshuffling},
    it holds for any $l \in [N]$  that $\norm{x_{kN+l} - x_{kN }} \leq l\eta_{kN} M \leq \hat{\varepsilon}$. 
	 We can conclude that $\sup_{l \in [N]} \norm{x_{kN+l} - x_{kN}}  \leq \hat{\varepsilon}$.

    Finally, for any $k \geq 0$ and any $j$ such that $kN \leq j < (k+1)N$, the condition $\norm{x_j - x_{kN}} \leq \hat{\varepsilon}$ implies that $x_{j+1} \in x_j - \eta_{kN} \ca{U}_{i_j}^{\hat{\varepsilon}}(x_{kN})$. Then by taking the summation over the epoch and using the fact that $\{i_j: kN \leq j \leq (k+1)N-1\}$ is a permutation of $[N]$ (i.e., reshuffling), we can conclude that 
	\begin{equation*}
		\begin{aligned}
			x_{(k+1)N}  \in{}& x_{kN} - \sum_{j = kN}^{(k+1)N-1} \eta_j \ca{U}_{i_j}^{\hat{\varepsilon}}(x_{kN}) 
			\subseteq x_{kN} - N\eta_{kN} \left( \frac{1}{N}\sum_{l=1}^N\ca{U}_l^{\hat{\varepsilon}} (x_{kN}) \right) \\
			\subseteq{}& x_{kN} - N\eta_{kN} \conv\left( \caH^{\tilde{\varepsilon}} (x_{kN}) \right). 
		\end{aligned}
	\end{equation*}
	This completes the proof. 
\end{proof}

{Lemma \ref{Le_stable_reshuffling_xkN_scheme} demonstrates that for any compact subset $\Y$, any given $\tilde{\varepsilon} > 0$, any $x_{kN} \in \Y$, and any sufficiently small stepsizes $\{\eta_l: kN \leq l < (k+1)N\}$, the update scheme from $x_{kN}$ to $x_{(k+1)N}$ fits into \eqref{Eq_stable_iterate}. It is worth noting that although the evaluation of the set-valued mappings $\{\ca{U}_i\}$ are inexact in \eqref{Eq_stable_reshuffling} and characterized by $\{\ca{U}_{i_k}^{p(\xk)\eta_k}\}$, we can choose sufficiently small stepsizes $\{\eta_l: kN \leq l < (k+1)N\}$ to control the accumulated errors in the evaluation of $\{\ca{U}_i\}$, and thus guaranteeing the validity of \eqref{Eq_Le_stable_reshuffling_xkN_scheme_0}.}

In the following, we present Proposition \ref{Prop_stable_reshuffling_UB} to show that with sufficiently small stepsizes, the sequence $\{\xk\}$ generated by \eqref{Eq_stable_reshuffling} is uniformly bounded. 
\begin{proof}[Proof of Proposition \ref{Prop_stable_reshuffling_UB}]
 {Let $\hat{r} = \tilde{r}/2$.
	From Proposition \ref{Prop_abstract_stability}, we can conclude that, for any $\hat{r} > 4 \max\{1, \sup_{x \in \X_0 \cup \A} \Psi(x) \}$, there exists $\alpha_2 > 0$ such that for any sequences  $\{\tilde{\eta}_k\}$ and  $\{\tilde{\delta}_k\}$ upper-bounded by $\alpha_2$, and for any sequence $\{ \hat{x}_k \}$ generated by 
	\begin{equation}
		\label{Eq_Prop_stable_reshuffling_UB_0}
		\hat{x}_{k+1} \in \hat{x}_k - \tilde{\eta}_k \conv\left(\caH^{\tilde{\delta}_k}(\hat{x}_k) \right), 
	\end{equation}
	it holds that $\{ \hat{x}_k \} \subseteq \ca{L}_{\hat{r}}$. 
    
    Let $M_{\Psi}$ be the Lipschitz constant of $\Psi$ over $\ca{L}_{2\hat{r}}$, and 
    \begin{equation*}
        M_U:= \sup_{x\in (\ca{L}_{2\hat{r}} + \bb{B}_1(0)), ~ 1\leq i\leq N} \norm{\ca{U}_i(x)} + \sup_{x\in (\ca{L}_{2\hat{r}} + \bb{B}_1(0))}p(x) .
    \end{equation*} 
    Then based on this choice of $\alpha_2$ and $\hat{r}$, it holds that 
    \begin{equation}
        \label{Eq_Prop_stable_reshuffling_dist}
        \mathrm{dist}(\X_0, \Rn \setminus \ca{L}_{\hat{r}}) \geq \frac{\hat{r}}{2 M_{\Psi}}. 
    \end{equation}
    
    We first claim that for any $\alpha_1 \in \left(0, \frac{\min\{1, \hat{r}, \alpha_2\}}{2(N+1)(M_U+1)(M_{\Psi}+1)}\right)$, and any sequence $\{\eta_k\}$ upper-bounded by $\alpha_1$, it holds for the sequence $\{x_k\}$ in \eqref{Eq_stable_reshuffling} that 
    \begin{equation}
        \label{Eq_Prop_stable_reshuffling_UB_2}
        \{x_k: k=0,1,\ldots,N\} \subseteq \ca{L}_{\hat{r}}, \qquad \forall x_0 \in \X_0. 
    \end{equation}
    Otherwise, suppose there exists $0\leq j^* < N$ such that $x_{j^*} \in \ca{L}_{\hat{r}}$ and $x_{j^*+1} \notin \ca{L}_{\hat{r}}$. It follows from the choice of $\{\eta_k\}$ that $\eta_j p(x_j) \leq 1$ and $\eta_j \leq 1$ hold for any $j\leq j^*$. Then from the update scheme \eqref{Eq_stable_reshuffling}, we have that 
    \begin{equation*}
        \begin{aligned}
            \norm{x_{j^*+1} - x_0} \leq \sum_{j = 0}^{j^*} \eta_j \norm{\ca{U}_{i_j}^{\eta_j p(x_j)}(x_j)} \leq \sum_{j = 0}^{j^*} \eta_j \left(\norm{\bigcup_{y \in \bb{B}_{1}(x_j)}\ca{U}_{i_j}(y)} + \eta_j p(x_j)\right) \leq \sum_{j = 0}^{j^*} \eta_j (M_U+1) \leq \frac{\hat{r}}{2(M_{\Psi}+1)}.
        \end{aligned}
    \end{equation*}
    Using \eqref{Eq_Prop_stable_reshuffling_dist},
    the above inequality implies that $x_{j^*+1} \in \ca{L}_{\hat{r}}$, which  contradicts the assumption that $x_{j^*+1}\not\in \ca{L}_{\hat{r}}$. Therefore, we can conclude the validity of \eqref{Eq_Prop_stable_reshuffling_UB_2}.

    Additionally, notice that $M_{\Psi}$ is the Lipschitz constant of $\Psi$ over $\ca{L}_{2\hat{r}}$. Then from the definition of $M_{\Psi}$, it holds that  
    \begin{equation}
        \label{Eq_Prop_stable_reshuffling_UB_1}
        \mathrm{dist}(\ca{L}_{\hat{r}}, \mathbb{R}^n \setminus \ca{L}_{2\hat{r}}) \geq \frac{\hat{r}}{M_{\Psi}}. 
    \end{equation}

    Now  we aim to prove the uniform boundedness of $\{x_l\}$ by induction. 
    Consider the induction statement: suppose 
    \begin{equation*}
		\sup_{l\geq 0}\eta_l \leq \frac{\min\{1, \hat{r}, \alpha_1\}}{4(N+1)(M_U + 1)(M_{\Psi} +1)} ,
	\end{equation*}
    then for $k\geq 1$, we have that $\{x_l: 0\leq l \leq kN\} \subseteq \ca{L}_{2\hat{r}}$ and $\{x_{lN}: 0 \leq l\leq k\}\subseteq \ca{L}_{\hat{r}}$.

    We first check the validity of the induction statement at the initial step $k=1$. Notice that $x_0 \in \X_0 \subseteq \ca{L}_{\hat{r}}$, then \eqref{Eq_Prop_stable_reshuffling_UB_2} illustrates that $\{x_l: l=0,1,\ldots,N\} \subseteq \ca{L}_{\hat{r}} \subseteq \ca{L}_{2\hat{r}}$ and $\{x_0,x_N\} \subset \ca{L}_{\hat{r}}$.
    Hence the induction statement holds for $k=1$.

    Suppose the induction statement holds for $k\geq 1$. 
    We will prove that the statement holds for $k+1$.
    Based on the choice of $M_U$, for any $\xk \in \ca{L}_{2\hat{r}}$ and any $\xkp$ generated by \eqref{Eq_stable_reshuffling} with $\eta_k \leq \alpha_1$, it follows from the choice of $\alpha_1$ that 
    \begin{equation*}
        \norm{\xkp - \xk} \leq \eta_k\norm{ \ca{U}_{i_k}^{\eta_k p(\xk)}(\xk)} \leq \eta_k \left(\eta_k p(\xk) + \sup_{y \in \bb{B}_{\eta_k p(\xk)}(\xk)}\norm{\ca{U}_{i_k}(y)} \right) \leq \eta_k M_{U}. 
    \end{equation*}
    We then claim that $\{x_{l}: kN \leq l\leq (k+1)N\} \subseteq \ca{L}_{2\hat{r}}$ and aim to prove this claim by contradiction. That is, we assume that there exists $kN \leq l^* < (k+1)N$ such that 
    $\{x_l: kN \leq l\leq l^*\} \subseteq \ca{L}_{2\hat{r}}$ and $x_{l^*+1} \notin \ca{L}_{2\hat{r}}$. Then by the definition of $M_U$, it holds that 
    \begin{equation}
        \begin{aligned}
            &\norm{x_{l^*+1} - x_{kN}} \leq \sum_{l = kN}^{l^*} \eta_l \norm{\ca{U}_{i_l}^{\eta_l p(x_l)}(x_l)} \leq \sum_{l = kN}^{l^*}\eta_l \left(\eta_l p(x_l) + \sup_{y \in \bb{B}_{\eta_l p(x_l)}(x_l)}\norm{\ca{U}_{i_l}(y)} \right) \\
            \leq{}& \sum_{l = kN}^{l^*}\eta_l M_U \leq M_U N \sup_{l\geq 0}\eta_l \leq  M_U N \cdot \frac{\min\{1, \hat{r}, \alpha_1\}}{4(N+1)(M_U + 1)(M_{\Psi} +1)} \leq  \frac{\hat{r}}{2(M_{\Psi} +1)} < \mathrm{dist}(\ca{L}_{\hat{r}}, \Rn \setminus \ca{L}_{2\hat{r}}). 
        \end{aligned}
    \end{equation}
    Here the last inequality follows from \eqref{Eq_Prop_stable_reshuffling_UB_1}. 
    On the other hand, by the induction hypothesis, $x_{kN} \in \ca{L}_{\hat{r}}$.
    Since $x_{l^*+1} \notin \ca{L}_{2\hat{r}}$, we have that 
    $\|x_{l^*+1} -x_{kN}\| \geq \mathrm{dist}(\ca{L}_{\hat{r}}, \mathbb{R}^n \setminus \ca{L}_{2\hat{r}})$,
    which leads to a contradiction. Therefore, we can conclude that 
    \begin{equation}
        \label{Eq_Prop_stable_reshuffling_UB_4}
        \{ x_l: kN \leq l\leq (k+1)N\} \subseteq \ca{L}_{2\hat{r}}. 
    \end{equation}
    
     
     Next, from Lemma \ref{Le_stable_reshuffling_xkN_scheme} with $\Y$ chosen as $\ca{L}_{2\hat{r}}$ and $\tilde{\varepsilon}$ chosen as $\alpha_1$, there exists $\alpha_0 \in (0, \frac{\alpha_1}{N})$ such that for any $\{\eta_k\}$ upper-bounded by $\alpha_0$, whenever $\{x_{l}: l\leq (k+1)N\} \subset \ca{L}_{2\hat{r}}$, the sequence $\{x_{lN}: 0\leq l\leq k+1\}$ in \eqref{Eq_stable_reshuffling} satisfies 
	\begin{equation}
        \label{Eq_Prop_stable_reshuffling_UB_3}
		x_{(l+1)N} \in  x_{lN} - N\eta_{lN} \conv\left(\caH^{\hat{\delta}_l}(x_{lN} ) \right),
	\end{equation}
	where $\{\hat{\delta}_l\}$ is upper-bounded by $\alpha_1$ (as we utilize Lemma A.2 with $\tilde{\varepsilon} = \alpha_1$). Therefore, by choosing $\{\eta_l\}$ such that
	\begin{equation*}
		\sup_{l\geq 0}\eta_l \leq \min\left\{ \frac{\alpha_0 }{4N (M_U + 1)}, \frac{\min\{1, \hat{r}, \alpha_1\}}{4(N+1)(M_U + 1)(M_{\Psi} +1)}  \right\},
	\end{equation*}
	and combining Proposition 3.6, \eqref{Eq_Prop_stable_reshuffling_UB_0}, and \eqref{Eq_Prop_stable_reshuffling_UB_3} together, we can conclude that  whenever $\{ x_l: l\leq kN\} \subseteq \ca{L}_{2\hat{r}}$, the sequence $\{x_{lN}: l\leq k+1\}$ lies in $\ca{L}_{\hat{r}}$. Together with \eqref{Eq_Prop_stable_reshuffling_UB_4}, we arrive at the conclusion that $\{x_l: l \leq (k+1)N\} \subseteq \ca{L}_{2\hat{r}}$ and $\{x_{lN}: 0 \leq l\leq k+1\}\subseteq \ca{L}_{\hat{r}}$.
    This completes the induction step.
    
     As a result, by induction, we can conclude that, for any $k\geq 0$, it holds that $\{x_l: 0\leq l \leq kN\} \subseteq \ca{L}_{2\hat{r}}$. This immediately shows that the sequence $\{x_l\}$ is restricted in $\ca{L}_{2\hat{r}}$, hence completing the proof of the proposition. }
\end{proof}

Finally, it is worth mentioning that Lemma \ref{Le_stable_reshuffling_xkN_scheme} illustrates that $\{x_{kN}: k = 1,2,\ldots\}$ can be viewed as an inexact discretization of the differential inclusion \eqref{Eq_stable_DI}. Moreover, Proposition \ref{Prop_stable_reshuffling_UB} illustrates that $\{\xk\}$ generated by \eqref{Eq_stable_reshuffling} is uniformly bounded. Therefore, we can conclude that the proof of Theorem \ref{Theo_stable_reshuffling_convergence} directly follows from the successive applications of Proposition \ref{Prop_stable_reshuffling_UB}, Lemma \ref{Le_stable_reshuffling_xkN_scheme}, and Theorem \ref{Theo_abstract_stability_convergence_sequence}, hence is omitted for simplicity. 


\subsubsection{Proof of Theorem \ref{Theo_stable_stochastic_main}}
\label{Subsection_proof_sto}
In this subsection, we present the proof of Theorem \ref{Theo_stable_stochastic_main}.

We first give some basic notations for the proofs throughout this subsection.
For any given $r \in \bb{R}$, we define the stopping time $\tau_r$ as 
\begin{equation}
	\tau_r := \inf\{k\geq 0: \xk \notin \ca{L}_{r}\}. 
\end{equation}
Then from \eqref{Eq_stable_random}, the update scheme for $\{x_{k\land \tau_r}\}$ can be expressed as 
\begin{equation}
	\label{Eq_stable_stopping}
	x_{(k+1) \land \tau_r } \in x_{k \land \tau_r} - c\eta_k\left( \caH^{\delta_k}(x_{k \land \tau_r}) + \chi(x_{k \land \tau_r}, \zeta_{k+1})  \right) \mathbbm{1}_{\tau_r > k}. 
\end{equation}
Here $\mathbbm{1}_{\tau_r > k}$ refers to the indicator function for the event $\{\tau_r > k\}$.  
It is easy to verify that almost surely over the event $\{\tau_r = \infty\}$, the update scheme in \eqref{Eq_stable_stopping} coincides with the update scheme in \eqref{Eq_stable_random}. Moreover, for any given $r \in \bb{R}$, let $\xi_{r, k+1} = \chi(x_{k \land \tau_r}, \zeta_{k+1}) \mathbbm{1}_{\tau_r > k}$, and $\ca{F}_k$ be the smallest $\sigma$-algebra that contains $\{(x_i, \zeta_i): i\leq k\}$. In addition, we choose $M_{r, \xi} > 0$ such that $M_{r, \xi} \geq \sup_{x \in \ca{L}_r} \norm{\chi(x, \zeta)} $ for almost every $\zeta \in \Omega$.

Furthermore, for any $c > 0$, similar to the definitions in \eqref{Eq_def_lambda_Lambda}, we define the mappings $\lambda_c: \bb{N}_+ \to \bb{R}_+$ and $\Lambda_c: \bb{R}_+ \to \bb{N}_+$ as
\begin{equation}
	\lambda_c(0) := 0, \quad \lambda_c(i) := \sum_{k = 0}^{i-1} c\eta_k, \quad \Lambda_c(t) := \sup  \{k \geq 0: t\geq \lambda_c(k)\}.
\end{equation}
Then from Definition \ref{Defin_controlled_noise}, given a sequence of positive numbers $\{\eta_k\}$ and $\alpha, T > 0$, we say a sequence of vectors $\{\xi_k\}$ is  $(\alpha, T,  \{c\eta_k\} )$-controlled, if and only if  
\begin{equation}
	\sup_{s \leq i\leq \Lambda_c(\lambda_c(s) + T) }  \norm{ \sum_{k = s}^i c\eta_k \xi_{k+1}} \leq \alpha
\end{equation}
holds for any $s \geq 0$.

Our proof strategy relies on a truncation argument to ensure the boundedness of the iterates. Specifically, with the level set $\ca{L}_r$ and a corresponding stopping time $\tau_r$, we construct the truncated sequence $\{x_{k \land \tau_r}\}$, which never leaves $\ca{L}_r$ and coincides with $\{\xk\}$ over the event $\{\tau_r = +\infty\}$. Then we prove that, with sufficiently small stepsizes, the noise terms $\{\xi_{r, k+1}\}$ are controlled with high probability. Consequently, by applying Theorem \ref{Theo_abstract_stability_convergence_sequence}, we prove that the event $\{\tau_r = +\infty\}$ holds with high probability, which further shows that  the original iterates ${\xk}$ are restricted within $\ca{L}_r$ with high probability.

Based on the aforementioned notations, we have the following lemma showing that the sequence $\{\xi_{r,k}\}$ is a uniformly bounded martingale difference sequence with respect to the filtration $\{\ca{F}_k\}$. 
\begin{lem}
	\label{Le_stoppingtime_UB}
	Suppose Assumption \ref{Assumption_stochastic} holds, then for any $r \in \bb{R}$ and any $k\geq 0$, $\xi_{r, k+1}$ is adapted to $\ca{F}_{k+1}$ and 
	\begin{equation}
		\bb{E}[\xi_{r, k+1} |\ca{F}_k] = 0, \quad \norm{\xi_{r, k+1}} \leq M_{r, \xi}. 
	\end{equation}
\end{lem}
\begin{proof}
	From Assumption \ref{Assumption_stochastic}, it is easy to verify that $\{\xi_{r, k}\}$ is uniformly bounded. Moreover, since $\zeta_{k+1}$ is measurable with respect to $\ca{F}_{k+1}$, the random variable $\xi_{r, k+1}$ is measurable with respect to $\ca{F}_{k+1}$ for any $k\geq 0$.  Furthermore, from the definition of $\tau_r$, it holds  for any $k\geq 0$ that 
	\begin{equation}
		\begin{aligned}
			\bb{E}[\xi_{r, k+1} |\ca{F}_k] = \bb{E}[\chi(x_{k \land \tau_r}, \zeta_{k+1}) \mathbbm{1}_{\tau_r > k} | \ca{F}_k]=  \bb{E}[\chi(\xk, \zeta_{k+1}) \mathbbm{1}_{\tau_r > k} | \ca{F}_k] = \bb{E}[\chi(x_{k}, \zeta_{k+1}) | \ca{F}_k] \cdot  \mathbbm{1}_{\tau_r > k} = 0.
		\end{aligned}
	\end{equation}
	This completes the proof. 
\end{proof}

The following proposition  characterizes the concentration property for the random variables $\{\xi_{r, k}\}$. The proof technique of Proposition \ref{Prop_beniam_bounded} directly follows from Lemma \ref{Le_stoppingtime_UB},  \cite[Proposition 4.2]{benaim2006dynamics},  and \cite[Proposition 4.4]{benaim2006dynamics}. 
\begin{prop}
	\label{Prop_beniam_bounded}
	Suppose Assumption \ref{Assumption_f_definable} and Assumption \ref{Assumption_stochastic} hold. Then for any $s, \kappa, T, r >0$, we have that  
	\begin{equation}
		\bb{P}\left(  \sup_{s \leq i\leq \Lambda_c(\lambda_c(s) + T) -1}  \norm{ \sum_{k = s}^i c\eta_k \xi_{r, k+1}} \geq \kappa \right) \leq C_1 \exp\left(   \frac{-\kappa^2}{C_2 M_{r, \xi}^2 c^2\sum_{k = s}^{\Lambda_c(\lambda_c(s) + T)-1} \eta_k^2} \right). 
	\end{equation}
	Here $C_1, C_2$ are positive constants that only depend on $n$.  
\end{prop}
\begin{proof}
    Since the martingale difference sequence $\{\xi_{r, k}\}$ is  uniformly bounded by $M_{r, \xi}$, $\{\xi_{r, k}\}$ is sub-Gaussian for any $k\geq 0$. Then by \cite[Remark 4.3]{benaim2006dynamics}, for any $w \in \Rn$, it holds for any $k\geq 0$ that 
	\begin{equation*}
		\bb{E}\left[ \exp\left( \inner{w, \xi_{r, k+1}} \right) | \ca{F}_k \right] \leq \exp\left( \frac{M_{r, \xi}^2}{2}\norm{w}^2 \right),
	\end{equation*}
        holds almost surely. Therefore, for any $w \in \Rn$ and any $C > 0$, let 
        \begin{equation*}
            Z_i := \exp\left[  \inner{Cw, \sum_{k = s}^i c\eta_{k-1} \xi_{r, k}} - \frac{M_{r, \xi}^2C^2}{2}\sum_{k = s}^i c^2\eta_k^2 \norm{w}^2  \right].
        \end{equation*} 
        Then  for any $i\geq 0$, we have that $\bb{E}[Z_{i+1} | \ca{F}_i] \leq Z_{i}$ (i.e., $\{Z_i\}$ is a supermartingale). 
	Hence for any $\kappa > 0$, and any $C > 0$, it holds that 
	\begin{equation*}
		\begin{aligned}
			&\bb{P}\left( \sup_{s\leq i \leq \Lambda_c(\lambda_c(s) + T)-1} \inner{w, \sum_{k = s}^i c\eta_k \xi_{r, k+1}} \geq \kappa \right)={}\bb{P}\left( \sup_{s\leq i \leq \Lambda_c(\lambda_c(s) + T)-1} \inner{Cw, \sum_{k = s}^i c\eta_k \xi_{r, k+1}} \geq C\kappa \right)\\
			\leq{}& \bb{P}\left( \sup_{s\leq i \leq \Lambda_c(\lambda_c(s) + T)-1} Z_{i+1} \geq \exp\left( C\kappa - \frac{M_{r, \xi}^2C^2}{2} \sum_{k = s}^{\Lambda_c(\lambda_c(s) + T)-1} c^2\eta_k^2 \norm{w}^2 \right) \right)\\
			\leq{}& \exp\left( \left(\frac{M_{r, \xi}^2}{2} \norm{w}^2 \sum_{k = s}^{\Lambda_c(\lambda_c(s) + T)-1} c^2\eta_k^2\right)C^2   - \kappa C \right).
		\end{aligned}
	\end{equation*}
	Here the second inequality holds since $\{Z_k\}$  is a nonnegative supermartingale and $\bb{E}[Z_{s}] \leq 1$. Then from the arbitrariness of $C$, set $C = \frac{\kappa}{M_{r, \xi}^2\norm{w}^2 \sum_{k = s}^{\Lambda_c(\lambda_c(s) + T)-1}c^2\eta_k^2}$, it holds that 
	\begin{equation*}
		\bb{P}\left( \sup_{s\leq i \leq \Lambda_c(\lambda_c(s) + T)-1} \inner{w, \sum_{k = s}^i c\eta_k \xi_{r, k+1}} \geq \kappa \right) \leq \exp \left(\frac{-\kappa^2}{2M_{r, \xi}^2c^2\norm{w}^2 \sum_{k = s}^{\Lambda_c(\lambda_c(s) + T)-1}\eta_k^2 }\right). 
	\end{equation*}
	From the arbitrariness of $w$, there exists constants $C_1$ and $C_2$ that only depend on $n$ such that 
	\begin{equation*}
		\bb{P}\left( \sup_{s\leq i \leq \Lambda_c(\lambda_c(s) + T)-1} \norm{\sum_{k = s}^i c\eta_k \xi_{r, k+1}} \geq \kappa \right) \leq  C_1 \exp\left(\frac{-\kappa^2}{C_2M_{r, \xi}^2 c^2 \sum_{k = s}^{\Lambda_c(\lambda_c(s) + T)-1} \eta_k^2 }\right).
	\end{equation*}
    This completes the proof. 
\end{proof}

Based on Proposition \ref{Prop_beniam_bounded}, we present Proposition \ref{Prop_stable_highprob} to prove that the evaluation noises $\{\xi_{r, k}\}$ can be controlled with high probability. 
\begin{prop}
	\label{Prop_stable_highprob}
	Suppose Assumption \ref{Assumption_f_definable} and Assumption \ref{Assumption_stochastic} hold. Then for any $\varepsilon, \kappa, T, r >0$, there exists $\alpha_{ub}> 0$ such that for any $c \in (0,\alpha_{ub})$, it holds that 
	\begin{equation}
		\bb{P}\left(  \left\{ \exists s \geq 0, \text{ such that } \sup_{s \leq i\leq \Lambda_c(\lambda_c(s) + T)-1 }  \norm{ \sum_{k = s}^i c \eta_k \xi_{r, k+1} } \geq \kappa  \right\} \right) \leq \varepsilon. 
	\end{equation}
\end{prop}
\begin{proof}
	For any given  $\varepsilon, \kappa, T >0$, let $c_0 = \frac{\kappa}{9\sqrt{C_2 M_{r, \xi}^2 T\sup_{k\geq 0}\eta_k}}$, then it is easy to verify that  for any $c \leq c_0$, it holds that 
	\begin{equation}
		\sup_{k\geq 0} \exp\left(   \frac{-\kappa^2}{9C_2 M_{r, \xi}^2 c^2T \eta_k} \right) \leq \frac{1}{10}.
	\end{equation}
	Let  $M_{\chi} = C_1 \sum_{k = 0}^{\infty} \exp\left(   \frac{-\kappa^2}{9C_2 M_{r, \xi}^2 c_0^2 T \eta_k} \right)$, and $\alpha_{ub} = \frac{c_0}{\sqrt{1+ \log(2M_{\chi}/\varepsilon)}}$. Then  Assumption \ref{Assumption_stochastic} illustrates that $M_{\chi} < \infty$, and for any $c \in (0, \alpha_{ub})$, it holds that 
	\begin{equation}
		\begin{aligned}
			&\sum_{k = 0}^{\infty} C_1 \exp\left(   \frac{-(\frac{\kappa}{3})^2}{C_2 M_{r, \xi}^2 c^2T \eta_k} \right) = \sum_{k = 0}^{\infty} C_1 \exp\left(   \frac{-\kappa^2}{9C_2 M_{r, \xi}^2 c_0^2T \eta_k} \cdot \frac{c_0^2}{c^2} \right) = C_1 \sum_{k = 0}^{\infty} \left(\exp\left(   \frac{-\kappa^2}{9C_2 M_{r, \xi}^2 c_0^2T \eta_k}  \right)\right)^{ \frac{c_0^2}{c^2}}\\
			\leq{}&  C_1 \sum_{k = 0}^{\infty} \exp\left(   \frac{-\kappa^2}{9C_2 M_{r, \xi}^2 c_0^2T \eta_k}  \right) \cdot \left(\frac{1}{10} \right)^{ \frac{c_0^2}{c^2}-1} \leq  M_{\chi}\cdot \left(\frac{1}{10} \right)^{ \frac{c_0^2}{c^2}-1} \leq \varepsilon. 
		\end{aligned}
	\end{equation}
	
	Therefore, from Proposition \ref{Prop_beniam_bounded}, for any $l\in \bb{N}_+$, we have
	\begin{equation}
		\sum_{l = 0}^{\infty} \bb{P}\left(  \sup_{\Lambda_c(lT) \leq i\leq \Lambda_c(lT+T)-1 }  \norm{ \sum_{k = \Lambda_c(lT)}^i c\eta_k \xi_{r, k+1} } \geq \frac{\kappa}{3} \right) \leq \varepsilon. 
	\end{equation}
	This implies 
	\begin{equation}
		\label{Eq_Prop_stable_highprob_0}
		\bb{P}\left(  \left\{ \exists l \in \bb{N}_+, \text{ such that } \sup_{\Lambda_c(lT) \leq i\leq \Lambda_c(lT+T) -1}  \norm{ \sum_{k = \Lambda_c(lT)}^i c \eta_k \xi_{r, k+1} } \geq \frac{\kappa}{3}  \right\} \right) \leq \varepsilon. 
	\end{equation}
	
	Notice that for any $l \in \bb{N}_+$ and any $s \in [\Lambda_c(lT), \Lambda_c(lT+T))$, it holds that 
	\begin{equation}
		\begin{aligned}
			&\sup_{s \leq i\leq \Lambda_c(\lambda_c(s) + T)-1 }  \norm{ \sum_{k = s}^i c \eta_k \xi_{r, k+1} } \\
			\leq{}& 2    \sup_{\Lambda_c(lT) \leq i\leq \Lambda_c(lT+T)-1 }  
			\norm{ \sum_{k =\Lambda_c(lT)}^i c \eta_k \xi_{r, k+1} } + \sup_{\Lambda_c((l+1)T) \leq i\leq \Lambda_c((l+1)T +T) -1}  \norm{ \sum_{k = \Lambda_c((l+1)T)}^i c \eta_k \xi_{r, k+1} }.
		\end{aligned}
	\end{equation}
	Then together with \eqref{Eq_Prop_stable_highprob_0}, we can conclude that 
	\begin{equation}
		\begin{aligned}
			&\bb{P}\left(  \left\{ \exists s \geq 0, \text{such that } \sup_{s \leq i\leq \Lambda_c(\lambda_c(s) + T) -1}  \norm{ \sum_{k = s}^i c \eta_k \xi_{r, k+1} } \geq \kappa  \right\} \right) \\
			\leq{}& \bb{P}\left(  \left\{ \exists l \in \bb{N}_+, \text{ such that } \sup_{\Lambda_c(lT) \leq i\leq \Lambda_c(lT+T)-1 }  \norm{ \sum_{k = \Lambda_c(lT)}^i c \eta_k \xi_{r, k+1} } \geq \frac{\kappa}{3}  \right\} \right)\\
			\leq{}& \varepsilon. 
		\end{aligned}
	\end{equation}
	This completes the proof. 
\end{proof}

By combining Proposition \ref{Prop_abstract_stability} and Proposition \ref{Prop_beniam_bounded}, the following theorem illustrates that with sufficiently small stepsizes, the iterates generated by \eqref{Eq_stable_random} {are} uniformly bounded with high probability. 
\begin{theo}
	\label{Theo_stable_stochastic}
	Suppose Assumption \ref{Assumption_f_definable} and Assumption \ref{Assumption_stochastic} hold, let $\X_0$ be any compact subset of $\Rn$ and $x_0 \in \X_0$ in \eqref{Eq_stable_random}.   Then for any $\varepsilon>0$, there exist $\alpha, \tilde{r} > 0$ such that for any $c  \in (0, \alpha)$ and any $\alpha$-upper-bounded sequence  $\{\delta_k\}$, 
	it holds for any sequence $\{\xk\}$ generated by \eqref{Eq_stable_random} that 
	\begin{equation}
		\bb{P}\left( \tau_{\tilde{r}} = \infty  \right) \geq 1-\varepsilon. 
	\end{equation}
\end{theo}
\begin{proof}
	From Proposition \ref{Prop_abstract_stability}, for any given $\varepsilon> 0$, there exists $\alpha_{ub} > 0$, $T_{ub}>0$, 
	$\tilde{r} > 0$ such that for any $c, \delta \leq \alpha_{ub}$, and any   $(\alpha_{ub}, T_{ub}, \{c\eta_k\})$-controlled sequence $\{\xi_{r, k}\}$, 
	the sequence $\{\xk\}$ generated by \eqref{Eq_stable_iterate} with $x_0 \in \X_0$ is restricted in $\ca{L}_{\tilde{r}}$. 
	
	Using Proposition \ref{Prop_stable_highprob},  there exists $\alpha_1 \in (0, \alpha_{ub})$ such that for any $c \leq \alpha_1$, the sequence $\{\chi(x_{k\land \tau_{\tilde{r}}}, \zeta_{k+1}) \mathbbm{1}_{\tau_{\tilde{r}} > k}\}$ is controlled by $(\alpha_{ub}, T_{ub}, \{c\eta_k\})$ with probability at least $1-\varepsilon$.

    Let   $\Omega_1$ be the event where the sequence $\{\chi(x_{k\land \tau_{\tilde{r}}}, \zeta_{k+1}) \mathbbm{1}_{\tau_{\tilde{r}} > k}\}$ is controlled by $(\alpha_{ub}, T_{ub}, \{c\eta_k\})$. Then Proposition \ref{Prop_stable_highprob} illustrates that $\bb{P}(\Omega_1) \geq 1-\varepsilon$.  Moreover, almost surely over $\Omega_1$, for any $c, \delta \leq \alpha_1$, Proposition \ref{Prop_abstract_stability} illustrates that the sequence $\{x_{k\land \tau_{\tilde{r}}}\}$ lies in the region $\ca{L}_{\tilde{r}}$. Therefore, from the definition of the stopping time $\tau_{\tilde{r}}$, we can conclude that for almost every $\omega \in \Omega_1$, $\omega \in \{\tau_{\tilde{r}} = \infty\}$.  As a result, we can conclude that $\bb{P}\left( \tau_{\tilde{r}} = \infty  \right) \geq \bb{P}(\Omega_1) \geq 1-\varepsilon$. 
	This completes the proof. 
\end{proof}

Now we present the proof of Theorem \ref{Theo_stable_stochastic_main} based on the results presented in Proposition \ref{Prop_abstract_stability} and \cite[Proposition 3.33]{benaim2005stochastic}. 
\begin{proof}[Proof of Theorem \ref{Theo_stable_stochastic_main}]
	From Theorem \ref{Theo_stable_stochastic}, for any compact subset $\X_0 \subset \Rn$ and any $\varepsilon \in (0,1)$, there exist $\alpha, r > 0$ such that for any $c \in (0,\alpha)$, it holds that $\bb{P}\left( \left\{\tau_r = \infty \right\} \right) \geq 1-\varepsilon$.

	Moreover, the uniform boundedness of $\{x_{k\land \tau_r}\}$ implies the uniform boundedness of $\{\chi(x_{k\land \tau_r}, \zeta_{k+1}) \mathbbm{1}_{\tau_r > k}\}$. Then together with Lemma \ref{Le_stoppingtime_UB} and  \cite[Proposition 1.3]{benaim2005stochastic}, we have that 
	\begin{equation}
		\lim_{s\to \infty} \sup_{s \leq i\leq \Lambda_c(\lambda_c(s) + T)-1 }  \norm{ \sum_{k = s}^i c\eta_k \chi(x_{k\land \tau_r}, \zeta_{k+1}) \mathbbm{1}_{\tau_r > k}} = 0
	\end{equation}
	holds almost surely for any $T > 0$. As a result, almost surely over the event $\left\{\tau_r = \infty \right\}$, we have
	\begin{equation}
		\label{Eq_Theo_convergence_main_0}
		\lim_{s\to \infty} \sup_{s \leq i\leq \Lambda_c(\lambda_c(s) + T)-1 }  \norm{ \sum_{k = s}^i c\eta_k \chi(\xk, \zeta_{k+1}) } = 0
	\end{equation}
	holds for any $T>0$. 
	
	Notice that almost surely over the event $\left\{\tau_r = \infty \right\}$, the update scheme of $\{\xk\}$ becomes 
	\begin{equation}
		\xkp \in \xk - c\eta_k\left( \caH^{\delta_k}(\xk) + \chi(\xk, \zeta_{k+1})  \right). 
	\end{equation}
	Then by combining \eqref{Eq_Theo_convergence_main_0},  and \cite[Theorem 3.6, Remark 1.5(ii), Proposition 3.27]{benaim2005stochastic}, we can conclude that,  almost surely over the event $\left\{\tau_r = \infty \right\}$,  any cluster point of $\{\xk\}$ lies in $\A$ and the sequence $\{\Psi(\xk)\}$ converges, and $\bb{P}\left( \left\{\tau_r = \infty \right\} \right) \geq 1-\varepsilon$.  This completes the proof. 
\end{proof}

\subsection{Proof for Section \ref{Section_GSGD} and Section \ref{Section_ADAM}}

\subsubsection{Proof of Proposition \ref{Prop_Lyapunov}}
\label{Subsection_appendix_Prop_Lyapunov}
\begin{proof}[Proof of Proposition \ref{Prop_Lyapunov}]
	For any trajectory $(x(s), m(s))$ of the differential inclusion \eqref{Eq_DI_ST}, there exist measurable mappings $l_{f}$ and $l_{\phi}$ such that, for almost every $s\geq 0$, we have $l_{f}(s) \in \D_f(x(s))$, $l_{\phi}(s) \in \partial \phi(m(s))$,  such that  
	\begin{equation*}
		(\dot{x}(s), \dot{m}(s)) \in 
		-\left[
		\begin{matrix}
			l_{\phi}(s) + \rho l_{f}(s)\\
			\tau m(s) - \tau l_{f}(s)\\
		\end{matrix}
		\right].
	\end{equation*}
	Therefore, for almost every $s \geq 0$, we have
	\begin{equation*}
		\begin{aligned}
			&\inner{(\dot{x}(s), \dot{m}(s)), \D_{h}(x(s), m(s))}\\
			={}& \inner{-l_{\phi}(s) - \rho l_{f}(s), \D_f(x(s)) }  -  \inner{  m(s) -  l_{f}(s), \partial\phi(m(s))}\\
			\ni{}& \inner{-l_{\phi}(s) - \rho l_{f}(s), l_f(s) } -   \inner{ m(s) -  l_{f}(s), l_{\phi}(s)}= -\rho \inner{l_{f}(s) , l_{f}(s) } -   \inner{ m(s) , l_{\phi}(s)}. \\
		\end{aligned}
	\end{equation*}
	Now, for any $t \geq 0$, it holds from Definition \ref{Defin_conservative_field_path_int} that 
	\begin{equation*}
		\begin{aligned}
			&h(x(t), m(t)) - h(x(0), m(0)) =  \int_{0}^t \inf_{d_s \in \D_{h}(x(s), m(s))} \inner{(\dot{x}(s), \dot{m}(s)), d_s  } \mathrm{d}s\\
			\leq{}& \int_{0}^t-\rho \inner{l_{f}(s) , l_{f}(s) } -   \inner{ m(s) , l_{\phi}(s)} \mathrm{d}s. 
		\end{aligned}
	\end{equation*}
	Then from Assumption \ref{Assumption_framework}(1),  we can conclude that for any $t_2\geq t_1 \geq 0$, it holds that 
	\begin{equation}
		\label{Eq_Prop_Lyapunov_0}
		h(x(t_1), m(t_1)) \geq h(x(t_2), m(t_2)). 
	\end{equation}
	
	For any  $(x(0), m(0)) \notin \A_h$, we have either $m(0)\neq 0$ or $0\notin \D_f(x(0)) $.  When $m(0)\neq 0$, from the continuity of $m(s)$, there exists $T>0$ such that $\norm{m(s) } > 0$ holds for any $s \in [0, T]$. 
	Also, since $l_{\phi}(s) \in \partial \phi( m(s) )$, Assumption \ref{Assumption_framework}(1) implies that $\inner{m(s) , l_{\phi}(s)} > 0$ for almost every $s \in [0, T]$. Therefore, together with \eqref{Eq_Prop_Lyapunov_0}, we can conclude that 
	\begin{equation*}
		\begin{aligned}
			&h(x(t), m(t)) - h(x(0), m(0)) \leq -\int_{0}^t \inner{ m(s), l_{\phi}(s)} \mathrm{d}s \leq    -\int_{0}^{\min\{t, T\}} \inner{ m(s), l_{\phi}(s)} \mathrm{d}s < 0. 
		\end{aligned}
	\end{equation*}
	
	On the other hand, when $m(0) = 0$ and $0 \notin \D_f(x(0))$, there exists $w \in \Rn$ and $\delta > 0$, such that $\norm{w} = 1$ and $ \inf_{d \in \D_{f}(x(0))}\inner{w, d} \geq \delta$. Then from the outer-semicontinuity of $\D_{f}$, there exists $T > 0$ and a constant $\delta >0$ such that  for any $s \in [0, T]$, we have
	\begin{equation*}
		\inf_{d \in \D_{f}(x(s))}\inner{w, d} \geq \frac{\delta}{2} \quad \text{and} \quad \norm{m(s)} \leq \frac{\delta}{4}.
	\end{equation*}
	Therefore, it holds for any $\tilde{t} \in [0, T]$ that 
	\begin{equation*}
		\begin{aligned}
			&\norm{m(\tilde{t})} \geq |\inner{m(\tilde{t}), w}| \geq \int_{0}^{\tilde{t}} \inner{w, \dot{m}(s)} \mathrm{d}s \\
			={}& \int_{0}^{\tilde{t}} \inner{w, -\tau m(s) + \tau l_f(s)} \mathrm{d}s
			\geq \tau \int_{0}^{\tilde{t}} \left(\inf_{d \in \D_{f}(x(s))}\inner{w, d} - \norm{m(s)} \right) \mathrm{d}s\\
			\geq{}& \frac{\tau \delta \tilde{t}}{4}. 
		\end{aligned}
	\end{equation*}
	Then for any $s \in (0, T)$, it holds that $\inner{m(s), l_{\phi}(s)} > 0$. Together with  \eqref{Eq_Prop_Lyapunov_0}, we achieve that
	\begin{equation*}
		\begin{aligned}
			h(x(t), m(t)) - h(x(0), m(0)) \leq -\int_{0}^t \inner{m(s), l_{\phi}(s)} \mathrm{d}s
			\leq -\int_{0}^{\min\{t, T\}} \inner{m(s), l_{\phi}(s)} \mathrm{d}s < 0. 
		\end{aligned}
	\end{equation*}
	As a result, $h$ is a Lyapunov function for the differential inclusion with $\A_h$ as its stable set. Hence we complete the proof. 
\end{proof}

\subsubsection{Proof of Corollary \ref{Coro_convergence_main}}
\label{Subsection_appendix_Coro_convergence_main}
\begin{proof}[Proof of Corollary \ref{Coro_convergence_main}]
	We first verify the validity of Assumption \ref{Assumption_f_definable} and Assumption \ref{Assumption_stable_reshuffling}, with $\caH = \ca{G}$, $\{\ca{U}_i\} = \{\ca{G}_i\}$,  $\Psi = h$ and $\A = \A_h$. Proposition \ref{Prop_lyapunov_regularity} and Proposition \ref{Prop_Lyapunov} illustrate that the differential inclusion \eqref{Eq_DI_ST} admits $h$ as its definable coercive Lyapunov function with stable set $\A_h:= \{(x, m) \in \Rn \times \Rn: 0 \in \D_f(x), m = 0\}$. This verifies the validity of Assumption \ref{Assumption_f_definable}(2) with $\Psi = h$.  Assumption \ref{Assumption_f_definable}(3) directly follows from Proposition \ref{Prop_lyapunov_regularity} and the definability of $f$. 
	
	Moreover, the {local boundedness} of $\{\ca{G}_i: 1\leq i\leq N\}$ illustrates that the update scheme \eqref{Eq_Framework} can be expressed as 
	\begin{equation}
		(\xkp, \mkp) \in (\xk,\mk)-\eta_k\ca{G}_{i_k}^{\eta_k\tilde{p}(\xk, \mk)} (\xk, \mk),
	\end{equation}
	with $\tilde{p}(x, m) := \tau \left(\norm{m} + \sup_{1\leq i\leq N}\norm{\D_{f_i}(x)}\right)$. Then the local boundedness of $\{\D_{f_i}\}$ leads to the local boundedness of $\tilde{p}$, which verifies Assumption \ref{Assumption_stable_reshuffling}(2). Moreover,  Assumption \ref{Assumption_stable_reshuffling}(1) directly follows from Assumption \ref{Assumption_reshuffling}, while Assumption \ref{Assumption_stable_reshuffling}(3) holds from the expression of $\ca{G}$.

	By choosing $\X_0 = \X$ ($\X$ is defined in Assumption \ref{Assumption_framework}(2)) in Theorem \ref{Theo_stable_reshuffling_convergence}, for any $\varepsilon > 0$, let $q_{\varepsilon} = \sup\{ s \geq 0: \phi(m) \leq \frac{\tau \varepsilon}{4}, \forall\, m\in \bb{B}_s(0) \}$, there exists $\alpha > 0$ (depending on $\varepsilon$) and $\alpha_{ub} > 0$ (independent of $\varepsilon$), such that for any $(\alpha_{ub}, \alpha)$-asymptotically-upper-bounded sequence $\{\eta_k\}$, we have 
	\begin{equation*}
		\begin{aligned}
			\limsup_{k\to \infty} \mathrm{dist}\left( (\xk, \mk), \A_h  \right) \leq \min\left\{q_{\varepsilon}, \frac{\varepsilon}{4}\right\}, \limsup_{k\to \infty} \mathrm{dist}\left( h(\xk, \mk), \{h(x, m): (x, m) \in \A_h \}  \right) \leq  \frac{\varepsilon}{4}. 
		\end{aligned}
	\end{equation*}
	From the expression of $\A_h$ in Proposition \ref{Prop_Lyapunov}, simple norm manipulations imply that 
    \begin{equation*}
        \mathrm{dist}\left( (x, m), \A_h  \right) \geq \frac{1}{2} \left(\mathrm{dist}\left( x, \{x \in \Rn: 0\in \D_f(x)\}  \right) + \norm{m} \right). 
    \end{equation*}
    Hence $\limsup_{k\to \infty} \mathrm{dist}\left( \xk, \{x \in \Rn: 0\in \D_f(x)\}  \right) \leq \varepsilon$.  
	
	In addition, notice that $|h(x, m) - h(\tilde{x}, 0)| \geq |f(x) - f(\tilde{x})| - \frac{1}{\tau}|\phi(m)|$ holds for any $x,m \in \Rn$ and $\tilde{x} \in \{\tilde{x} \in \Rn: 0\in \D_f(\tilde{x})\}$. Thus we have 
	\begin{equation*}
		\begin{aligned}
			&\limsup_{k\to \infty}\mathrm{dist}\left( f(\xk), \{f(x): 0 \in \D_f(x) \}  \right)  = \limsup_{k\to \infty}\mathrm{dist}\left( h(\xk, 0), \{h(x, m): (x, m) \in \A_h \}  \right)\\
			\leq{}& \limsup_{k\to \infty}\mathrm{dist}\left( h(\xk,\mk), \{h(x, m): (x, m) \in \A_h \}  \right) + \frac{1}{\tau}\limsup_{k\to \infty} \phi(\mk)\\
			\leq{}& \frac{\varepsilon}{4} + \frac{\varepsilon}{4} \leq \varepsilon. 
		\end{aligned}
	\end{equation*}
	This completes the proof.

\end{proof}

\subsubsection{Proof of Proposition \ref{Prop_Lyapunov_ADAM_WK}}
\label{Subsection_appendix_Prop_Lyapunov_ADAM_WK}
\begin{proof}[Proof of Proposition \ref{Prop_Lyapunov_ADAM_WK}]
    For any $K > 0$ and any trajectory of the differential inclusion \eqref{Eq_DI_ADAM_WK}, there exists $l_f: \bb{R}_+ \to \Rn$ and $l_u: \bb{R}_+ \to \Rn$ such that $l_f(s) \in \D_f(x(s))$ and $l_u(s) \in \ca{V}(x(s), m(s)) \cdot  \mathbbm{I}_{K}(v(s))$ for almost every $s \geq 0$, and 
    \begin{equation}
        \left( \dot{x}(s), \dot{m}(s), \dot{v}(s) \right) 
        = 
        \left[\begin{matrix}
            -(\ca{P}_+(v(s)) + \varepsilon_0 )^{-\frac{1}{2}} \odot  (m(s) + \rho l_f(s)) \\
            -\tau_1 m(s) + \tau_1 l_f(s)\\
            -\tau_2 v(s) + \tau_2  l_u(s)\\
        \end{matrix}\right].
    \end{equation}
    From the formulation of $h_K$ and Lemma \ref{Le_basic_property_hK}, it holds that 
    \begin{equation}
        \begin{aligned}
            &\inner{\D_{h_K}(x(s), m(s), v(s)),\left( \dot{x}(s), \dot{m}(s), \dot{v}(s) \right)  }\\
            ={}& -\inner{\D_f(x(s)), (\ca{P}_+(v(s)) + \varepsilon_0 )^{-\frac{1}{2}} \odot (m(s) + \rho l_f(s)) } \\
            & + \inner{(\ca{P}_+(v(s)) + \varepsilon_0 )^{-\frac{1}{2}} \odot m(s), -m(s) +  l_f(s)}- \tau_2 K (1-\mathbbm{I}_{K}(v(s))) \norm{v(s)}\\
            & + \frac{\tau_2}{4\tau_1} \inner{ m(s)^{2} \odot (\ca{P}_+(v(s))+\varepsilon_0)^{-\frac{3}{2}} \odot \partial \ca{P}_+(v(s)) , v(s) - l_u(s)  } \\
            \ni{}& -\inner{l_f(s), (\ca{P}_+(v(s)) + \varepsilon_0 )^{-\frac{1}{2}} \odot (m(s) + \rho l_f(s)) } \\
            & + \inner{(\ca{P}_+(v(s)) + \varepsilon_0 )^{-\frac{1}{2}} \odot m(s), -m(s) +  l_f(s)}- \tau_2 K (1-\mathbbm{I}_{K}(v(s))) \norm{v(s)}\\
            & + \frac{\tau_2}{4\tau_1} \inner{ m(s)^{2} \odot (\ca{P}_+(v(s))+\varepsilon_0)^{-\frac{3}{2}} \odot \partial \ca{P}_+(v(s)) , v(s) - l_u(s)  }\\
            \leq{}& - \inner{(\ca{P}_+(v(s)) + \varepsilon_0 )^{-\frac{1}{2}} \odot m(s), m(s)} - \tau_2 K (1-\mathbbm{I}_{K}(v(s))) \norm{v(s)}\\
            & +  \frac{\tau_2}{4\tau_1}\inner{ m(s) \odot (\ca{P}_+(v(s))+\varepsilon_0)^{-\frac{3}{2}} \odot \ca{P}_+(v(s)),  m(s)  } \\
            \leq{}& - \frac{\tau_2 \varepsilon_0}{4\tau_1} \inner{ m(s) \odot (\ca{P}_+(v(s))+\varepsilon_0)^{-\frac{3}{2}}, m(s)  } - \tau_2 K (1-\mathbbm{I}_{K}(v(s)))  \norm{v(s)}. 
        \end{aligned}
    \end{equation}
    Here the first inequality follows from the fact that $l_u(s) \geq 0$ and $\partial\ca{P}_+(v) \odot v = \ca{P}_+(v)$. Therefore, we can conclude that for any initial point $(x(0), m(0), v(0)) \in \Rn \times \Rn \times \Rn$, it holds for any $t \geq 0$ that  
    \begin{equation}
        \label{Eq_Prop_Lyapunov_ADAM_WK_0}
        \begin{aligned}
            &h_K(x(t), m(t), v(t)) - h_K(x(0), m(0), v(0)) \\
            ={}&  \int_{0}^{t} {\inf_{l \in \D_{h_K}(x(s), m(s), v(s))}} \inner{l, (\dot{x}(s), \dot{m}(s), \dot{v}(s)) } \mathrm{d}s\\
            \leq{}&  - \int_{0}^{t}  \left( \frac{\tau_2\varepsilon_0}{4\tau_1} \inner{ m(s) \odot (\ca{P}_+(v(s))+\varepsilon_0)^{-\frac{3}{2}}, m(s)  } + \tau_2 K (1-\mathbbm{I}_{K}(v(s)))  \norm{v(s)} \right) \mathrm{d}s. 
        \end{aligned}
    \end{equation}
    As a result, we can conclude that for any trajectory of the differential inclusion \eqref{Eq_DI_ADAM_WK}, it holds for any $t > 0$ that $h_K(x(t), m(t), v(t)) \leq h_K(x(0), m(0), v(0))$. 

    Now consider the case when $(x(0), m(0), v(0)) \notin  \A_K$. Suppose there exists some $T > 0$ such that  
    \begin{equation}
        \label{Eq_Prop_Lyapunov_ADAM_WK_1}
        h_K(x(T), m(T), v(T)) =  h_K(x(0), m(0), v(0)). 
    \end{equation}  
    Then \eqref{Eq_Prop_Lyapunov_ADAM_WK_0} illustrates that $m(s) = 0$ and $(1-\mathbbm{I}_{K}(v(s))) \cdot \norm{v(s)} = 0$ holds for almost every $s \in [0, T]$. As a result, we have
    \begin{equation*}
        0 = \dot{m}(s) \in  -\tau_1 m(s) + \tau_1 \D_f(x(s)) = \tau_1 \D_f(x(s))
    \end{equation*}
    holds for almost every $s \in [0, T]$. Then with the facts that $(x(t), m(t), v(t))$ is absolutely continuous and $\D_f$ is graph-closed and locally bounded, we can conclude that 
    $0 \in \D_f(x(0))$, $\norm{m(0)} = 0$ and $\norm{v(0)} \leq K$, which contradicts the fact that $(x(0), m(0), v(0)) \notin  \A_K$. As a result, we can conclude that for any $T>0$, whenever $(x(0), m(0), v(0)) \notin  \A_K$, it holds that 
    \begin{equation*}
        h_K(x(T), m(T), v(T)) <  h_K(x(0), m(0), v(0)). 
    \end{equation*}
    This completes the proof. 
\end{proof}

\subsubsection{Proof of Corollary \ref{Coro_convergence_ADAM_reshuffling}}
\label{Section_appendix_Coro_convergence_ADAM_reshuffling}
\begin{proof}[Proof of Corollary \ref{Coro_convergence_ADAM_reshuffling}]
    We first analyze the convergence properties of the sequence $\{(\xk, \mk, \vk)\}$ generated by \eqref{Eq_Framework_ADAM_auxiliary} for a sufficiently large $K > 0$. 
    
    Employing the same proof techniques as those utilized in Proposition \ref{Prop_Stability_ADAM_WK}, we can establish the validity of Assumption \ref{Assumption_f_definable} and Assumption \ref{Assumption_stable_reshuffling} with $\caH = \W_K$, $\{\ca{U}_i\} = \{\W_{K,i}\}$, $\Psi = h_K$ and $\A = \A_K$. Furthermore, the update scheme \eqref{Eq_Framework_ADAM_auxiliary} fits into the following scheme:
    \begin{equation}
        (\xkp, \mkp, \vkp) \in (\xk, \mk, \vk)-\eta_k\W_{K, i_k}^{\eta_k \tilde{p}(\xk, \mk, \vk)}(\xk, \mk, \vk), 
    \end{equation}
    where $\tilde{p}$ is defined as \eqref{Eq_Prop_Stability_ADAM_WK_1}. 
    
    As a result, from Proposition \ref{Prop_Stability_ADAM_WK}, let $K_{ub} > 0$ be defined in the same way as that in Proposition \ref{Prop_Stability_ADAM_WK}, we can conclude that there exists $\alpha_{0} > 0$, such that for any $\alpha_{0}$-upper-bounded sequence $\{\eta_k\}$, the sequence $\{(\xk, \mk, \vk)\}$ generated by \eqref{Eq_Framework_ADAM_auxiliary} with $K = K_{ub} +1$ satisfies  
    \begin{equation}
        \label{Eq_Theo_convergence_ADAM_reshuffling_0}
        h_K(\xk, \mk, \vk)\leq \tilde{r}, \quad \sup_{k\geq 0}\norm{\vk} < K_{ub},
    \end{equation}
    with $\tilde{r} = 8 \max\{1, \sup_{ (x, m, v) \in \X \cup \A } h(x, m, v)\} + 1$. 
    
    Moreover, by Theorem \ref{Theo_stable_reshuffling_convergence}, there exists $\alpha > 0$ (depending on $\varepsilon$) and $\alpha_{ub} \in (0, \alpha_0)$ (independent of $\varepsilon$), such that for any $(\alpha_{ub}, \alpha)$-asymptotically-upper-bounded sequence $\{\eta_k\}$, we have 
    \begin{equation}
        \limsup_{k\to \infty} \mathrm{dist}\left( \xk, \{x \in \Rn: 0\in \D_f(x)\}  \right) \leq \varepsilon, \quad \limsup_{k\to \infty} \mathrm{dist}\left( f(\xk), \{f(x): 0\in \D_f(x)\}  \right) \leq \varepsilon. 
    \end{equation}

    Furthermore, recall that $K_{ub} > 0$ is defined in the same way as that in Proposition \ref{Prop_Stability_ADAM_WK}, then we can conclude from the update scheme \eqref{Eq_Framework_ADAM_auxiliary} that $\sup_{k\geq 0}\norm{\vk} < K_{ub}$ when $K = K_{ub}+1$. As a result, the constraint $\norm{v} \leq K$ is inactive for \eqref{Eq_Framework_ADAM_auxiliary}. Therefore, the sequence $\{\xk, \mk, \vk\}$ exactly follows the update scheme in \eqref{Eq_Framework_ADAM}. This completes the proof. 
\end{proof}


    

\section{Developing Stochastic Subgradient Methods with Convergence Guarantees}

\subsection{SGD-type Methods with Convergence Guarantees}
\label{Section_developing_GSGD}
In this subsection,  we demonstrate that by specific choices of the auxiliary function $\phi$, the framework \eqref{Eq_Framework} yields variants of several well-known SGD-type methods, including heavy-ball SGD, SignSGD, and normalized SGD. More importantly, following the results on the global stability of the framework \eqref{Eq_Framework},  our developed SGD-type methods directly enjoy convergence guarantees in the training of nonsmooth neural networks. Table \ref{Table_explaining_SGD} provides a brief overview of the results, while a comprehensive discussion is presented in the ensuing subsections.

\begin{table}[htbp]
	\centering
	\begin{tabular}{c|ccc}
		\hline
		Variants of SGD-type methods & Update scheme & $\phi(m)$ & $\partial \phi(m)$  \\ \hline
		Heavy-ball SGD & \eqref{Eq_Implementation_SGDW} & $\frac{1}{2} \norm{m}^2$ & $m$ \\
		SignSGD & \eqref{Eq_Implementation_SignSGD} & $\norm{m}_1$ & $\mathrm{sign}(m)$ \\
		Normalized SGD & \eqref{Eq_Implementation_NormalSGD} & $\norm{m}$ & $\mathrm{regu}(m)$ \\ 
		ClipSGD & \eqref{Eq_Implementation_ClipSGD} &  \eqref{Eq_ClipSGD_0} & $\mathrm{clip}_C(m)$ \\\hline
	\end{tabular}
	\caption{A brief summary of variants of several well-known SGD-type methods by specific choices of the auxiliary function $\phi$ in the framework \eqref{Eq_Framework}.}
	\label{Table_explaining_SGD}
\end{table}

\subsubsection{Heavy-ball SGD}
When we choose the auxiliary mapping $\phi$ as $\phi(m) = \frac{1}{2} \norm{m}^2$, the framework \eqref{Eq_Framework} becomes
\begin{equation}
	\label{Eq_Implementation_SGDW}
	\left\{
	\begin{aligned}
		g_k \in{}& \D_{f_{i_k}}(\xk), \\
		\mkp ={}& \mk + \tau \eta_k( g_k - \mk),\\
		\xkp ={}& \xk - \eta_k ( \mkp  +\rho g_k).
	\end{aligned}
	\right.
\end{equation}
Then the method in \eqref{Eq_Implementation_SGDW} is exactly the heavy-ball SGD \cite{polyak1964some} with Nesterov momentum \cite{nesterov2013gradient}, whose convergence properties on minimizing nonsmooth definable functions are established in \cite{ruszczynski2020convergence,le2023nonsmooth,josz2023global}.  Therefore, by choosing $\phi(m) = \frac{1}{2} \norm{m}^2$, we can explain the convergence properties of the heavy-ball SGD method, which aligns with the existing results in  \cite{ruszczynski2020convergence,le2023nonsmooth,josz2023global}.

\subsubsection{SignSGD}
To develop efficient variants of the SGD method, SignSGD \cite{bernstein2018signsgd} emerges as a notable alternative, which employs the sign function to normalize the update direction. The SignSGD method iterates by the following schemes, 
\begin{equation}
	\label{Eq_Example_SignSGD}
	\begin{aligned}
		\xkp \in{}& \xk - \eta_k  \mathrm{sign}(g_k )  .
	\end{aligned}
\end{equation}
However, the SignSGD method in \eqref{Eq_Example_SignSGD} does not have any convergence guarantee in solving stochastic optimization problems in general, even if we assume the differentiability of the objective functions. Furthermore, as shown later in Remark \ref{Rmk_counter_example_SignSGD}, the SignSGD method in \eqref{Eq_Example_SignSGD} fails to converge on nonsmooth optimization problems, even in deterministic settings. Therefore, although the numerical experiments exhibit the high efficiency of the SignSGD method and its variants (e.g., Lion \cite{chen2023symbolic}), their convergence properties in training nonsmooth neural networks are largely unexplored.

Motivated by the SignSGD method \cite{bernstein2018signsgd} and Lion method \cite{chen2023symbolic}, we consider choosing $\phi(m) = \norm{m}_1$ in the framework \eqref{Eq_Framework}, which yields the following sign-map regularized momentum SGD method (SRSGD), 
\begin{equation}
	\label{Eq_Implementation_SignSGD}
	\tag{SRSGD}
	\left\{
	\begin{aligned}
		g_k \in{}& \D_{f_{i_k}}(\xk), \\
		\mkp ={}& \mk + \tau \eta_k( g_k - \mk),\\
		\xkp \in{}& \xk - \eta_k ( \mathrm{sign}(\mkp) +\rho g_k).
	\end{aligned}
	\right.
\end{equation}
Compared with the SignSGD method, we can conclude that the method in \eqref{Eq_Implementation_SignSGD} is a variant of the SignSGD method, where we introduce the momentum terms for acceleration. More importantly, our proposed method in \eqref{Eq_Implementation_SignSGD} enjoys theoretical guarantees in nonsmooth optimization, especially in training nonsmooth neural networks.

\begin{rmk}
	\label{Rmk_counter_example_SignSGD}
	It is important to note that while the set-valued mapping $(\partial \phi) \circ \D_f$ is convex-valued for heavy-ball SGD methods, it may not be convex-valued for arbitrary choices of $\phi$ satisfying Assumption \ref{Assumption_framework}(1).  Consequently, the SignSGD method in \eqref{Eq_Example_SignSGD} might fail to converge even in deterministic settings.

	In the following, we present a simple counterexample that illustrates the non-convergence properties of the SignSGD method in nonsmooth deterministic settings, which employs the update scheme for minimizing a function $f$ with $\D_f = \partial f$, 
	\begin{equation}
		\label{Eq_Rmk_counter_example_1}
		\xkp \in \xk - \eta_k \mathrm{sign}(\partial f(\xk)).
	\end{equation}
	More precisely, we aim to show that for a special objective function $f$ and a particular initial point, there exists a sequence $\{\xk\}$ generated by \eqref{Eq_Rmk_counter_example_1} that is uniformly bounded but cannot converge to any Clarke stationary point of $f$.

	Consider the function $f$ in $\bb{R}^2$ defined as $f(u, v) = |2u + v| + |u+10|$. The Clarke stationary point of $f$ is $(-10, 20)$. However, notice that $(0,0) \in \mathrm{conv} \left(\mathrm{sign}(\partial f(0, 0)) \right)$. Hence some trajectories of the differential inclusion $\frac{\mathrm{d}x}{\mathrm{d}t} \in -\mathrm{sign}(\partial f(x))$ may fail to converge to Clarke stationary points of $f$. 
	
	Now we aim to construct the counter-example. 
	It can be easily verified that for any $\varepsilon \in [-5, 5]$, we have $(1,1) \in \mathrm{sign}(\partial f(\varepsilon, \varepsilon))$ and $(-1,-1) \in \mathrm{sign}(\partial f(-\varepsilon, -\varepsilon))$.
	Now, we choose a function $\chi: \bb{R}^2 \to \bb{R}^2$ such that $\chi(u, v) \in \mathrm{sign}(\partial f(u, v))$ for any $(u, v) \in \bb{R}^2$, and $\chi(0,0) = (1,1)\in \mathrm{sign}(\partial f(0, 0))$. 
	Then for any $(u_0, v_0) \in \bb{R}^2$ with $u_0 = v_0 = \varepsilon_0$, for some $\varepsilon_0 \in (0, \frac{1}{3})$, let $\{(u_k, v_k)\}$ be the sequence generated by
	\begin{equation}
		(u_{k+1}, v_{k+1}) \in (u_k, v_k) - \eta_k \chi(u_k, v_k),
	\end{equation}
	where $\eta_k \leq \frac{1}{3}$. It is easy to verify that the sequence $\{(u_k, v_k)\}$ follows the update scheme in \eqref{Eq_Rmk_counter_example_1}. 
	Moreover, with $(d_{1,k}, d_{2, k}) = \chi(u_k, v_k)$, it can be easily verified that $d_{1,k} = d_{2, k}$ holds for all $k\geq 0$, and $\sup_{k\geq 0} |u_k| \leq 1$.
	As a result, for any $k\geq 0$, we have $u_k = v_k$, implying that $\{(u_k, v_k)\}$ lies on the line segment $\{(u, v): u-v = 0, |u| \leq 1\}$, hence any cluster point of $\{(u_k, v_k)\}$ is not a Clarke stationary point of $f$. 
	
\end{rmk}

\begin{rmk}
    For the well-recognized Lion method \cite{chen2023symbolic}, its update scheme can be reformulated as 
    \begin{equation}
        \label{Eq_Example_Lion}
        \left\{
        \begin{aligned}
            g_k \in{}& \D_{f_{i_k}}(\xk), \\
		\mkp ={}& \mk + \tau_1 \eta_k( g_k - \mk),\\
		\xkp \in{}& \xk - \eta_k ( \mathrm{sign}( (1-\tau_2\eta_k) \mk + \tau_2 \eta_k g_k)  ).
        \end{aligned}
        \right.
    \end{equation}
    Here $\tau_1, \tau_2 > 0$ are the parameters for the momentum terms in \eqref{Eq_Example_Lion}. Then it is easy to verify that the scheme \eqref{Eq_Example_Lion} fits into \eqref{Eq_Framework} with $\phi(m) = \norm{m}_1$. Therefore, the results in Corollary \ref{Coro_convergence_diminishing}, Corollary \ref{Coro_convergence_main_sto}, and Corollary \ref{Coro_AS_avoid} provide theoretical guarantees for the scheme \eqref{Eq_Example_Lion} in minimizing nonsmooth nonconvex functions, which further explain its convergence in the training of nonsmooth neural networks.  
\end{rmk}

\subsubsection{Normalized SGD}

The normalized SGD method, proposed by \cite{you2017large,you2019large}, is a popular variant of SGD-type methods and exhibits its high efficiency in training deep neural networks. Let $\ca{I}_1,\ldots,\ca{I}_{\tilde{N}}$ be non-intersecting subsets of $[n]$, and $\cup_{i \in [\tilde{N}]} \ca{I}_i= [n]$. In addition, for any $x \in \Rn$ and any $\ca{I} \subseteq [n]$, we define $(x)_{\ca{I}} := \{y\in \Rn: \text{$(y)_i = (x)_i$ for any $i \in \ca{I}$, otherwise, $(y)_i = 0$}\}$. 

The normalized SGD method in \cite{you2017large,you2019large} employs the following update scheme, 
\begin{equation}
	\label{Eq_Example_NormalSGD}
	(\xkp)_{\ca{I}_j} \in (\xk)_{\ca{I}_j} - \eta_k \mathrm{regu}( (g_k)_{\ca{I}_j} ), \qquad \forall j \in [\tilde{N}]. 
\end{equation}
Notice that when we choose $\tilde{N} = n$, the normalized SGD recovers the SignSGD method in \eqref{Eq_Example_SignSGD}. Therefore, there exists a nonsmooth convex function such that the normalized SGD method may fail to converge, as demonstrated in Remark \ref{Rmk_counter_example_SignSGD}. In the presence of that counter-example, the normalized SGD method lacks firm convergence guarantees in training nonsmooth neural networks. 

In this part, motivated by the normalized SGD method in \eqref{Eq_Example_NormalSGD}, we consider the following specific instantiation of the framework \eqref{Eq_Framework} by selecting $\phi(m) = \sum_{j = 1}^{\tilde{N}} \norm{(m)_{\ca{I}_j}}$,  
\begin{equation}
	\label{Eq_Implementation_NormalSGD}
	\left\{
	\begin{aligned}
		g_k \in{}& \D_{f_{i_k}}(\xk), \\
		\mkp ={}& \mk + \tau \eta_k( g_k - \mk),\\
		(\xkp)_{\ca{I}_j} \in{}& (\xk)_{\ca{I}_j} - \eta_k \left(\mathrm{regu}( (\mkp)_{\ca{I}_j} ) + \rho(g_k)_{\ca{I}_j} \right), \qquad \forall j \in [\tilde{N}]. \\
	\end{aligned}
	\right.
\end{equation}
The update scheme \eqref{Eq_Implementation_NormalSGD} is similar to the update scheme of the normalized SGD method \cite{you2017large,you2019large}. Hence \eqref{Eq_Implementation_NormalSGD} can be regarded as a momentum-accelerated version for the normalized SGD method proposed by \cite{you2017large,you2019large}. Moreover, our convergence results presented in Corollary \ref{Coro_convergence_main} and Corollary \ref{Coro_AS_avoid} provide theoretical guarantees for the momentum normalized SGD method in \eqref{Eq_Implementation_NormalSGD}, especially in the training of nonsmooth neural networks. 

\subsubsection{ClipSGD}

Some recent works  \cite{zhang2020adaptive,zhang2020improved} employ the clipping operator to avoid extreme values in the update direction of the SGD method to stabilize the training process, which yields the following SGD method with clipping (ClipSGD),
\begin{equation}
	\label{Eq_Example_ClipSGD}
	\xkp = \xk - \eta_k \mathrm{clip}_{C}(g_k).
\end{equation}
However, most of the existing works for ClipSGD  \cite{zhang2020adaptive,zhang2020improved} focus on its convergence in minimizing continuously differentiable functions. The convergence guarantees for the ClipSGD method in \eqref{Eq_Example_ClipSGD} are limited,  especially in training nonsmooth neural networks.  To the best of our knowledge, only \cite{xiao2023adam} investigated the convergence of the momentum-accelerated ClipSGD in training nonsmooth neural networks. Their analysis requires the assumption of the uniform boundedness of the iterates, and the clip parameter $C$ should increase to infinity. 

Motivated by the ClipSGD method in \eqref{Eq_Example_ClipSGD}, we consider to choose $\phi$ as the  Huber loss function defined by, 
\begin{equation}
	\label{Eq_ClipSGD_0}
	\phi(x) = \sum_{i = 1}^n \tilde{s}(x_i), \quad \text{where} \quad 
	\tilde{s}(x) = \begin{cases}
		\frac{1}{2} x^2 & \text{when } |x| \leq C,\\
		C|x| - \frac{1}{2} C^2& \text{when }|x| > C.
	\end{cases}
\end{equation}
Such a choice of the auxiliary function $\phi$ in \eqref{Eq_Framework} corresponds to the following SGD-type method,
\begin{equation}
	\label{Eq_Implementation_ClipSGD}
	\left\{
	\begin{aligned}
		g_k \in{}& \D_{f_{i_k}}(\xk), \\
		\mkp ={}& \mk + \tau \eta_k( g_k - \mk),\\
		\xkp ={}& \xk - \eta_k ( \mathrm{clip}_C(\mkp) +\rho g_k).
	\end{aligned}
	\right.
\end{equation}
The method in \eqref{Eq_Implementation_ClipSGD} can be viewed as a momentum-accelerated version of the ClipSGD method. Based on the convergence properties of the framework \eqref{Eq_Framework} established in Section 4, we can easily conclude that \eqref{Eq_Implementation_ClipSGD} converges in training nonsmooth neural networks. Therefore, our results, when applied to analyze the ClipSGD method in \eqref{Eq_Implementation_ClipSGD}, extend the existing results in \cite{zhang2020adaptive,zhang2020improved}, and provide theoretical guarantees for applying the ClipSGD method in the training of nonsmooth neural networks.

\subsection{ADAM-family methods with convergence guarantees}
\label{Section_developing_ADAM}
In this subsection,  we demonstrate that by specifying the choices of the set-valued mapping $\ca{V}$, the framework \eqref{Eq_Framework_ADAM} yields variants of several well-known ADAM-family methods, including ADAM, NADAM, and AdaBelief. Moreover, these ADAM-family methods inherit convergence guarantees from the results established in Corollary \ref{Coro_convergence_ADAM_reshuffling}, Corollary \ref{Coro_convergence_ADAM_WRS}, and Corollary \ref{Coro_AS_avoid_ADAM}.

\subsubsection{ADAM}
The ADAM method, introduced by \cite{kingma2014adam}, has emerged as one of the most widely adopted stochastic subgradient methods in the training of neural networks. Motivated by the ADAM method investigated in \cite{barakat2021convergence,xiao2023adam}, we consider the following scheme,
\begin{equation}
	\label{Eq_Implementation_ADAM}
	\left\{
	\begin{aligned}
		g_k \in{}& \D_{f_{i_k}}(\xk), \\
		\mkp ={}& \mk + \tau_1 \eta_k( g_k - \mk),\\
            \vkp ={}& \vk + \tau_2 \eta_k( g_k \odot g_k - \vk),\\
		\xkp ={}& \xk - \eta_k  (\ca{P}_{+}(\vkp) + \varepsilon_0)^{-\frac{1}{2}} \odot (\mkp + \rho g_k).
	\end{aligned}
	\right.
\end{equation}
Here $\rho \geq 0$ refers to the parameter for Nesterov momentum. When $\rho = 0$, the scheme \eqref{Eq_Implementation_ADAM} corresponds to the ADAM method. Conversely, when $\rho > 0$, it aligns with the NADAM method proposed by \cite{dozat2016incorporating}.

It is easy to verify that the scheme \eqref{Eq_Implementation_ADAM} fits into the scheme \eqref{Eq_Framework_ADAM} by choosing the set-valued mapping $\ca{V}$ in \eqref{Eq_Framework_ADAM} as 
\begin{equation*}
    \ca{V}(x, m) = \frac{1}{N} \sum_{i = 1}^N  \mathrm{conv}\Big(\{d\odot d: d \in \D_{f_i}(x)\} \Big).  
\end{equation*}
Notice that the set-valued mapping $\ca{V}$ is locally bounded and graph-closed. Moreover, for any $(x, m) \in \Rn \times \Rn$, it holds that $\ca{V}(x, m) \subset \Rn_+$. Therefore, our convergence results presented in Corollary \ref{Coro_convergence_ADAM_reshuffling}, Corollary \ref{Coro_convergence_ADAM_WRS}, and Corollary \ref{Coro_AS_avoid_ADAM} directly provide theoretical guarantees for  \eqref{Eq_Implementation_ADAM}, especially in the training of nonsmooth neural networks.

\subsubsection{AdaBelief}
The AdaBelief method, proposed by \cite{zhuang2020adabelief}, is a popular variant of the ADAM-family method, demonstrating high efficiency in the training of deep neural networks. Motivated by the AdaBelief method presented in \cite{dozat2016incorporating,xiao2023adam}, we consider the following scheme for AdaBelief method,
\begin{equation}
	\label{Eq_Implementation_AdaBelief}
	\left\{
	\begin{aligned}
		g_k \in{}& \D_{f_{i_k}}(\xk), \\
		\mkp ={}& \mk + \tau_1 \eta_k( g_k - \mk),\\
            \vkp ={}& \vk + \tau_2 \eta_k( (g_k-\mkp) \odot (g_k-\mkp) - \vk),\\
		\xkp ={}& \xk - \eta_k  (\ca{P}_{+}(\vkp) + \varepsilon_0)^{-\frac{1}{2}} \odot (\mkp + \rho g_k).
	\end{aligned}
	\right.
\end{equation}
Then the AdaBelief method in \eqref{Eq_Implementation_AdaBelief} fits into the scheme \eqref{Eq_Framework_ADAM} with 
\begin{equation*}
    \ca{V}(x, m) = \frac{1}{N} \sum_{i = 1}^N  \mathrm{conv}\Big(\{ (d-m)\odot (d-m): d \in \D_{f_i}(x)\} \Big),
\end{equation*}
which is locally bounded, graph-closed, and lies in $\Rn_+$ for all $(x, m) \in \Rn \times \Rn$. Consequently, the convergence results established in Corollaries \ref{Coro_convergence_ADAM_reshuffling}, \ref{Coro_convergence_ADAM_WRS} and \ref{Coro_AS_avoid_ADAM} can be directly applied to establish the convergence properties of the scheme \eqref{Eq_Implementation_AdaBelief}.

\section{Numerical Experiments}
\label{Section_numerical}
In this section, we aim to show that our proposed scheme \eqref{Eq_Framework} can yield efficient stochastic subgradient methods for training nonsmooth neural networks. By choosing the auxiliary function $\phi(m) = \norm{m}_1$, we consider the following sign-map regularized momentum SGD method (SRSGD),
\begin{equation}
	\label{Eq_Numerical_SRSGD}
	\tag{SRSGD}
	\left\{
	\begin{aligned}
		g_k \in{}& \D_{f_{i_k}}(\xk), \\
		\mkp ={}& \mk + \tau \eta_k( g_k - \mk),\\
		\xkp \in{}& \xk - \eta_k ( \mathrm{sign}(\mkp) +\rho g_k).
	\end{aligned}
	\right.
\end{equation}
Although the formulation of the SRSGD is similar to the SIGNUM developed in \cite{bernstein2018signsgd}, these two methods differ in the choice of the stepsizes for $\{\xk\}$ and $\{\mk\}$. In SIGNUM, the stepsizes for $\{\mk\}$ are fixed as a constant while the stepsizes for $\{\xk\}$ are diminishing. On the other hand, the ratios of the stepsizes for $\{\mk\}$ and those for $\{\xk\}$ are fixed as a constant in SRSGD. As a result, the existing results on the convergence of SIGNUM are restricted to differentiable cases, while the convergence of SRSGD enjoys convergence guarantees in minimizing nonsmooth nonconvex path-differentiable functions.

We evaluate the numerical performance of the SRSGD by comparing it with previously released state-of-the-art methods, including Lion (\url{https://github.com/google/automl}) \cite{chen2023symbolic}, Adam \cite{kingma2014adam}, and AdamW \cite{loshchilov2017decoupled}. In our empirical analysis, we specifically investigate the efficiency of the SRSGD method in training ResNet-50  \citep{he2016deep} for image classification tasks on the CIFAR-10 and CIFAR-100 datasets \citep{krizhevsky2009learning}.  It is worth mentioning that the ResNet-50 neural network employs ReLU as its activation function, hence the loss functions in our numerical experiments are non-Clarke-regular but definable. Moreover, we maintained a batch size of $128$ and a fixed weight decay parameter of $10^{-4}$ for consistency across different test cases. All numerical experiments in this section are conducted on a server equipped with an Intel Xeon 6342 CPU and 4 NVIDIA GeForce RTX 3090 GPUs, running Python 3.8 and PyTorch 1.9.0.

\begin{figure}[!tbp]
	\centering
	\subfigure[Test accuracy, CIFAR-10]{
		\begin{minipage}[t]{0.33\linewidth}
			\centering
			\includegraphics[width=\linewidth]{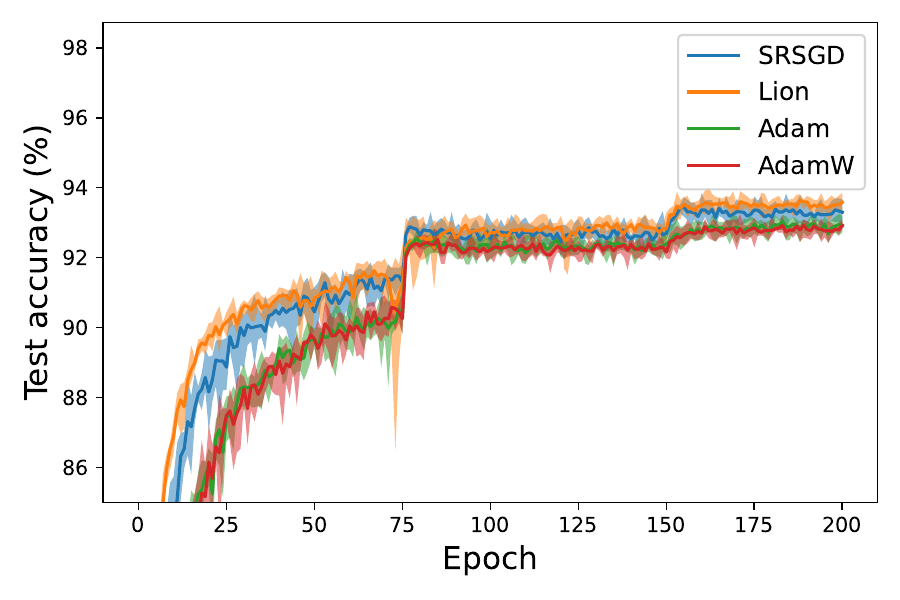}
			\label{Fig:Fig_Test2_Cifar10_Lion_acc}
		\end{minipage}%
	}%
	\subfigure[Test loss, CIFAR-10]{
		\begin{minipage}[t]{0.33\linewidth}
			\centering
			\includegraphics[width=\linewidth]{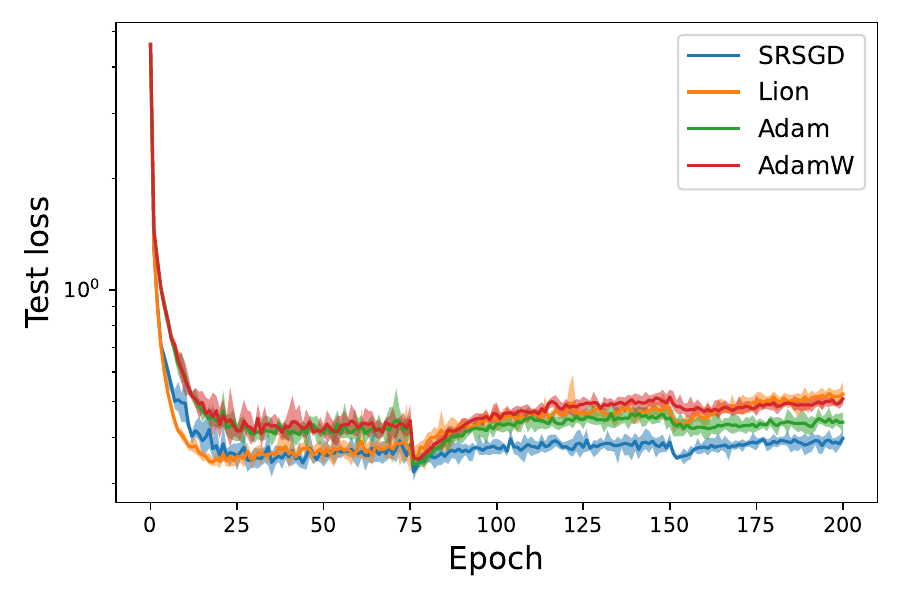}
			\label{Fig:Fig_Test2_Cifar10_Lion_testloss}
		\end{minipage}%
	}%
	\subfigure[Training loss, CIFAR-10]{
		\begin{minipage}[t]{0.33\linewidth}
			\centering
			\includegraphics[width=\linewidth]{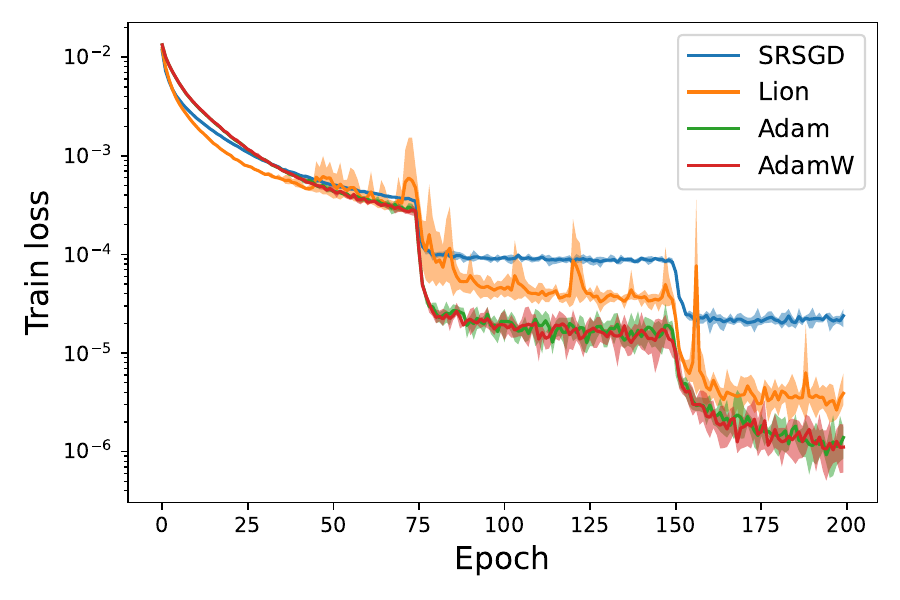}
			\label{Fig:Fig_Test2_Cifar10_Lion_trainloss}
		\end{minipage}%
	}%
	
	\subfigure[Test accuracy, CIFAR-100]{
		\begin{minipage}[t]{0.33\linewidth}
			\centering
			\includegraphics[width=\linewidth]{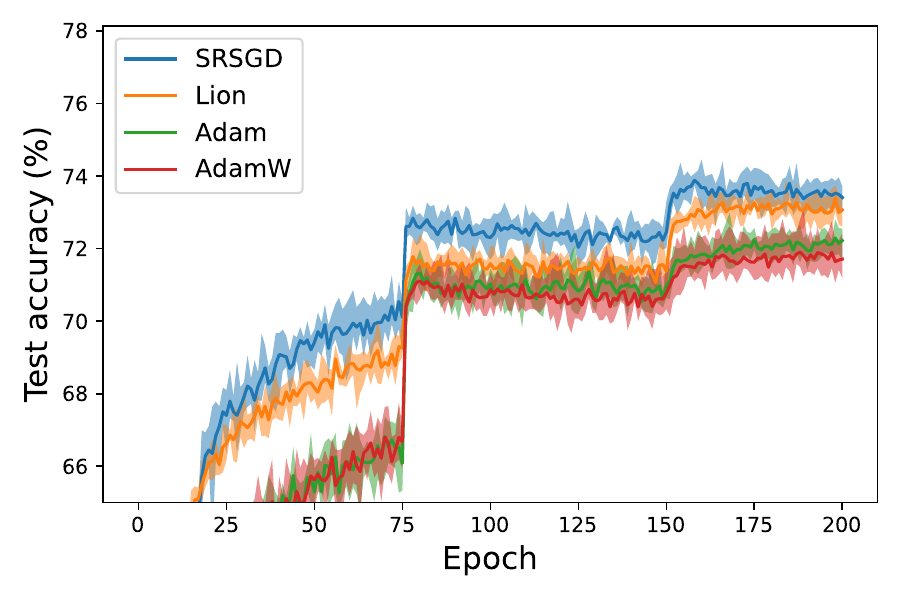}
			\label{Fig:Fig_Test2_Cifar100_Lion_acc}
		\end{minipage}%
	}%
	\subfigure[Test loss, CIFAR-100]{
		\begin{minipage}[t]{0.33\linewidth}
			\centering
			\includegraphics[width=\linewidth]{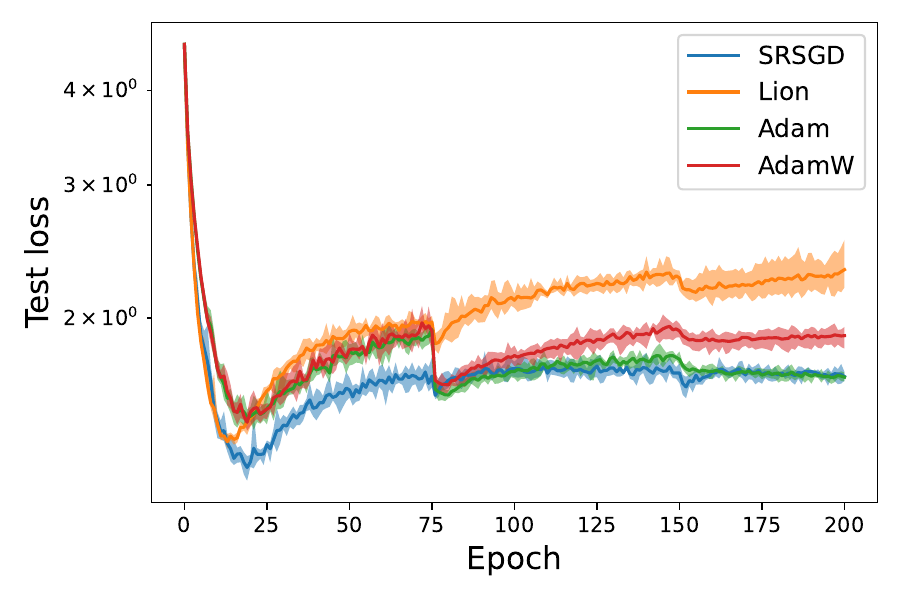}
			\label{Fig:Fig_Test2_Cifar100_Lion_testloss}
		\end{minipage}%
	}%
	\subfigure[Training loss, CIFAR-100]{
		\begin{minipage}[t]{0.33\linewidth}
			\centering
			\includegraphics[width=\linewidth]{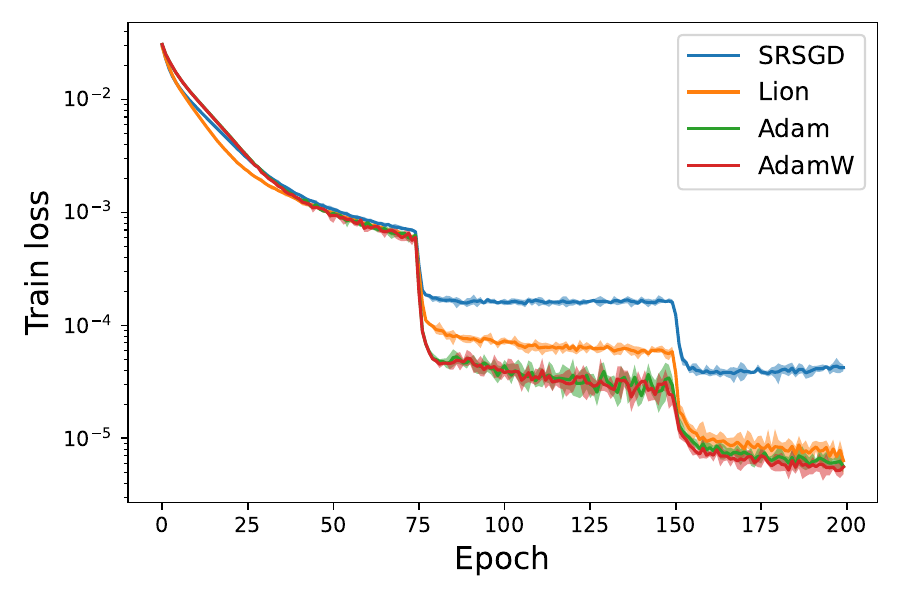}
			\label{Fig:Fig_Test2_Cifar100_Lion_trainloss}
		\end{minipage}%
	}%
	\caption{Test results on the performance of all the compared methods for training ResNet50 on CIFAR-10 and CIFAR-100 datasets. }
	\label{Fig_Test_Lion}
\end{figure}

In our numerical experiments, we test the performance of all the compared methods for training ResNet-50 \citep{he2016deep} on image classification tasks using the CIFAR-10 and CIFAR-100 datasets \citep{krizhevsky2009learning}. With the initial stepsize $\eta_0$,  we decay the stepsizes by a factor of $0.2$ at the $75$-th and $150$-th epoch for all the methods under comparison. Moreover, we use a grid search to find the best combination of $\eta_0$ and $\tau = \frac{1-\beta_1}{\eta_0}$ from $\eta_0 \in \{5 \times 10^{-4}, 1\times 10^{-4}, 5\times 10^{-5}, 1\times 10^{-5}\}$ and $\beta_1 \in \{0.8, 0.9, 0.95, 0.99, 0.999\}$, respectively. In addition, for Lion and Adam, the parameter $\beta_2$ is set to the recommended values of $0.99$ and $0.999$, respectively, to optimize the variance estimator.   The goal is to find the combination of $(\eta_0, \tau)$ that yields the most significant accuracy improvement after 20 epochs. The other parameters for the previously released subgradient methods remain fixed at their default values.

Figure \ref{Fig_Test_Lion} exhibits the efficacy of our proposed SRSGD method compared to Lion, Adam, and AdamW. Here, the metrics "test loss" and "training loss" measure the discrepancy between the network outputs and the actual targets on the test and training datasets, respectively. It is worth noting that the test losses are reported excluding the values of the regularization term, as the weight decay employed in Lion and AdamW does not correspond to quadratic regularization \cite{loshchilov2017decoupled,chen2023symbolic}. As a result, although SRSGD achieves the highest training losses in all the instances in Figure \ref{Fig_Test_Lion}, these training losses do not directly reflect the progress in training. Moreover, we can observe from Figure \ref{Fig_Test_Lion} that the test accuracy of SRSGD closely parallels that of Lion and surpasses those achieved by Adam and AdamW. Furthermore, the test loss curves for SRSGD are consistently lower than those of the competing methods, suggesting a potential enhancement in generalization capabilities when compared to Lion, Adam, and AdamW.  These numerical results highlight the high efficiency of the SRSGD method. Consequently, it can be concluded that our proposed framework \eqref{Eq_Framework} can yield efficient subgradient methods for training nonsmooth neural networks with guaranteed convergence properties.

\begin{acknowledgements}
{
We thank the editors and the reviewers for their valuable suggestions that have greatly helped to improve our paper.}
\end{acknowledgements}

%
\section*{Conflict of interest}

The authors declare that they have no conflict of interest.

\bibliographystyle{plain}
\bibliography{ref}

\end{document}